\documentclass[12pt]{article}

\pdfoutput=1

\usepackage[utf8]{inputenc}
\usepackage[T1]{fontenc}
\usepackage[english]{babel} 

\usepackage{graphicx,color,nicefrac,enumerate,listings}
\usepackage{amsmath,amsfonts,amsthm,amssymb}
\usepackage{mathtools} 
\usepackage[retainorgcmds]{IEEEtrantools}
\usepackage{dsfont}
\usepackage{bigints}

\usepackage{geometry}
\numberwithin{equation}{section}

\newcommand{\dWs}{\ud W_s}

\makeatletter
\newcommand{\pushright}[1]{\ifmeasuring@#1\else\omit\hfill$\displaystyle#1$\fi\ignorespaces}
\newcommand{\pushleft}[1]{\ifmeasuring@#1\else\omit$\displaystyle#1$\hfill\fi\ignorespaces}
\makeatother
\newcommand{\+}{\mkern1mu}
\newcommand{\tp}{\mkern2mu}

\let\temp\phi
\let\phi\varphi
\let\varphi\temp
\renewcommand{\epsilon}{\varepsilon}

\newcommand{\twovector}[2]{\begin{pmatrix} #1 \\ #2 \end{pmatrix}} 
 
\newcommand{\ud}{\,\mathrm{d}}	
\newcommand{\dx}{\ud x}
\newcommand{\ds}{\ud s}
\newcommand{\dt}{\ud t}

\newcommand{\e}{\mathrm{\+ e}}

\newcommand{\C}{\mathbb{C}}
\newcommand{\K}{\mathbb{K}}
\newcommand{\N}{\mathbb{N}}
\renewcommand{\P}{\mathbb{P}}

\newcommand{\R}{\mathbb{R}}

\newcommand{\bbB}{\mathbb{B}}

\newcommand{\bbH}{\mathbb{H}}

\newcommand{\bbM}{\mathbb{M}}

\newcommand{\bbU}{\mathbb{U}}

\newcommand{\cB}{\mathcal{B}}

\newcommand{\cE}{\mathcal{E}}
\newcommand{\cF}{\mathcal{F}}

\newcommand{\cJ}{\mathcal{J}}

\newcommand{\cP}{\mathcal{P}}

\newcommand{\cX}{\mathcal{X}}

\newcommand{\bfLambda}{\mathbf{\Lambda}}

\newcommand{\bfA}{\mathbf{A}}
\newcommand{\bfB}{\mathbf{B}}

\newcommand{\bfF}{\mathbf{F}}

\newcommand{\bfH}{\mathbf{H}}

\newcommand{\bfP}{\mathbf{P}}

\newcommand{\bfS}{\mathbf{S}}

\newcommand{\bfX}{\mathbf{X}}

\DeclarePairedDelimiter{\abs}{\lvert}{\rvert}

\DeclarePairedDelimiter{\norm}{\lVert}{\rVert}

\DeclareMathOperator{\id}{id}

\DeclareMathOperator{\tr}{trace}

\DeclareMathOperator{\E}{\mathbb{E}}

\newcommand{\DeclareMathSetSC}[2]{\newcommand{#1}{\text{\textsc{#2}}}} 
\newcommand{\DeclareMathSetBF}[2]{\newcommand{#1}{\text{\textbf{#2}}}} 
\newcommand{\DeclareMathSetTT}[2]{\newcommand{#1}{\text{\texttt{#2}}}} 
\DeclareMathSetSC{\OPER}{Oper}
\DeclareMathSetSC{\TEST}{Test}
\DeclareMathSetSC{\EMPTY}{Empty}
\DeclareMathSetSC{\INF}{Inf}
\DeclareMathSetSC{\EQUIV}{Equiv}
\DeclareMathSetSC{\KON}{Kon}
\DeclareMathSetSC{\IND}{Ind}
\DeclareMathSetSC{\SAT}{Sat}
\DeclareMathSetSC{\VC}{vc}
\DeclareMathSetSC{\CNF}{cnf}
\DeclareMathSetSC{\first}{First}
\DeclareMathSetSC{\follow}{Follow}
\DeclareMathSetBF{\error}{error}
\DeclareMathSetTT{\code}{code}

\theoremstyle{definition}
\newtheorem{Def}{Definition}[section]
\theoremstyle{plain}
\newtheorem{lemma}[Def]{Lemma}
\newtheorem{theorem}[Def]{Theorem}
\newtheorem{cor}[Def]{Corollary}

\theoremstyle{remark}

\newtheorem{remark}[Def]{Remark}

\usepackage[backend=bibtex,style=numeric,firstinits=true,sorting=nty,maxbibnames=99]{biblatex}
\DeclareNameAlias{default}{last-first}

\AtBeginBibliography{%
}

\DefineBibliographyStrings{english}{%
	andothers = {\addcomma\addspace\textsc{et\addabbrvspace al}\adddot},
	and = {\textsc{and}}
}

\DeclareFieldFormat
[article,inbook,incollection,inproceedings,patent,thesis,unpublished]
{title}{#1}

\renewbibmacro{in:}{%
	\ifentrytype{article}{%
	}{%
	\printtext{\bibstring{in}\intitlepunct}%
}%
}

\renewbibmacro*{volume+number+eid}{%
	\printfield{volume}%
	\setunit*{\addcomma\space}%
	\printfield{number}%
	\setunit{\addcomma\space}%
	\printfield{eid}}

\DeclareFieldFormat{pages}{#1}

\renewbibmacro*{publisher+location+date}{%
	\printlist{publisher}%
	\setunit*{\addcomma\space}%
	\printlist{location}%
	\setunit*{\addcomma\space}%
	\usebibmacro{date}%
	\newunit}

\begin{document}

\title{Weak convergence rates for spatial spectral \\
		Galerkin approximations of semilinear stochastic \\
		wave equations with multiplicative noise}
\author{Ladislas Jacobe de Naurois, Arnulf Jentzen, and Timo Welti\\[2mm]
\emph{ETH Z\"{u}rich, Switzerland}}
\maketitle

\begin{abstract}
Stochastic wave equations appear in several models for evolutionary processes subject to random forces, such as the motion of a strand of DNA in a liquid or heat flow around a ring. Semilinear stochastic wave equations can typically not be solved explicitly, but the literature contains a number of results which show that numerical approximation processes converge with suitable rates of convergence to solutions of such equations. In the case of approximation results for strong convergence rates, semilinear stochastic wave equations with both additive or multiplicative noise have been considered in the literature. In contrast, the existing approximation results for weak convergence rates assume that the diffusion coefficient of the considered semilinear stochastic wave equation is constant, that is, it is assumed that the considered wave equation is driven by additive noise, and no approximation results for multiplicative noise are known. The purpose of this work is to close this gap and to establish sharp weak convergence rates for semilinear stochastic wave equations with multiplicative noise. In particular, our weak convergence result establishes as a special case essentially sharp weak convergence rates for the hyperbolic Anderson model. Our method of proof makes use of the Kolmogorov equation, the H\"{o}lder-inequality for Schatten norms, and the mild It\^{o} formula.
\end{abstract}

\tableofcontents

\section{Introduction}

In the field of numerical approximations for stochastic evolution equations one distinguishes between two conceptually fundamentally  different error criteria, that is, strong convergence and weak convergence. In the case of finite dimensional stochastic ordinary differential equations, both strong and weak convergence are quite well understood nowadays; see, e.g., the standard monographs Kloeden \& Platen~\cite{Kloeden1992} and Milstein~\cite{Milstein1995}. In the case of infinite dimensional stochastic partial differential equations with regular nonlinearities strong convergence rates are essentially well understood, but weak convergence rates are still far away from being well comprehended (see, e.g., \cite{AnderssonKovacsLarsson2014, AnderssonKruseLarsson2013, AnderssonLarsson2012, BrehierKopec2013, Brehier2014, Brehier2012, ConusJentzenKurniawan2014, DeBouardDebussche2006, Debussche2011, DebusschePrintems2009, GeissertKovacsLarsson2009, Hausenblas2003, Hausenblas2010, JentzenKurniawan2015, Kopec2014, KovacsLindnerSchilling2014, KovacsLarssonLindgren2012, KovacsLarssonLindgren2013, KovacsPrintems2014, Kruse2014, LindnerSchilling2013, Shardlow2003, Wang2014, Wang2015, WangGan2013} for several weak convergence results in the literature). In this work we are interested in weak convergence rates for stochastic wave equations. Stochastic wave equations can be used for modelling several evolutionary processes subject to random forces. Examples include the motion of a DNA molecule floating in a fluid and the dilatation of shock waves throughout the sun (see, e.g., Section~1 in Dalang~\cite{Dalang1962}), as well as heat conduction around a ring (see, e.g., Thomas~\cite{Thomas2012}). Of course, these problems usually involve complicated nonlinearities and are inaccessible for current numerical analysis. Nonetheless, numerical examination of simpler model problems as the ones considered in the present work are a key first step. Even though a number of strong convergence rates for stochastic wave equations are available (see, e.g., \cite{AntonCohenLarssonWang2015, CohenLarssonSigg2013, CohenQuer-Sardanyons2015, KovacsLarssonLindgren2013, KovacsLarssonSaedpanah2010, Quer-SardanyonsSanz-Sole2006, Walsh2006, Wang2015, WangGanTang2014}), the existing weak convergence results for stochastic wave equations in the literature (see, e.g., \cite{Hausenblas2010, KovacsLindnerSchilling2014, KovacsLarssonLindgren2012, KovacsLarssonLindgren2013, Wang2015}) assume that the diffusion coefficient is constant, in other words, that the equation is driven by additive noise. The purpose of this work is to establish essentially sharp weak convergence rates for semilinear stochastic wave equations in the case of multiplicative noise.

To illustrate the main result of this article, we consider the following setting as a special case of our general framework (see Section~\ref{subsec:WeakSetting} below). Let $ ( H , \langle \cdot, \cdot \rangle_{ H } , \norm{\cdot}_{ H } ) $ and $ ( U , \langle \cdot , \cdot \rangle_{ U } , \norm{\cdot}_{ U } ) $ be separable $ \R $-Hilbert spaces, let $ T \in ( 0 , \infty ) $, let $ ( \Omega , \cF , \P ) $ be a probability space with a normal filtration $ ( \cF_t )_{ t \in [ 0 , T ] } $, let $ ( W_t )_{ t \in [ 0 , T ] } $ be an $ \id_U $-cylindrical $ ( \cF_t )_{ t \in [ 0 , T ] } $-Wiener process, let $ \{ e_n \}_{ n \in \N } \subseteq H $ be an orthonormal basis of $ H $, let $ \{ \lambda_n \}_{ n \in \N } \subseteq (0,\infty) $ be an increasing sequence, let $ A \colon D(A) \subseteq H \to H $ be the linear operator such that $ D(A) = \bigl\{ v \in H \colon \sum_{ n \in \N } \abs{ \lambda_n \langle e_n,v \rangle_{H} }^2 < \infty \bigr\} $ and such that for all $ v \in D(A) $ it holds that $ A v = \sum_{ n \in \N } - \lambda_n \langle e_n,v \rangle_{ H } e_n $, let $ ( H_r , \langle \cdot , \cdot \rangle_{ H_r } , \norm{\cdot}_{ H_r } )$, $ r \in \R $, be a family of interpolation spaces associated to $ - A $ (see, e.g., Definition~3.5.25 in \cite{Jentzen2015}), let $ ( \bfH_r , \langle \cdot , \cdot \rangle_{ \bfH_r } , \norm{\cdot}_{ \bfH_r } )$, $ r \in \R $, be the family of $ \R $-Hilbert spaces such that for all $ r \in \R $ it holds that $ ( \bfH_r , \langle \cdot , \cdot \rangle_{ \bfH_r } , \norm{\cdot}_{ \bfH_r } ) = \bigl( H_{ \nicefrac{r}{2} } \times H_{ \nicefrac{r}{2} - \nicefrac{1}{2} }, \langle \cdot, \cdot \rangle_{ H_{ \nicefrac{r}{2} } \times H_{ \nicefrac{r}{2} - \nicefrac{1}{2} } }, \norm{ \cdot }_{ H_{ \nicefrac{r}{2} } \times H_{ \nicefrac{r}{2} - \nicefrac{1}{2} } } \bigl) $, let $ P_N \colon \bigcup_{ r \in \R } H_r \to \bigcup_{ r \in \R } H_r $, $ N \in \N \cup \{ \infty \} $, be the mappings such that for all $ N \in \N \cup \{ \infty \} $, $ r \in \R $, $ v \in H_r $ it holds that $ P_N (v) =  \sum_{ n = 1 }^N \langle (\lambda_n)^{-r} e_n , v \rangle_{ H_r } (\lambda_n)^{-r} e_n $, let $ \bfP_N \colon \bigcup_{ r \in \R } \bfH_r \to \bigcup_{ r \in \R } \bfH_r $, $ N \in \N \cup \{ \infty \} $, be the mappings such that for all $ N \in \N \cup \{ \infty \} $, $ r \in \R $, $ ( v , w ) \in \bfH_r $ it holds that $ \bfP_N ( v , w ) = \bigl( P_N ( v ) , P_N ( w ) \bigr) $, let $ \bfA \colon D ( \bfA ) \subseteq \bfH_0 \to \bfH_0 $ be the linear operator such that $ D ( \bfA ) = \bfH_1 $ and such that for all $ (v,w) \in \bfH_1 $ it holds that $ \bfA( v , w ) = ( w , A v ) $, and let $ \gamma \in ( 0,\infty ) $, $ \beta \in ( \nicefrac{\gamma}{2}, \gamma ] $, $ \rho \in [0, 2( \gamma - \beta ) ] $, $ C_{ \bfF }, C_{ \bfB } \in [ 0, \infty ) $, $ \xi \in L^2 ( \P \vert_{\cF_0} ; \bfH_{ 2( \gamma - \beta ) } ) $, $ \bfF \in \mathrm{Lip}^0( \bfH_0, \bfH_0 ) $, $ \bfB \in \mathrm{Lip}^0( \bfH_0, L_2( U, \bfH_0 ) ) $ satisfy that $ ( -A )^{ - \beta } \in L_1( H_0 ) $, $ \bfF \vert_{ \bfH_\rho } \in \mathrm{Lip}^0( \bfH_\rho, \bfH_{ 2( \gamma - \beta ) } ) $, $ \bfB \vert_{ \bfH_\rho } \in \mathrm{Lip}^0( \bfH_\rho, L_2( U, \bfH_\rho ) \cap L( U, \bfH_\gamma ) ) $, \smash{$ \bfF \vert_{ \bigcap_{ r \in \R } \bfH_r } $} \smash{$ \in C_{\mathrm{b}}^2( \bigcap_{ r \in \R } \bfH_r, \bfH_0 ) $}, \smash{$ \bfB \vert_{ \bigcap_{ r \in \R } \bfH_r } \: \in \: C_{ \mathrm{b} }^2( \bigcap_{ r \in \R } \bfH_r, L_2( U, \bfH_0 ) ) $}, \smash{$ C_{ \bfF } \; = \; \sup_{ x, v_1, v_2 \in \cap_{ r \in \R } \bfH_r, \; \norm{ v_1 }_{ \bfH_0 } \vee \norm{ v_2 }_{ \bfH_0 } \leq 1 } $} \smash{$ \norm{ \bfF''(x)( v_1, v_2 ) }_{ \bfH_0 } < \infty $}, and \smash{$ C_{ \bfB } = \sup_{ x, v_1, v_2 \in \cap_{ r \in \R } \bfH_r, \; \norm{ v_1 }_{ \bfH_0 } \vee \norm{ v_2 }_{ \bfH_0 } \leq 1 }  \norm{ \bfB''(x)( v_1, v_2 ) }_{ L_2( U, \bfH_0) } < \infty $}.

\begin{theorem} \label{thm:introduction}
Assume the above setting. Then
\begin{itemize}
\item[(i)]
it holds that there exist up to modifications unique $( \cF_t )_{ t \in [0,T] } $-predictable stochastic processes $ \bfX^N = ( X^N, \cX^N ) \colon [ 0 , T ] \times \Omega \to \bfP_N ( \bfH_\rho ) $, $ N \in N \cup \{ \infty \} $, which satisfy for all $ N \in N \cup \{ \infty \} $, $ t \in [ 0 , T ] $ that $ \sup_{ s \in [ 0 , T ] } \norm{ \bfX_s^N }_{ L^2 ( \P ; \bfH_\rho ) } < \infty $ and $ \P $-a.s.\ that
\begin{equation} \label{eq:intro,2}
\bfX_t^N =  \e^{\bfA t} \bfP_N \xi + \int_0^t \e^{ \bfA(t-s) } \bfP_N \bfF( \bfX_s^N ) \ds  + \int_0^t \e^{ \bfA(t-s) } \bfP_N \bfB( \bfX_s^N ) \dWs
\end{equation}
\item[(ii)]
and it holds that
\begin{IEEEeqnarray}{l}
\begin{split}
& \sup_{ N \in \N } \sup_{ \phi \in C^{2}_\mathrm{b}( \bfH_0 , \R ) \setminus \{ 0 \} } \biggl( \frac{ ( \lambda_N )^{ \gamma - \beta } \; \abs[\big]{ \E \bigl[ \phi \bigl( \bfX_T^\infty \bigr)\bigr] - \E \bigl[\phi \bigl( \bfX_T^N \bigr) \bigr] } }{ \norm{ \phi }_{ C_{ \mathrm{b} }^2( \bfH_0, \R ) } } \biggr) \\
& \leq ( 1 \vee T ) \bigl( 1 \vee \norm{ \xi }_{ L^2( \P; \bfH_\rho ) }^2 \bigr) \\
& \quad \cdot \Bigl( \norm{ \xi }_{ L^1(\P ; \bfH_{ 2(\gamma - \beta ) } ) } + \norm[\big]{ \bfF \vert_{ \bfH_\rho } }_{ \mathrm{Lip}^0( \bfH_\rho , \bfH_{ 2(\gamma - \beta ) } ) } + 2 \norm{ (-A)^{ -\beta } }_{ L_1(H_0) } \norm[\big]{ \bfB \vert_{ \bfH_\rho } }_{ \mathrm{Lip}^0( \bfH_\rho , L( U , \bfH_{\gamma} ) ) }^2 \Bigr) \IEEEeqnarraynumspace \\
& \quad \cdot \Bigl( 1 \vee \bigl[ T \bigl( C_{ \bfF }^2 + 2 C_{ \bfB }^2 \bigr) \bigr]^{ \nicefrac{1}{2} } \Bigr) \exp \bigl( T \bigl[ \tfrac{1}{2} + 3 \abs{ \bfF }_{ \mathrm{Lip}^0( \bfH_0, \bfH_0 ) } + 4 \abs{ \bfB }_{ \mathrm{Lip}^0( \bfH_0, L_2( U, \bfH_0 ) ) }^2 \bigr] \bigr) \\
& \quad \cdot \exp \Bigl( T \Bigl[ 2 \norm[\big]{ \bfF \vert_{ \bfH_\rho } }_{ \mathrm{Lip}^0( \bfH_\rho, \bfH_\rho ) } + \norm[\big]{ \bfB \vert_{ \bfH_\rho } }_{ \mathrm{Lip}^0( \bfH_\rho, L_2( U, \bfH_\rho ) ) }^2 \Bigr] \Bigr)  < \infty.
\end{split}
\end{IEEEeqnarray}
\end{itemize}
\end{theorem}

Theorem~\ref{thm:introduction} is a consequence of the more general results in Remark~\ref{rmk:existence_regularity} and Theorem~\ref{thm:weakrates} below (see Corollary~\ref{cor:weak_rates_estimate}). Our proof of Theorem~\ref{thm:weakrates} uses, as usual in the case of weak convergence analysis, the Kolmogorov equation (see \eqref{eq:weakrates,19} below) as well as the H\"{o}lder inequality for Schatten norms (see \eqref{eq:weakrates,10} below). In addition, the proof of Theorem~\ref{thm:weakrates} employs the mild It\^{o} formula (see Corollary~1 in Da~Prato et al.~\cite{DaPratoJentzenRoeckner2010}) to obtain suitable a priori estimates for solutions of \eqref{eq:intro,2} (see Lemma~\ref{lem:finiteness} and \eqref{eq:weakrates,18} in Section~\ref{subsec:weak_rates} below for details). The detailed proof of Theorem~\ref{thm:introduction} and Theorem~\ref{thm:weakrates}, respectively, can also be found in Section~\ref{subsec:weak_rates}.

Next we illustrate Theorem~\ref{thm:introduction} by a simple example (cf. Corollary~\ref{cor:Anderson_model}). In the case where $ ( H, \langle \cdot, \cdot \rangle_H, \norm{\cdot}_H ) = ( U , \langle \cdot , \cdot \rangle_{ U } , \norm{\cdot}_{ U } ) = \bigl( L^2( \lambda_{ (0,1) }; \R ), \langle \cdot, \cdot \rangle_{ L^2( \lambda_{ (0,1) }; \R ) }, \norm{ \cdot }_{ L^2( \lambda_{ (0,1) }; \R ) } \bigr) $, $ \xi = ( \xi_0, \xi_1 ) \in H_0^1( (0,1); \R ) \times H $, $ \bfF = 0 $, where $ A \colon D(A) \subseteq H \to H $ is the Laplacian with Dirichlet boundary conditions on $ H $, and where $ \bfB \colon H \times H_{ -\nicefrac{1}{2} } \to L_2( H, H \times H_{ -\nicefrac{1}{2} } ) $ is the mapping which satisfies for all $ (v,w) \in H \times H_{ -\nicefrac{1}{2} } $, $ u \in C( [0,1], \R ) $ and $ \lambda_{ (0,1) } $-a.e.\ $ x \in ( 0, 1 ) $ that $ \bigl( \bfB( v, w ) u \bigr) (x) = \bigl( 0,  v(x) \cdot u(x) \bigr) $, the stochastic processes $ X^N \colon [ 0 , T ] \times \Omega \to P_N ( H ) $, $ N \in \N \cup \{ \infty \} $, are mild solutions of the SPDEs
\begin{equation} \label{eq:intro,1}
\ddot{X}_t (x) = \tfrac{ \partial^2 }{ \partial x^2 } X_t (x) + P_N X_t (x) \dot{W}_t(x)
\end{equation}
with $ X_t(0) = X_t(1) =0 $, $ X_0(x) = ( P_N \xi_0 ) (x) $, $ \dot{X}_0 (x) = ( P_N \xi_1 ) (x) $ for $ t \in [ 0, T ] $, $ x \in ( 0, 1 ) $, $ N \in \N \cup \{ \infty \} $. In the case $ N = \infty $, \eqref{eq:intro,1} is known as the hyperbolic Anderson model in the literature (see, e.g., Conus et al.~\cite{ConusJosephKhoshnevisanShiu2011}). Theorem~\ref{thm:introduction} applied to \eqref{eq:intro,1} ensures for all $ \phi \in C_{ \mathrm{b} }^2( H, \R ) $, $ \epsilon \in (0,\infty) $ that there exists a real number $ C \in [0, \infty) $ such that for all $ N \in \N $ it holds that 
\begin{equation}
\abs[\big]{ \E \bigl[ \phi \bigl(X_T^\infty \bigr)\bigr] - \E \bigl[\phi \bigl(X_T^N \bigr) \bigr] } \leq C \cdot N^{ \epsilon - 1 }
\end{equation}
(see Corollary~\ref{cor:Anderson_model}). We thus prove that the spectral Galerkin approximations converge with the weak rate 1- to the solution of the hyperbolic Anderson model. The weak rate 1- is exactly twice the well-known strong convergence rate of the hyperbolic Anderson model. To the best of our knowledge, Theorem~\ref{thm:introduction} is the first result in the literature that establishes an essentially sharp weak convergence rate for the hyperbolic Anderson model. Theorem~\ref{thm:introduction} also establishes essentially sharp weak convergence rates for more general semilinear stochastic wave equations (see Corollary~\ref{cor:wave_equation} and Corollary~\ref{cor:Anderson_model} below).

The remainder of this article is organized as follows. In Sections~\ref{subsec:notation} and \ref{subsec:setting} the general notation and framework is presented. Section~\ref{subsec:preliminaries} states mostly well-known existence, uniqueness, and regularity results, while Section~\ref{subsec:basic_properties} collects basic properties about the interpolation spaces and the semigroup associated to the deterministic wave equation. The main result of this article, Theorem~\ref{thm:weakrates} below, is stated and proven in Section~\ref{subsec:weak_rates}. Finally, Section~\ref{subsec:examples} shows how this abstract result can be applied to relevant problems, in particular, the hyperbolic Anderson model (see Corollary~\ref{cor:wave_equation} and Corollary~\ref{cor:Anderson_model} below).

\subsection{Notation} \label{subsec:notation}

Throughout this article the following notation is used.
For a set $ A $ we denote by $ \cP( A ) $  the power set of $ A $ and by $ \cP_0( A ) $ the set of all finite subsets of $ A $.
Furthermore, for two sets $ A $ and $ B $ we denote by $ A \triangle B $ be the set given by $ A \triangle B = ( A \setminus B ) \cup ( B \setminus A ) $ and by $ \bbM( A , B ) $ the set of all mappings from $ A $ to $ B $.
In addition, let $ (\cdot) \wedge (\cdot), (\cdot) \vee (\cdot) \colon \R^2 \to \R $ be the mappings with the property that for all $ x, y \in \R $ it holds that $ x \wedge y = \min\{ x, y \} $ and $ x \vee y = \max\{ x, y \} $.
Moreover, let $ \Gamma \colon ( 0 , \infty ) \to ( 0 , \infty ) $ be the Gamma function, that is, for all $ x \in ( 0 , \infty  ) $ it holds that $ \Gamma ( x ) = \int_0^{ \infty } t^{ ( x - 1 ) } \tp \e^{ - t }  \dt $, and let $ \mathcal{E}_r \colon [ 0 , \infty ) \to [ 0 , \infty )$, $ r \in ( 0 , \infty ) $, be the mappings such that for all $ r \in ( 0 , \infty ) $, $ x \in [ 0 , \infty ) $ it holds that \smash{$ \mathcal{E}_{r} [x] = \bigl[ \sum_{ n = 0 }^{ \infty } \frac{ x^{2n} \Gamma (r)^n }{ \Gamma ( nr + 1 ) } \bigr]^{ \nicefrac{1}{2} } $} (cf. Chapter~7 in Henry~\cite{Henry1981} and, e.g., Definition~1.3.1 in \cite{Jentzen2015}).
Furthermore, for a metric space $ ( E, d_E ) $, a dense subset $ A \subseteq E $, a complete metric space $ ( F, d_F ) $, a uniformly continuous mapping $ f \colon A \to F $, and the unique mapping $ \tilde{f} \in C( E, F ) $ with the property that $ \tilde{f} \vert_{ A } = f $ (see, e.g., Proposition~2.5.19 in \cite{Jentzen2015}), we often write, for simplicity of presentation, $ f $ instead of $ \tilde{f} $ in the following.
In addition, for two $ \R $-Banach spaces $ ( V, \norm{\cdot}_{ V } ) $ and $ ( W, \norm{\cdot}_{ W } ) $ with $ V \neq \{0\} $, an open subset $ U \subseteq V $, and a natural number $ k \in \N = \{ 1,2,3,\ldots \} $, let $ \abs{\cdot}_{ C_\mathrm{b}^{k} ( U , W ) }, \norm{\cdot}_{ C_\mathrm{b}^{k} ( U , W ) } \colon C^{k} ( U , W ) \to [ 0 , \infty ] $ be the mappings with the property that for all $ f \in C^k( U, W ) $ it holds that
\begin{align}
\abs{f}_{ C_\mathrm{b}^{ k } ( U, W ) } & = \sup_{ x \in U } \norm{ f^{(k)}(x) }_{ L^{(k)} ( V , W ) } = \sup_{ x \in U } \sup_{ v_1,\ldots,v_k \in V \setminus \{0\} } \frac{ \norm{ f^{(k)} (x) (v_1,\ldots, v_k) }_W }{ \norm{v_1}_V \cdot \ldots \cdot \norm{v_k}_V }, \\
\norm{f}_{ C_\mathrm{b}^{ k } ( U, W ) } & = \norm{ f(0) }_{W} + \sum_{ \ell = 1 }^{ k } \abs{f}_{ C_\mathrm{b}^{ \ell } ( U , W ) },
\end{align}
and we denote by $ { C_\mathrm{b}^{ k } ( U, W ) } $ the set given by $ C_\mathrm{b}^{ k } ( U, W ) = \bigl\{ f \in C^{ k } ( U, W ) \colon \norm{f}_{ C_\mathrm{b}^{ k } ( U, W ) } < \infty \bigr\}$.
Moreover, for two $ \R $-Banach spaces $ ( V, \norm{\cdot}_{ V } ) $ and $ ( W, \norm{\cdot}_{ W } ) $ with $ V \neq \{0\} $, an open subset $ U \subseteq V $, and a number $ k \in \N_0 = \{ 0,1,2,\ldots \} $, let $ \abs{\cdot}_{ \mathrm{Lip}^k ( U , W ) }, \norm{\cdot}_{ \mathrm{Lip}^{k} ( U , W ) } \colon C^{ k } ( U , W ) \to [ 0 , \infty ] $ be the mappings with the property that for all $ f \in C^k( U , W ) $ it holds that
\begin{align}
\abs{f}_{ \mathrm{Lip}^k ( U, W ) } & = 
\begin{cases}
\sup_{ \stackrel{ x,y \in U,  }{ x \neq y } } \Bigl( \frac{ \norm{ f (x) - f (y) }_W }{ \norm{x - y}_{ V } } \Bigr) & \colon k = 0, \\
\sup_{ \stackrel{ x,y \in U,  }{ x \neq y } } \Bigl( \frac{ \norm{ f^{ (k) }(x) - f^{ (k) }(y) }_{ L^{ (k) } ( V , W ) } }{ \norm{x - y}_{ V } } \Bigr) & \colon k \in \N,
\end{cases} \\
\norm{f}_{ \mathrm{Lip}^k ( U, W ) } &= \norm{ f(0) }_{W} + \sum_{ \ell = 0 }^{ k } \abs{ f }_{ \mathrm{Lip}^{\ell} ( U, W ) },
\end{align}
and we denote by $ \mathrm{Lip}^k ( U, W ) $ the set given by $ \mathrm{Lip}^k ( U, W ) = \bigl\{ f \in C^{ k } ( U, W ) \colon \norm{f}_{ \mathrm{Lip}^{k} ( U, W ) } < \infty \bigr\} $.
Additionally, for two normed $ \R $-vector spaces $ ( V, \norm{\cdot}_{ V } ) $ and $ ( W, \norm{\cdot}_{ W } ) $ let $ \norm{\cdot}_{ \mathrm{LG}( V , W ) } \colon \bbM ( V , W ) \to [ 0 , \infty ] $ be the mapping such that for all $ f \in \bbM( V, W ) $ it holds that \smash{$ \norm{f}_{ \mathrm{LG}( V, W ) } = \sup_{ v \in V } \bigl( \frac{ \norm{f(v)}_W }{ \max \{ 1 , \norm{v}_V \} } \bigr) $}.
For an $ \R $-Hilbert space $ ( H , \langle \cdot , \cdot \rangle_{H} , \norm{\cdot}_{H} )$ let $ \cJ^{H} \colon L^{(2)} ( H , \R ) \to L( H ) $ be the mapping with the property that for all $ \beta \in L^{ ( 2 ) } ( H , \R ) $, $ h_1, h_2 \in H $ it holds that $ \beta ( h_1 , h_2 ) =  \langle h_1 , \cJ_{ \beta }^{H} h_2 \rangle_{H} $.
Furthermore, for $ \R $-Hilbert spaces $ ( H_i , \langle \cdot , \cdot \rangle_{ H_i } , \norm{\cdot}_{ H_i } ) $, $ i \in \{1,2\} $, let $ \norm{\cdot}_{ L_{p} ( H_1, H_2 ) } \colon L( H_1 , H_2 ) \to [ 0 , \infty ] $, $ p \in [ 1 , \infty ) $, be the mappings with the property that for all $ p \in [ 1 , \infty ) $, $ A \in L( H_1, H_2 ) $ it holds that \smash{$ \norm{ A }_{ L_{p} ( H_1 , H_2 ) } = \bigl( \tr_{ H_1 } ( ( A^{ \star } A )^{ \nicefrac{p}{2} } ) \bigr)^{ \nicefrac{1}{p} } $}, we denote by $ L_{p} ( H_1, H_2 ) $ the set given by $ L_{p} ( H_1, H_2 ) = \bigl\{ A \in L( H_1, H_2 ) \colon \norm{A}_{L_p( H_1, H_2 ) } < \infty \bigr\} $, and we call $ L_{p} ( H_1, H_2 ) $ the Schatten $p$-class of bounded linear operators from $ H_1 $ to $ H_2 $. 
For brevity, for an $ \R $-Hilbert space $ ( H , \langle \cdot , \cdot \rangle_{ H } , \norm{\cdot}_{ H } ) $ and a number $ p \in [ 1 , \infty ) $, we denote $ L_p ( H , H ) $ by $ L_p( H ) $ and we call $ L_p( H ) $ the Schatten $p$-class of bounded linear operators on $ H $.
In addition, for an $ \R $-Hilbert space $ (H, \langle \cdot,\cdot \rangle_H, \norm{\cdot}_H ) $, an orthonormal basis $ \bbB \subseteq H $ of $ H $, a mapping $ \lambda \colon \bbB \to \R $, a linear operator $ A \colon D(A) \subseteq H \to H $ satisfying that $ D(A) = \bigl\{ v \in H \colon \sum_{ b \in \bbB } \abs{\lambda_b \langle b, v \rangle_{H} }^2 < \infty \bigr\} $ and that for all $ v \in D(A) $ it holds that $ Av = \sum_{ b \in \bbB } \lambda_b \langle b , v \rangle_H b $, and a mapping $ \phi \colon \R \to \R $, let $ \phi( A ) \colon D( \phi(A) ) \subseteq H \to H $ be the linear operator satisfying that $ D( \phi(A) ) = \bigl\{ v \in H \colon \sum_{ b \in \bbU } \abs{ \phi(\lambda_b) \langle b , v \rangle_{H} }^2 < \infty \bigr\} $ and that for all $ v \in D( \phi(A) ) $ it holds that $ \phi( A ) v = \sum_{ b \in \bbB } \phi ( \lambda_b ) \langle b , v \rangle_{H} b $.
For two $ \R $-inner product spaces $ (V, \langle \cdot, \cdot \rangle_V, \norm{\cdot}_V ) $ and $ (W, \langle \cdot, \cdot \rangle_W, \norm{\cdot}_W) $ we denote by $ ( V \times W, \langle \cdot, \cdot \rangle_{ V \times W}, \norm{ \cdot }_{ V \times W } ) $ the $ \R $-inner product space such that for all $ x_1 = (v_1, w_1), x_2 = (v_2, w_2) \in V \times W $ it holds that $ \langle x_1, x_2 \rangle_{ V \times W } = \langle v_1, v_2 \rangle_V + \langle w_1, w_2 \rangle_W $.
Finally, for a Borel measurable set $ A \in \cB( \R ) $ we denote by $ \lambda_A \colon \cB( A ) \to [ 0, \infty ] $ the Lebesgue-Borel measure on $ A $.

\subsection{Setting}\label{subsec:setting}

Let $ ( U , \langle \cdot , \cdot \rangle_{ U } , \norm{\cdot}_{ U } ) $ be a separable $ \R $-Hilbert space, let $ \bbU \subseteq U $ be an orthonormal basis of $ U $, let $ T \in ( 0 , \infty ) $, let $ ( \Omega , \cF , \P ) $ be a probability space with a normal filtration $ ( \cF_t )_{ t \in [ 0 , T ] } $, and let $ ( W_t )_{ t \in [ 0 , T ] } $ be an $ \id_U $-cylindrical $ ( \cF_t )_{ t \in [ 0 , T ] } $-Wiener process.

\section{Preliminaries}

\subsection{Existence, uniqueness, and regularity results for stochastic evolution equations} \label{subsec:preliminaries}

Theorem~\ref{thm:existence} below is a direct consequence of Theorem~7.4 in Da~Prato \& Zabczyk~\cite{DaPratoZabczyk1992}.

\begin{theorem} \label{thm:existence}
Assume the setting in Section~\ref{subsec:setting}, let $ ( H , \langle \cdot, \cdot \rangle_{ H } , \norm{\cdot}_{ H } ) $ be a separable $ \R $-Hilbert space, let $ S \colon [ 0, \infty ) \to L(H) $ be a strongly continuous semigroup, and let $ p \in [2, \infty) $, $ F \in \mathrm{Lip}^0( H, H ) $, $ B \in \mathrm{Lip}^0( H, L_2( U, H ) ) $, $ \xi \in L^p( \P\vert_{\cF_0}; H ) $. Then there exists an up to modifications unique $ ( \cF_t )_{t \in [0,T]} $-predictable stochastic process $ X \colon [0,T] \times \Omega \to H $ such that for all $ t \in [0,T] $ it holds that $ \sup_{ s \in [0,T] } \norm{ X_s }_{ L^p( \P; H ) } < \infty $ and $ \P $-a.s.\ that
\begin{equation}
X_t = S_t \xi + \int_0^t S_{t-s} F( X_s ) \ds + \int_0^t S_{t-s} B( X_s ) \dWs.
\end{equation}
\end{theorem}

\begin{remark} \label{rem:existence}
Assume the setting in Section~\ref{subsec:setting}, let $ ( H , \langle \cdot, \cdot \rangle_{ H } , \norm{\cdot}_{ H } ) $ be a separable $ \R $-Hilbert space, let $ S \colon [ 0, \infty ) \to L(H) $ be a strongly continuous semigroup, and let $ F \in \mathrm{Lip}^0( H, H ) $, $ B \in \mathrm{Lip}^0( H, L_2( U, H ) ) $. Then Theorem~\ref{thm:existence} shows that there exist up to modifications unique $ ( \cF_t )_{t \in [0,T]} $-predictable stochastic processes $ X^x \colon [0,T] \times \Omega \to H $, $ x \in H $, such that for all $ x \in H $, $ t \in [0,T] $, $ p \in [2, \infty) $ it holds that $ \sup_{ s \in [0,T] } \norm{ X_s^x }_{ L^p( \P; H ) } < \infty $ and $ \P $-a.s.\ that
\begin{equation}
X_t^x = S_t x + \int_0^t S_{t-s} F( X_s^x ) \ds + \int_0^t S_{t-s} B( X_s^x ) \dWs.
\end{equation}
\end{remark}

\begin{lemma} \label{lem:kolmogorov}
Assume the setting in Section~\ref{subsec:setting}, let $ ( H , \langle \cdot, \cdot \rangle_{ H } , \norm{\cdot}_{ H } ) $ be a finite-dimensional $ \R $-vector space, let $ A \in L(H) $, $ F \in C_\mathrm{b}^2( H, H ) $, $ B \in C_\mathrm{b}^2( H, L_2( U, H ) ) $, $ \phi \in C_\mathrm{b}^2( H, \R ) $, let $ X^x \colon [0,T] \times \Omega \to H $, $ x \in H $, be $ ( \cF_t )_{t \in [0,T]} $-predictable stochastic processes satisfying that for all $ x \in H $, $ t \in [0,T] $ it holds that $ \sup_{ s \in [0,T] } \norm{ X_s^x }_{ L^2( \P; H ) } < \infty $ and $ \P $-a.s.\ that
\begin{equation}
X_t^x = \e^{ At } x + \int_0^t \e^{ A(t-s) }F( X_s^x ) \ds + \int_0^t \e^{ A(t-s) } B( X_s^x ) \dWs,
\end{equation}
and let $ u \colon [0,T] \times H \to \R $ be the mapping with the property that for all $ t \in [0,T] $, $ x \in H $ it holds that $ u(t,x) = \E[ \phi( X_t^x ) ] $. Then
\begin{itemize}
\item[(i)]
it holds that $ u \in C^{ 1 , 2 }( [0,T] \times H, \R ) $,
\item[(ii)]
it holds for all $ (t,x) \in [0,T] \times H $ that
\begin{equation}
\bigl( \tfrac{\partial}{\partial t} u \bigr) (t,x) = \bigl( \tfrac{\partial}{\partial x} u \bigr) (t,x) [ A x + F(x) ] + \tfrac{1}{2} \sum_{ u \in \bbU } \bigl( \tfrac{\partial^2}{\partial x^2} u \bigr) (t,x)( B(x)u, B(x)u ),
\end{equation}
\item[(iii)]
and it holds that
\begin{align}
& \begin{aligned}
\sup_{t \in [0,T]} \abs{ u( t, \cdot ) }_{ C_\mathrm{b}^1( H, \R ) } & \leq \abs{ \phi }_{ C_{ \mathrm{b} }^1( H, \R ) } \biggl[ \sup_{ s \in [0,T]} \norm[\big]{ \e^{ As } }_{ L(H) } \biggr] \\
& \quad \cdot \exp \biggl( T \bigl[ \abs{ F }_{ C_{ \mathrm{b} }^1( H, H ) } + \tfrac{1}{2} \abs{ B }_{ C_{ \mathrm{b} }^1( H, L_2( U,H ) ) }^2 \bigr] \sup_{ s \in [0,T]} \norm[\big]{ \e^{ As } }_{ L(H) }^2 \biggr) < \infty,
\end{aligned} \\
& \begin{aligned}
& \sup_{t \in [0,T]} \abs{ u( t, \cdot ) }_{ C_\mathrm{b}^2( H, \R ) } \\
& \leq \norm{ \phi }_{ C_{ \mathrm{b} }^2( H, \R ) } \Bigl( 1 \vee \Bigl[ T \bigl( \abs{ F }_{ C_{ \mathrm{b} }^2( H, H ) }^2 + 2 \abs{ B }_{ C_{ \mathrm{b} }^2( H, L_2( U,H ) ) }^2 \bigr) \Bigr]^{ \nicefrac{1}{2} } \Bigr) \biggl[ \sup_{ s \in [0,T]} \norm[\big]{ \e^{ As } }_{ L(H) }^3 \biggr] \\
& \quad \cdot \exp \biggl( T \bigl[ \tfrac{1}{2} + 3 \abs{ F }_{ C_{ \mathrm{b} }^1( H, H ) } + 4 \abs{ B }_{ C_{ \mathrm{b} }^1( H, L_2( U,H ) ) }^2 \bigr] \sup_{ s \in [0,T]} \norm[\big]{ \e^{ As } }_{ L(H) }^4 \biggr) < \infty.
\end{aligned}
\end{align}
\end{itemize}
\end{lemma}

\begin{proof}[Proof of Lemma~\ref{lem:kolmogorov}]
It is well-known that the assumptions that $ \phi \in C_\mathrm{b}^2( H, \R ) $, $ F \in C_\mathrm{b}^2( H, H ) $, $ B \in C_\mathrm{b}^2( H, L_2( U, H ) ) $ imply that (i) and (ii) hold, that there exist up to modifications unique $ ( \cF_t )_{ t \in [0,T] } $-predictable stochastic processes $ X^{ x, v_1}, X^{ x, v_1, v_2} \colon [0,T] \times \Omega \to H $, $ x, v_1, v_2 \in H $, satisfying for all $ x, v_1, v_2 \in H $, $ t \in [0,T] $, $ p \in [2, \infty) $ that $ \sup_{ s \in [0,T] } \bigl( \norm{ X_s^{ x, v_1 } }_{ L^p( \P; H ) } +  \norm{ X_s^{ x, v_1, v_2 } }_{ L^p( \P; H ) } \bigr) < \infty $ and $ \P $-a.s.\ that
\begin{align}
X_t^{ x, v_1 } &= \e^{ At }v_1 + \int_{0}^{t} \e^{ A(t-s) } F'( X_s^x ) X_s^{ x, v_1 } \ds + \int_{0}^{t} \e^{ A(t-s) } B'( X_s^x ) X_s^{ x, v_1 } \dWs, \\
\begin{split}
X_t^{ x, v_1, v_2 } &= \int_{0}^{t} \e^{ A(t-s) } \bigl( F''(X_s^x) ( X_s^{ x, v_1 }, X_s^{ x, v_2} ) + F'(X_s^x) X_s^{ x, v_1, v_2 } \bigr) \ds \\
& \quad + \int_{0}^{t} \e^{ A(t-s) } \bigl( B''(X_s^x) ( X_s^{ x, v_1 }, X_s^{ x, v_2} ) + B'(X_s^x) X_s^{ x, v_1, v_2 } \bigr) \dWs,
\end{split}
\end{align}
and that for all $ (t,x) \in [0,T] \times H $, $ v_1, v_2 \in H $ it holds that
\begin{align}
\bigl( \tfrac{\partial}{\partial x} u \bigr) (t,x) v_1 & = \E \bigl[ \phi'( X_t^x ) X_t^{ x, v_1 } \bigr], \label{eq:kolmogorov,5} \\
\bigl( \tfrac{\partial^2}{\partial x^2 } u \bigr) (t,x) ( v_1, v_2 ) & = \E \bigl[  \phi''( X_t^x ) ( X_t^{ x, v_1 }, X_t^{ x, v_2 } ) + \phi'( X_t^x ) X_t^{ x, v_1, v_2 } \bigr]. \label{eq:kolmogorov,6}
\end{align}
It thus remains to prove (iii). For this let $ \psi_p \colon H \to \R $, $ p \in [2, \infty) $, be the functions satisfying for all $ p \in [2, \infty) $, $ x \in H $ that $ \psi_p(x) = \norm{ x }_H^p $. Then note for all $ p \in [2, \infty) $, $ x, v_1, v_2 \in H $ that $ \psi_p \in C^2( H, \R ) $ and that
\begin{align}
\psi_p' (x) v_1 &= \begin{cases}
0 & \colon x = 0, \\
p \norm{ x }^{ p-2 } \langle x, v_1 \rangle_H & \colon x \neq 0,
\end{cases} \\
\begin{split}
\psi_p'' (x) (v_1, v_2) &= \begin{cases}
2 \langle v_1, v_2 \rangle_H & \colon p=2, \\
0 & \colon ( p \neq 2 ) \wedge (x = 0), \\
p \norm{ x }_H^{ p-2 } \langle v_1, v_2 \rangle_H + p(p-2) \norm{ x }_H^{ p-4 } \langle x, v_1 \rangle_H \langle x, v_2 \rangle_H & \colon x \neq 0.
\end{cases}
\end{split}
\end{align}
An application of the mild It\^{o} formula in Corollary~1 in Da~Prato et al.~\cite{DaPratoJentzenRoeckner2010} on the test functions $ \psi_p $, $ p \in [2, \infty) $, and the Cauchy-Schwarz inequality hence yield for all $ p \in [2, \infty) $, $ x, v \in H $, $ t \in [0,T] $ that
\begin{equation} \label{eq:kolmogorov,3}
\begin{split}
& \E \bigl[ \norm{ X_t^{ x, v } }_H^p \bigr] = \E \bigl[ \psi_p( X_t^{ x, v } ) \bigr] \\
& = \psi_p \bigl( \e^{ At } v \bigr) + \int_{0}^{t} \E \bigl[ \psi_p' \bigl( \e^{ A(t-s) } X_s^{ x, v } ) \e^{ A(t-s) } F'( X_s^x ) X_s^{ x, v } \bigr] \ds \\
& \quad + \frac{1}{2} \sum_{ u \in \bbU } \int_{0}^{t} \E \bigl[ \psi_p'' \bigl( \e^{ A(t-s) } X_s^{ x, v } \bigr) \bigl( \e^{ A(t-s) } \bigl( B'( X_s^x ) X_s^{ x, v } \bigr) u, \e^{ A(t-s) } \bigl( B'( X_s^x ) X_s^{ x, v } \bigr) u \bigr) \bigr] \ds \\
& \leq \norm{ v }_H^p \biggl[ \sup_{ s \in [0,T]} \norm[\big]{ \e^{ As } }_{ L(H) }^p \biggr] + p \biggl[ \sup_{ s \in [0,T]} \norm[\big]{ \e^{ As } }_{ L(H) }^p \biggr] \abs{ F }_{ C_{ \mathrm{b} }^1( H, H ) } \int_{0}^{t} \E \bigl[ \norm{ X_s^{ x, v } }_H^p \bigr] \ds \\
& \quad + \tfrac{p}{2} \biggl[ \sup_{ s \in [0,T]} \norm[\big]{ \e^{ As } }_{ L(H) }^p \biggr] \abs{ B }_{ C_{ \mathrm{b} }^1( H, L_2( U,H ) ) }^2 \int_{0}^{t} \E \bigl[ \norm{ X_s^{ x, v } }_H^p \bigr] \ds \\
& \quad + \tfrac{p(p-2)}{2}\biggl[ \sup_{ s \in [0,T]} \norm[\big]{ \e^{ As } }_{ L(H) }^p \biggr] \abs{ B }_{ C_{ \mathrm{b} }^1( H, L_2( U,H ) ) }^2 \int_{0}^{t} \E \bigl[ \norm{ X_s^{ x, v } }_H^p \bigr] \ds \\
& = \norm{ v }_H^p \biggl[ \sup_{ s \in [0,T]} \norm[\big]{ \e^{ As } }_{ L(H) }^p \biggr] \\
& \quad + p \biggl[ \sup_{ s \in [0,T]} \norm[\big]{ \e^{ As } }_{ L(H) }^p \biggr] \bigl( \abs{ F }_{ C_{ \mathrm{b} }^1( H, H ) } + \tfrac{p-1}{2} \abs{ B }_{ C_{ \mathrm{b} }^1( H, L_2( U,H ) ) }^2 \bigr) \int_{0}^{t} \E \bigl[ \norm{ X_s^{ x, v } }_H^p \bigr] \ds.
\end{split}
\end{equation}
Therefore, Gronwall's lemma shows for all $ p \in [2, \infty) $, $ x, v \in H $ that
\begin{equation} \label{eq:kolmogorov,1}
\begin{split}
& \sup_{ t \in [0,T] } \norm{ X_t^{ x, v } }_{ L^p( \P; H ) } \\
& \leq \norm{ v }_H  \biggl[ \sup_{ s \in [0,T]} \norm[\big]{ \e^{ As } }_{ L(H) } \biggr] \exp \biggl( T \bigl[ \abs{ F }_{ C_{ \mathrm{b} }^1( H, H ) } + \tfrac{p-1}{2} \abs{ B }_{ C_{ \mathrm{b} }^1( H, L_2( U,H ) ) }^2 \bigr] \sup_{ s \in [0,T]} \norm[\big]{ \e^{ As } }_{ L(H) }^p \biggr).
\end{split}
\end{equation}
Furthermore, applying again Corollary~1 in Da~Prato et al.~\cite{DaPratoJentzenRoeckner2010} on the test function $ \psi_2 $, the Cauchy-Schwarz inequality, and the fact that $ \forall \tp a, b \in \R $ it holds that $ a b \leq \tfrac{a^2 + b^2}{2} $ imply for all $ x, v_1, v_2 \in H $, $ t \in [0,T] $ that
\begin{equation} \label{eq:kolmogorov,4}
\begin{split}
& \E \bigl[ \norm{ X_t^{ x, v_1, v_2 } }_H^2 \bigr] \\
& = 2 \int_{0}^{t} \E \bigl[ \bigl\langle \e^{ A(t-s) } X_s^{ x, v_1, v_2 } ,  \e^{ A(t-s) } \bigl( F''(X_s^x) ( X_s^{ x, v_1 }, X_s^{ x, v_2} ) + F'(X_s^x) X_s^{ x, v_1, v_2 } \bigr) \bigr\rangle_H \bigr] \ds \\
& \quad + \int_{0}^{t} \E \bigl[ \norm[\big]{ \e^{ A(t-s) } \bigl( B''(X_s^x) ( X_s^{ x, v_1 }, X_s^{ x, v_2} ) + B'(X_s^x) X_s^{ x, v_1, v_2 } \bigr) }_{ L_2( U, H) }^2 \bigr] \ds \\
& \leq \int_{0}^{t} \norm[\big]{ \e^{ A(t-s) } F''(X_s^x) ( X_s^{ x, v_1 }, X_s^{ x, v_2} ) }_{ L^2( \P; H ) }^2 + \norm[\big]{ \e^{ A(t-s) } X_s^{ x, v_1, v_2 } }_{ L^2( \P; H ) }^2 \ds \\
& \quad + 2 \biggl[ \sup_{ s \in [0,T]} \norm[\big]{ \e^{ As } }_{ L(H) }^2 \biggr] \abs{ F }_{ C_{ \mathrm{b} }^1( H, H ) } \int_{0}^{t} \E \bigl[ \norm{ X_s^{ x, v_1, v_2 } }_H^2 \bigr] \ds \\
& \quad + 2 \int_{0}^{t} \E \bigl[ \norm[\big]{ \e^{ A(t-s) } B''(X_s^x) ( X_s^{ x, v_1 }, X_s^{ x, v_2} ) }_{ L_2( U, H) }^2 +  \norm[\big]{ \e^{ A(t-s) } B'(X_s^x) X_s^{ x, v_1, v_2 } }_{ L_2( U, H) }^2 \bigr] \ds \\
& \leq \biggl[ \sup_{s \in [0,T ] } \norm[\big]{ X_s^{ x, v_1 } }_{ L^4( \P; H ) }^2 \norm[\big]{ X_s^{ x, v_2 } }_{ L^4( \P; H ) }^2 \biggr]T \bigl( \abs{ F }_{ C_{ \mathrm{b} }^2( H, H ) }^2 + 2 \abs{ B }_{ C_{ \mathrm{b} }^2( H, L_2( U,H ) ) }^2 \bigr) \biggl[ \sup_{ s \in [0,T]} \norm[\big]{ \e^{ As } }_{ L(H) }^2 \biggr] \\
& \quad + 2 \biggl[ \sup_{ s \in [0,T]} \norm[\big]{ \e^{ As } }_{ L(H) }^2 \biggr] \bigl( \tfrac{1}{2} + \abs{ F }_{ C_{ \mathrm{b} }^1( H, H ) } + \abs{ B }_{ C_{ \mathrm{b} }^1( H, L_2( U,H ) ) }^2 \bigr) \int_{0}^{t} \E \bigl[ \norm{ X_s^{ x, v_1, v_2 } }_H^2 \bigr] \ds.
\end{split}
\end{equation}
Gronwall's lemma and \eqref{eq:kolmogorov,1}, therefore, imply for all $ x, v_1, v_2 \in H $ that
\begin{equation} \label{eq:kolmogorov,2}
\begin{split}
& \sup_{ t \in [0,T] } \norm{ X_t^{ x, v_1, v_2 } }_{ L^2( \P; H ) } \\
& \leq \biggl[ \sup_{s \in [0,T ] } \norm[\big]{ X_s^{ x, v_1 } }_{ L^4( \P; H ) } \norm[\big]{ X_s^{ x, v_2 } }_{ L^4( \P; H ) } \biggr] \Bigl[ T \bigl( \abs{ F }_{ C_{ \mathrm{b} }^2( H, H ) }^2 + 2 \abs{ B }_{ C_{ \mathrm{b} }^2( H, L_2( U,H ) ) }^2 \bigr) \Bigr]^{ \nicefrac{1}{2} } \\
& \quad \cdot \biggl[ \sup_{ s \in [0,T]} \norm[\big]{ \e^{ As } }_{ L(H) } \biggr] \exp \biggl( T \bigl[ \tfrac{1}{2} + \abs{ F }_{ C_{ \mathrm{b} }^1( H, H ) } + \abs{ B }_{ C_{ \mathrm{b} }^1( H, L_2( U,H ) ) }^2 \bigr] \sup_{ s \in [0,T]} \norm[\big]{ \e^{ As } }_{ L(H) }^2 \biggr) \\
& \leq \norm{ v_1 }_H  \norm{ v_2 }_H \Bigl[ T \bigl( \abs{ F }_{ C_{ \mathrm{b} }^2( H, H ) }^2 + 2 \abs{ B }_{ C_{ \mathrm{b} }^2( H, L_2( U,H ) ) }^2 \bigr) \Bigr]^{ \nicefrac{1}{2} } \biggl[ \sup_{ s \in [0,T]} \norm[\big]{ \e^{ As } }_{ L(H) }^3 \biggr] \\
& \quad \cdot \exp \biggl( T \bigl[ \tfrac{1}{2} + 3 \abs{ F }_{ C_{ \mathrm{b} }^1( H, H ) } + 4 \abs{ B }_{ C_{ \mathrm{b} }^1( H, L_2( U,H ) ) }^2 \bigr] \sup_{ s \in [0,T]} \norm[\big]{ \e^{ As } }_{ L(H) }^4 \biggr).
\end{split}
\end{equation}
Next note that \eqref{eq:kolmogorov,5}, \eqref{eq:kolmogorov,6}, \eqref{eq:kolmogorov,1}, and \eqref{eq:kolmogorov,2} ensure for all $ (t,x) \in [0,T] \times H $, $ v_1, v_2 \in H $ that
\begin{equation}
\begin{split}
\abs[\big]{ \bigl( \tfrac{\partial}{\partial x} u \bigr) (t,x) v_1 } & = \abs[\big]{ \E \bigl[ \phi'( X_t^x ) X_t^{ x, v_1 } \bigr] } \leq \abs{ \phi }_{ C_{ \mathrm{b} }^1( H, \R ) } \E \bigl[ \norm[\big]{ X_t^{ x, v_1 } }_H \bigr] \\
& \leq \norm{ v_1 }_H \abs{ \phi }_{ C_{ \mathrm{b} }^1( H, \R ) } \biggl[ \sup_{ s \in [0,T]} \norm[\big]{ \e^{ As } }_{ L(H) } \biggr] \\
& \quad \cdot \exp \biggl( T \bigl[ \abs{ F }_{ C_{ \mathrm{b} }^1( H, H ) } + \tfrac{1}{2} \abs{ B }_{ C_{ \mathrm{b} }^1( H, L_2( U,H ) ) }^2 \bigr] \sup_{ s \in [0,T]} \norm[\big]{ \e^{ As } }_{ L(H) }^2 \biggr)
\end{split}
\end{equation}
and
\begin{equation}
\begin{split}
& \abs[\big]{ \bigl( \tfrac{\partial^2}{\partial x^2 } u \bigr) (t,x) ( v_1, v_2 ) } = \abs[\big]{ \E \bigl[  \phi''( X_t^x ) ( X_t^{ x, v_1 }, X_t^{ x, v_2 } ) + \phi'( X_t^x ) X_t^{ x, v_1, v_2 } \bigr] } \\
& \leq \abs{ \phi }_{ C_{ \mathrm{b} }^2( H, \R ) } \norm[\big]{ X_t^{ x, v_1 } }_{ L^2( \P; H ) } \norm[\big]{ X_t^{ x, v_2 } }_{ L^2( \P; H ) } + \abs{ \phi }_{ C_{ \mathrm{b} }^1( H, \R ) } \E \bigl[ \norm[\big]{ X_t^{ x, v_1, v_2 } }_H \bigr] \\
& \leq \norm{ v_1 }_H \norm{ v_2 }_H \abs{ \phi }_{ C_{ \mathrm{b} }^2( H, \R ) } \biggl[ \sup_{ s \in [0,T]} \norm[\big]{ \e^{ As } }_{ L(H) }^2 \biggr] \exp \biggl( T \bigl[ 2 \abs{ F }_{ C_{ \mathrm{b} }^1( H, H ) } + \abs{ B }_{ C_{ \mathrm{b} }^1( H, L_2( U,H ) ) }^2 \bigr] \sup_{ s \in [0,T]} \norm[\big]{ \e^{ As } }_{ L(H) }^2 \biggr) \\
& \quad + \norm{ v_1 }_H  \norm{ v_2 }_H \abs{ \phi }_{ C_{ \mathrm{b} }^1( H, \R ) } \Bigl[ T \bigl( \abs{ F }_{ C_{ \mathrm{b} }^2( H, H ) }^2 + 2 \abs{ B }_{ C_{ \mathrm{b} }^2( H, L_2( U,H ) ) }^2 \bigr) \Bigr]^{ \nicefrac{1}{2} } \biggl[ \sup_{ s \in [0,T]} \norm[\big]{ \e^{ As } }_{ L(H) }^3 \biggr] \\
& \qquad \cdot \exp \biggl( T \bigl[ \tfrac{1}{2} + 3 \abs{ F }_{ C_{ \mathrm{b} }^1( H, H ) } + 4 \abs{ B }_{ C_{ \mathrm{b} }^1( H, L_2( U,H ) ) }^2 \bigr] \sup_{ s \in [0,T]} \norm[\big]{ \e^{ As } }_{ L(H) }^4 \biggr) \\
& \leq \norm{ v_1 }_H  \norm{ v_2 }_H \norm{ \phi }_{ C_{ \mathrm{b} }^2( H, \R ) } \Bigl( 1 \vee \Bigl[ T \bigl( \abs{ F }_{ C_{ \mathrm{b} }^2( H, H ) }^2 + 2 \abs{ B }_{ C_{ \mathrm{b} }^2( H, L_2( U,H ) ) }^2 \bigr) \Bigr]^{ \nicefrac{1}{2} } \Bigr) \biggl[ \sup_{ s \in [0,T]} \norm[\big]{ \e^{ As } }_{ L(H) }^3 \biggr] \\
& \quad \cdot \exp \biggl( T \bigl[ \tfrac{1}{2} + 3 \abs{ F }_{ C_{ \mathrm{b} }^1( H, H ) } + 4 \abs{ B }_{ C_{ \mathrm{b} }^1( H, L_2( U,H ) ) }^2 \bigr] \sup_{ s \in [0,T]} \norm[\big]{ \e^{ As } }_{ L(H) }^4 \biggr).
\end{split}
\end{equation}
This completes the proof of Lemma~\ref{lem:kolmogorov}.
\end{proof}

\subsection{Setting} \label{subsec:setting_prelim}

Let $ ( H , \langle \cdot, \cdot \rangle_{ H } , \norm{\cdot}_{ H } ) $ be a separable $ \R $-Hilbert space, let $ \bbH \subseteq H $ be an orthonormal basis of $ H $, let $ \lambda \colon \bbH \to \R $ be a mapping such that $ \sup_{ h \in \bbH } \lambda_{h} < 0 $, let $ A \colon D(A) \subseteq H \to H $ be the linear operator such that $ D(A) = \bigl\{ v \in H \colon \sum_{ h \in \bbH } \abs{ \lambda_h \langle h,v \rangle_{H} }^2 < \infty \bigr\} $ and such that for all $ v \in D(A) $ it holds that $ A v = \sum_{ h \in \bbH} \lambda_{ h } \langle h,v \rangle_{ H } h $, let $ ( H_r , \langle \cdot , \cdot \rangle_{ H_r } , \norm{\cdot}_{ H_r } )$, $ r \in \R $, be a family of interpolation spaces associated to $ - A $, let $ ( \bfH_r , \langle \cdot , \cdot \rangle_{ \bfH_r } , \norm{\cdot}_{ \bfH_r } )$, $ r \in \R $, be the family of $ \R $-Hilbert spaces such that for all $ r \in \R $ it holds that $ ( \bfH_r , \langle \cdot , \cdot \rangle_{ \bfH_r } , \norm{\cdot}_{ \bfH_r } ) = \bigl( H_{ \nicefrac{r}{2} } \times H_{ \nicefrac{r}{2} - \nicefrac{1}{2} }, \langle \cdot, \cdot \rangle_{ H_{ \nicefrac{r}{2} } \times H_{ \nicefrac{r}{2} - \nicefrac{1}{2} } }, \norm{ \cdot }_{ H_{ \nicefrac{r}{2} } \times H_{ \nicefrac{r}{2} - \nicefrac{1}{2} } } \bigl) $, and let $ \bfA \colon D ( \bfA ) \subseteq \bfH_0 \to \bfH_0 $ be the linear operator such that $ D ( \bfA ) = \bfH_1 $ and such that for all $ (v,w) \in \bfH_1 $ it holds that $ \bfA( v , w ) = ( w , A v ) $.

\subsection{Basic properties of the deterministic wave equation} \label{subsec:basic_properties}

\subsubsection{Basic properties of interpolation spaces associated to the deterministic wave equation}

\begin{lemma} \label{lem:Lambda_interpolation}
Assume the setting in Section~\ref{subsec:setting_prelim} and let $ \bfLambda \colon D(\bfLambda) \subseteq \bfH_0 \to \bfH_0 $ be the linear operator such that $ D(\bfLambda) = \bfH_1 $ and such that for all $ (v,w) \in \bfH_1 $ it holds that
\begin{equation}
\bfLambda (v,w) = \twovector{ \sum_{ h \in \bbH }  \abs{\lambda_h}^{ \nicefrac{1}{2} } \langle h , v \rangle_{ H_0 } h }{ \sum_{ h \in \bbH }  \abs{\lambda_h}^{ \nicefrac{1}{2} } \bigl\langle \abs{ \lambda_h }^{\nicefrac{1}{2}} h , w \bigr\rangle_{ H_{ -\nicefrac{1}{2} } } \abs{ \lambda_h }^{\nicefrac{1}{2}} h }.
\end{equation}
Then the $ \R $-Hilbert spaces $ ( \bfH_r , \langle \cdot , \cdot \rangle_{ \bfH_r } , \norm{\cdot}_{ \bfH_r } )$, $ r \in \R $, are a family of interpolation spaces associated to $ \bfLambda $.
\end{lemma}
\begin{proof}[Proof of Lemma~\ref{lem:Lambda_interpolation}]
Observe that $ \bfLambda \colon D(\bfLambda) \subseteq \bfH_0 \to \bfH_0 $ is a symmetric diagonal linear operator with $ \inf( \sigma_{ \mathrm{P} }( \bfLambda ) ) > 0 $ and that for all $ r \in [ 0, \infty ) $ it holds that
\begin{equation}
\begin{split}
& D( \bfLambda^r ) = \biggl\{ x \in \bfH_0 \colon \sum_{ h \in \bbH } \abs{ \lambda_h }^{ r } \abs{ \langle ( h, 0 ), x \rangle_{ \bfH_0 } }^2 + \abs{ \lambda_h }^{ r } \abs[\big]{ \bigl\langle \bigl( 0, \abs{ \lambda_h }^{ \nicefrac{1}{2} } h \bigr), x \bigr\rangle_{ \bfH_0 } }^2 < \infty \biggr\} \\
& = \biggl\{ ( v, w ) \in \bfH_0 \colon \sum_{ h \in \bbH } \abs{ \lambda_h }^{ r } \abs{ \langle h, v \rangle_{ H_0 } }^2 + \abs{ \lambda_h }^{ r } \abs[\big]{ \bigl\langle \abs{ \lambda_h }^{ \nicefrac{1}{2} } h , w \bigr\rangle_{ H_{ - \nicefrac{1}{2} } } }^2 < \infty \biggr\} \\
& = \biggl\{ v \in H_0 \colon \sum_{ h \in \bbH } \abs{ \lambda_h }^{ r } \abs{ \langle h, v \rangle_{ H_0 } }^2 < \infty \biggr\} \times \biggl\{ w \in H_{ - \nicefrac{1}{2} } \colon \sum_{ h \in \bbH } \abs{ \lambda_h }^{ r-1 } \abs[\big]{ \langle h , w \rangle_{ H_0 } }^2 < \infty \biggr\}  \\
& = H_{ \nicefrac{r}{2} } \times H_{ \nicefrac{r}{2} - \nicefrac{1}{2} } = \bfH_r.
\end{split}
\end{equation}
Moreover, for all $ r \in [ 0, \infty ) $, $ x_1 = ( v_1, w_1 ), x_2 = ( v_2, w_2 )  \in \bfH_r $ it holds that
\begin{equation}
\begin{split}
& \langle \bfLambda^r x_1, \bfLambda^r x_2 \rangle_{ \bfH_0 } = \biggl\langle \sum_{ h \in \bbH } \abs{\lambda_h}^{ \nicefrac{r}{2} } \langle h , v_1 \rangle_{ H_0 } h, \sum_{ h \in \bbH }  \abs{\lambda_h}^{ \nicefrac{r}{2} } \langle h , v_2 \rangle_{ H_0 } h \biggr\rangle_{ H_0 } \\
& \quad + \biggl\langle \sum_{ h \in \bbH } \abs{\lambda_h}^{ \nicefrac{r}{2} } \bigl\langle \abs{ \lambda_h }^{\nicefrac{1}{2}} h, w_1 \bigr\rangle_{ H_{ -\nicefrac{1}{2} } } \abs{ \lambda_h }^{\nicefrac{1}{2}} h, \sum_{ h \in \bbH }  \abs{\lambda_h}^{ \nicefrac{r}{2} } \bigl\langle \abs{ \lambda_h }^{\nicefrac{1}{2}} h, w_2 \bigr\rangle_{ H_{ -\nicefrac{1}{2} } } \abs{ \lambda_h }^{\nicefrac{1}{2}} h \biggr\rangle_{ H_{ - \nicefrac{1}{2} } } \\
& = \langle (-A)^{ \nicefrac{r}{2} } v_1, (-A)^{ \nicefrac{r}{2} } v_2 \rangle_{ H_0 } + \langle (-A)^{ \nicefrac{r}{2} } w_1, (-A)^{ \nicefrac{r}{2} } w_2 \rangle_{ H_{ - \nicefrac{1}{2} } } \\
& = \langle v_1, v_2 \rangle_{ H_{ \nicefrac{r}{2} } } + \langle w_1, w_2 \rangle_{ H_{ \nicefrac{r}{2} - \nicefrac{1}{2} } } = \langle x_1, x_2 \rangle_{ \bfH_r }.
\end{split}
\end{equation}
In addition, for all $ r \in ( - \infty, 0 ] $, $ x = ( v, w ) \in \bfH_0 $ it holds that
\begin{equation}
\begin{split}
\norm{ \bfLambda^r x }_{ \bfH_0 }^2 &= \norm[\bigg]{ \sum_{ h \in \bbH } \abs{\lambda_h}^{ \nicefrac{r}{2} } \langle h , v \rangle_{ H_0 } h }_{ H_0 }^2 + \norm[\bigg]{ \sum_{ h \in \bbH } \abs{\lambda_h}^{ \nicefrac{r}{2} } \bigl\langle \abs{ \lambda_h }^{\nicefrac{1}{2}} h, w \bigr\rangle_{ H_{ -\nicefrac{1}{2} } } \abs{ \lambda_h }^{\nicefrac{1}{2}} h }_{ H_{ - \nicefrac{1}{2} } }^2 \\
& = \norm{ (-A)^{ \nicefrac{r}{2} } v }_{ H_0 }^2 + \norm{ (-A)^{ \nicefrac{r}{2} } w }_{ H_{ - \nicefrac{1}{2} } }^2 = \norm{ v }_{ H_{ \nicefrac{r}{2} } }^2 + \norm{ w }_{ H_{ \nicefrac{r}{2} - \nicefrac{1}{2} } }^2 = \norm{ x }_{ \bfH_r }^2.
\end{split}
\end{equation}
This completes the proof of Lemma~\ref{lem:Lambda_interpolation}.
\end{proof}

\subsubsection{Basic properties of the deterministic linear wave equation}

The next elementary and well-known lemma can be found in a slightly different form, e.g., in Section~5.3 in Lindgren~\cite{Lindgren2012}.

\begin{lemma} \label{lem:semigroup}
Assume the setting in Section \ref{subsec:setting_prelim} and let $ \bfS \colon [ 0 , \infty ) \to L( \bfH_0 ) $ be the mapping such that for all $ t \in [ 0 , \infty ) $, $ (v,w) \in \bfH_0 $ it holds that
\begin{equation}
\bfS_t (v,w) = \twovector{ \cos\bigl( (-A)^{ \nicefrac{1}{2} } t \bigr) v + (-A)^{ - \nicefrac{1}{2} } \sin\bigl( (-A)^{ \nicefrac{1}{2} } t \bigr) w }{ - (-A)^{ \nicefrac{1}{2} } \sin\bigl( (-A)^{ \nicefrac{1}{2} } t \bigr) v + \cos\bigl( (-A)^{ \nicefrac{1}{2} } t \bigr) w }.
\end{equation}
Then $ \bfS \colon [ 0 , \infty ) \to L( \bfH_0 ) $ is a strongly continuous semigroup of bounded linear operators on $ \bfH_0 $ and $ \bfA \colon D ( \bfA ) \subseteq \bfH_0 \to \bfH_0 $ is the generator of $ \bfS $.
\end{lemma}

\begin{lemma} \label{lem:length_preservation}
Assume the setting in Section~\ref{subsec:setting_prelim}. Then for all $ t \in [ 0 , \infty ) $, $ x \in \bfH_0 $ it holds that $ \norm{ \e^{ \bfA t } x }_{ \bfH_0 } = \norm{ x }_{ \bfH_0 } $ and $ \sup_{ s \in [ 0, \infty ) } \norm{ \e^{ \bfA s } }_{ L( \bfH_0 ) } = 1 $.
\end{lemma}
\begin{proof}[Proof of Lemma~\ref{lem:length_preservation}]
Lemma~\ref{lem:semigroup} implies for all $ t \in [ 0 , \infty ) $, $ x = ( v, w ) \in \bfH_1 $ that
\begin{equation}
\begin{split}
\norm{ \e^{ \bfA t } x }_{ \bfH_0 }^2 & = \norm[\big]{ \cos\bigl( (-A)^{ \nicefrac{1}{2} } t \bigr) v + (-A)^{ - \nicefrac{1}{2} } \sin\bigl( (-A)^{ \nicefrac{1}{2} } t \bigr) w }_{ H_0 }^2 \\
& \quad + \norm[\big]{ - (-A)^{ \nicefrac{1}{2} } \sin\bigl( (-A)^{ \nicefrac{1}{2} } t \bigr) v + \cos\bigl( (-A)^{ \nicefrac{1}{2} } t \bigr) w }_{ H_{ - \nicefrac{1}{2} } }^2 \\
& = \norm[\big]{ \cos\bigl( (-A)^{ \nicefrac{1}{2} } t \bigr) v }_{ H_0 }^2 + \norm[\big]{ (-A)^{ - \nicefrac{1}{2} } \sin\bigl( (-A)^{ \nicefrac{1}{2} } t \bigr) w }_{ H_0 }^2 \\
& \quad + \norm[\big]{ (-A)^{ \nicefrac{1}{2} } \sin\bigl( (-A)^{ \nicefrac{1}{2} } t \bigr) v }_{ H_{ - \nicefrac{1}{2} } }^2 +  \norm[\big]{ \cos\bigl( (-A)^{ \nicefrac{1}{2} } t \bigr) w }_{ H_{ - \nicefrac{1}{2} } }^2 \\
& \quad + 2 \langle \cos\bigl( (-A)^{ \nicefrac{1}{2} } t \bigr) v, (-A)^{ - \nicefrac{1}{2} } \sin\bigl( (-A)^{ \nicefrac{1}{2} } t \bigr) w \rangle_{ H_0 } \\
& \quad - 2 \langle \sin\bigl( (-A)^{ \nicefrac{1}{2} } t \bigr) v, (-A)^{ \nicefrac{1}{2} } \cos\bigl( (-A)^{ \nicefrac{1}{2} } t \bigr) w \rangle_{ H_{ - \nicefrac{1}{2} } } \\
& = \norm[\big]{ \cos\bigl( (-A)^{ \nicefrac{1}{2} } t \bigr) v }_{ H_0 }^2 + \norm[\big]{ \sin\bigl( (-A)^{ \nicefrac{1}{2} } t \bigr) v }_{ H_0 }^2 \\
& \quad + \norm[\big]{ \sin\bigl( (-A)^{ \nicefrac{1}{2} } t \bigr) w }_{ H_{ - \nicefrac{1}{2} } }^2 + \norm[\big]{ \cos\bigl( (-A)^{ \nicefrac{1}{2} } t \bigr) w }_{ H_{ - \nicefrac{1}{2} } }^2 \\
& = \norm{ v }_{ H_0 }^2 + \norm{ w }_{ H_{ - \nicefrac{1}{2} } }^2 = \norm{ x }_{ \bfH_0 }^2.
\end{split}
\end{equation}
This completes the proof of Lemma~\ref{lem:length_preservation}.
\end{proof}

\begin{lemma} \label{lem:commute}
Assume the setting in Section~\ref{subsec:setting_prelim},  let $ P_{I} \colon \bigcup_{r\in\R} H_r \to \bigcup_{r\in\R} H_r$, $ I \in \cP( \bbH ) $, be the mappings such that for all $ I \in \cP( \bbH ) $, $ r \in \R $, $v \in H_r $ it holds that $ P_{I} (v) =  \sum_{ h \in I }  \langle \abs{\lambda_h}^{-r} h , v \rangle_{ H_r } \abs{\lambda_h}^{-r} h $, and let $ \bfP_{I} \colon \bigcup_{ r \in \R } \bfH_r \to \bigcup_{ r \in \R } \bfH_r $, $ I \in \cP( \bbH ) $, be the mappings such that for all $ I \in \cP( \bbH ) $, $ r \in \R $, $ ( v , w ) \in \bfH_r $ it holds that $ \bfP_{I} ( v , w ) = \bigl( P_{I} ( v ) , P_{I} ( w ) \bigr) $. Then for all $ I \in \cP( \bbH ) $, $ x \in \bfH_1 $ it holds that $ \bfA \bfP_I (x) = \bfP_I \bfA (x) $ and for all $ I \in \cP( \bbH ) $, $ t \in [ 0 , \infty ) $, $ x \in \bfH_0 $ it holds that $ \e^{ \bfA t } \bfP_I (x) = \bfP_I \e^{ \bfA t } (x) $.
\end{lemma}
\begin{proof}[Proof of Lemma~\ref{lem:commute}]
For all $ I \in \cP( \bbH ) $, $ x = (v, w) \in \bfH_1 $ it holds that
\begin{equation}
\bfP_I \bfA( x ) = \bfP_I ( w, Av ) = ( P_I (w) , P_I Av ) = ( P_I (w) , A P_I (v) ) = \bfA \bfP_I ( x ).
\end{equation}
In addition, Lemma~\ref{lem:semigroup} shows for all $ I \in \cP( \bbH ) $, $ t \in [ 0 , \infty ) $, $ x = ( v, w ) \in \bfH_0 $ that
\begin{equation}
\begin{split}
\e^{ \bfA t } \bfP_I (x) & = \twovector{ \cos\bigl( (-A)^{ \nicefrac{1}{2} } t \bigr) P_I (v) + (-A)^{ - \nicefrac{1}{2} } \sin\bigl( (-A)^{ \nicefrac{1}{2} } t \bigr) P_I (w) }{ - (-A)^{ \nicefrac{1}{2} } \sin\bigl( (-A)^{ \nicefrac{1}{2} } t \bigr) P_I (v) + \cos\bigl( (-A)^{ \nicefrac{1}{2} } t \bigr) P_I (w) } \\
& = \twovector{ P_I \bigl( \cos\bigl( (-A)^{ \nicefrac{1}{2} } t \bigr) v + (-A)^{ - \nicefrac{1}{2} } \sin\bigl( (-A)^{ \nicefrac{1}{2} } t \bigr) w \bigr) }{ P_I \bigl( - (-A)^{ \nicefrac{1}{2} } \sin\bigl( (-A)^{ \nicefrac{1}{2} } t \bigr) v + \cos\bigl( (-A)^{ \nicefrac{1}{2} } t \bigr) w \bigr) } \\
& =  \bfP_I \e^{ \bfA t } (x).
\end{split}
\end{equation}
The proof of Lemma~\ref{lem:commute} is thus completed.
\end{proof}

\section{Upper bounds for weak errors} \label{sec:WeakConvergenceRates}

\subsection{Setting}\label{subsec:WeakSetting}

Assume the setting in Section~\ref{subsec:setting}, let $ ( H , \langle \cdot, \cdot \rangle_{ H } , \norm{\cdot}_{ H } ) $ be a separable $ \R $-Hilbert space, let $ \bbH \subseteq H $ be an orthonormal basis of $ H $, let $ \lambda \colon \bbH \to \R $ be a mapping such that $ \sup_{ h \in \bbH } \lambda_{h} < 0 $, let $ A \colon D(A) \subseteq H \to H $ be the linear operator such that $ D(A) = \bigl\{ v \in H \colon \sum_{ h \in \bbH } \abs{ \lambda_h \langle h,v \rangle_{H} }^2 < \infty \bigr\} $ and such that for all $ v \in D(A) $ it holds that $ A v = \sum_{ h \in \bbH} \lambda_{ h } \langle h,v \rangle_{ H } h $, let $ ( H_r , \langle \cdot , \cdot \rangle_{ H_r } , \norm{\cdot}_{ H_r } )$, $ r \in \R $, be a family of interpolation spaces associated to $ - A $, let $ ( \bfH_r , \langle \cdot , \cdot \rangle_{ \bfH_r } , \norm{\cdot}_{ \bfH_r } )$, $ r \in \R $, be the family of $ \R $-Hilbert spaces such that for all $ r \in \R $ it holds that $ ( \bfH_r , \langle \cdot , \cdot \rangle_{ \bfH_r } , \norm{\cdot}_{ \bfH_r } ) = \bigl( H_{ \nicefrac{r}{2} } \times H_{ \nicefrac{r}{2} - \nicefrac{1}{2} }, \langle \cdot, \cdot \rangle_{ H_{ \nicefrac{r}{2} } \times H_{ \nicefrac{r}{2} - \nicefrac{1}{2} } }, \norm{ \cdot }_{ H_{ \nicefrac{r}{2} } \times H_{ \nicefrac{r}{2} - \nicefrac{1}{2} } } \bigl) $, let $ P_{I} \colon \bigcup_{r\in\R} H_r \to \bigcup_{r\in\R} H_r$, $ I \in \cP( \bbH ) $, be the mappings such that for all $ I \in \cP( \bbH ) $, $ r \in \R $, $v \in H_r $ it holds that $ P_{I} (v) =  \sum_{ h \in I }  \langle \abs{\lambda_h}^{-r} h , v \rangle_{ H_r } \abs{\lambda_h}^{-r} h $, let $ \bfP_{I} \colon \bigcup_{ r \in \R } \bfH_r \to \bigcup_{ r \in \R } \bfH_r $, $ I \in \cP( \bbH ) $, be the mappings such that for all $ I \in \cP( \bbH ) $, $ r \in \R $, $ ( v , w ) \in \bfH_r $ it holds that $ \bfP_{I} ( v , w ) = \bigl( P_{I} ( v ) , P_{I} ( w ) \bigr) $, let $ \bfA \colon D ( \bfA ) \subseteq \bfH_0 \to \bfH_0 $ be the linear operator such that $ D ( \bfA ) = \bfH_1 $ and such that for all $ (v,w) \in \bfH_1 $ it holds that $ \bfA( v , w ) = ( w , A v ) $, let $ \bfLambda \colon D(\bfLambda) \subseteq \bfH_0 \to \bfH_0 $ be the linear operator such that $ D(\bfLambda) = \bfH_1 $ and such that for all $ (v,w) \in \bfH_1 $ it holds that \smash{$ \bfLambda (v,w) = \bigl( \sum_{ h \in \bbH }  \abs{\lambda_h}^{ \nicefrac{1}{2} } \langle h , v \rangle_{ H_0 } h , \sum_{ h \in \bbH }  \abs{\lambda_h}^{ \nicefrac{1}{2} } \bigl\langle \abs{ \lambda_h }^{\nicefrac{1}{2}} h , w \bigr\rangle_{ H_{ -1/2 } } \abs{ \lambda_h }^{\nicefrac{1}{2}} h \bigr) $}, and let $ \gamma \in ( 0,\infty ) $, $ \beta \in ( \nicefrac{\gamma}{2}, \gamma ] $, $ \rho \in [0, 2( \gamma - \beta ) ] $, $ C_{ \bfF }, C_{ \bfB } \in [ 0, \infty) $, $ \xi \in L^2 ( \P \vert_{\cF_0} ; \bfH_{ 2( \gamma - \beta ) } ) $, $ \bfF \in \mathrm{Lip}^0( \bfH_0, \bfH_0 ) $, $ \bfB \in \mathrm{Lip}^0( \bfH_0, L_2( U, \bfH_0 ) ) $ satisfy that $ \bfLambda^{ -\beta } \in L_2( \bfH_0 ) $, $ \bfF \vert_{ \bfH_\rho } \in \mathrm{Lip}^0( \bfH_\rho, \bfH_{ 2( \gamma - \beta ) } ) $, $ \bfB \vert_{ \bfH_\rho } \in \mathrm{Lip}^0( \bfH_\rho, L_2( U, \bfH_\rho ) \cap L( U, \bfH_\gamma ) ) $, \smash{$ \bfF \vert_{ \bigcap_{ r \in \R } \bfH_r } \in C_{\mathrm{b}}^2( \bigcap_{ r \in \R } \bfH_r, \bfH_0 ) $}, \smash{$ \bfB \vert_{ \bigcap_{ r \in \R } \bfH_r } \in C_{ \mathrm{b} }^2( \bigcap_{ r \in \R } \bfH_r, $} \smash{$ L_2( U, \bfH_0 ) ) $}, \smash{$ C_{ \bfF } \ = \ \sup_{ x \in \cap_{ r \in \R } \bfH_r } \sup_{ v_1, v_2 \in \cap_{ r \in \R } \bfH_r, \; \norm{ v_1 }_{ \bfH_0 } \vee \norm{ v_2 }_{ \bfH_0 } \leq 1 } \; \norm{ \bfF''(x)( v_1, v_2 ) }_{ \bfH_0 } \; < \; \infty $}, and \smash{$ C_{ \bfB } = $} \smash{$ \sup_{ x \in \cap_{ r \in \R } \bfH_r } \sup_{ v_1, v_2 \in \cap_{ r \in \R } \bfH_r, \; \norm{ v_1 }_{ \bfH_0 } \vee \norm{ v_2 }_{ \bfH_0 } \leq 1 }  \norm{ \bfB''(x)( v_1, v_2 ) }_{ L_2( U, \bfH_0) } < \infty $}.

\subsection{Weak convergence rates for the Galerkin approximation} \label{subsec:weak_rates}

\begin{remark} \label{rmk:existence_regularity}
Assume the setting in Section \ref{subsec:WeakSetting}. The assumptions that $ \bfF \vert_{ \bfH_\rho } \in \mathrm{Lip}^0( \bfH_\rho, \bfH_{ 2( \gamma - \beta ) } ) $ and $ \bfB \vert_{ \bfH_\rho } \in \mathrm{Lip}^0( \bfH_\rho, L_2( U, \bfH_\rho ) \cap L( U, \bfH_\gamma ) ) $ then ensure that $ \bfF \vert_{ \bfH_\rho } \in \mathrm{Lip}^0( \bfH_\rho, \bfH_\rho ) $ and $ \bfB \vert_{ \bfH_\rho } \in \mathrm{Lip}^0( \bfH_\rho, L_2( U, \bfH_\rho ) ) $ and Theorem~\ref{thm:existence} hence shows that there exist up to modifications unique $ ( \cF_t )_{ t \in [0,T] } $-predictable stochastic processes $ X^I \colon [ 0 , T ] \times \Omega \to \bfP_I ( \bfH_\rho )$, $ I \in \cP( \bbH ) $, satisfying for all $ I \in \cP( \bbH ) $, $t \in [ 0 , T ] $ that $ \sup_{ s \in [ 0 , T ] } \norm{ X_s^I }_{ L^2 ( \P ; \bfH_\rho ) } < \infty $ and $ \P $-a.s.\ that
\begin{equation}
X_t^I =  \e^{\bfA t} \bfP_I \xi + \int_0^t \e^{ \bfA(t-s) } \bfP_I \bfF(X_s^I) \ds  + \int_0^t \e^{ \bfA(t-s) } \bfP_I \bfB(X_s^I) \dWs.
\end{equation}
\end{remark}

The following lemma provides global $L^2$-bounds on the stochastic processes $ X^I \colon [ 0 , T ] \times \Omega \to \bfP_I ( \bfH_\rho )$, $ I \in \cP( \bbH ) $, in Remark~\ref{rmk:existence_regularity}.

\begin{lemma} \label{lem:finiteness}
Assume the setting in Section~\ref{subsec:WeakSetting} and let $ X^I \colon [ 0 , T ] \times \Omega \to \bfP_{I} ( \bfH_\rho ) $, $ I \in \cP( \bbH ) $, be $( \cF_t )_{ t \in [0,T] } $-predictable stochastic processes such that for all $ I \in \cP( \bbH ) $, $ t \in [0,T] $ it holds that $ \sup_{ s \in [ 0 , T ] } \norm{ X_s^I }_{ L^2 ( \P ; \bfH_\rho ) } < \infty $ and $ \P $-a.s.\ that
\begin{equation}
X_t^I =  \e^{\bfA t} \bfP_{I} \xi + \int_0^t \e^{ \bfA(t-s) } \bfP_I \bfF(X_s^I) \ds  + \int_0^t \e^{ \bfA(t-s) } \bfP_I \bfB(X_s^I) \dWs.
\end{equation}
Then
\begin{equation}
\begin{split}
& \sup_{I \in \cP( \bbH )} \sup_{ t \in [ 0 , T ] }  \bigl( 1 \vee \norm[\big]{ X_t^I }_{ L^2( \P; \bfH_\rho ) } \bigr) \leq \bigl( 1 \vee \norm{ \xi }_{ L^2( \P; \bfH_\rho ) } \bigr) \\
& \quad \cdot \exp \Bigl( T \Bigl[ \norm[\big]{\bfF \vert_{ \bfH_\rho }  }_{ \mathrm{Lip}^0( \bfH_\rho, \bfH_\rho ) } + \tfrac{1}{2} \norm[\big]{ \bfB \vert_{ \bfH_\rho } }_{ \mathrm{Lip}^0( \bfH_\rho, L_2( U, \bfH_\rho ) ) }^2 \Bigr] \Bigr) < \infty.
\end{split}
\end{equation}
\end{lemma}

\begin{proof}[Proof of Lemma~\ref{lem:finiteness}]
Corollary~1 in Da~Prato et al.~\cite{DaPratoJentzenRoeckner2010}, Lemma~\ref{lem:length_preservation}, and the Cauchy-Schwarz inequality ensure for all $ I \in \cP( \bbH ) $, $ t \in [0,T] $ that
\begin{equation}
\begin{split}
&\begin{aligned}
\E \bigl[ \norm[\big]{ X_t^I }_{ \bfH_\rho }^2 \bigr] & = \E \bigl[ \norm[\big]{ \e^{\bfA t} \bfP_{I} \xi }_{ \bfH_\rho }^2 \bigr] + 2 \int_{0}^{t} \E \bigl[ \bigl\langle \e^{ \bfA(t-s) } X_s^I, \e^{ \bfA(t-s) } \bfP_I \bfF(X_s^I) \bigr\rangle_{ \bfH_\rho } \bigr] \ds \\
& \quad + \int_{0}^{t} \E \bigl[ \norm[\big]{ \e^{ \bfA(t-s) } \bfP_I \bfB(X_s^I) }_{ L_2( U, \bfH_\rho ) }^2 \bigr] \ds \\
\end{aligned} \\
& \leq \E \bigl[ \norm{ \bfP_{I} \xi }_{ \bfH_\rho }^2 \bigr] + 2 \int_{0}^{t} \Bigl( \norm{ \bfP_I \bfF(0) }_{ \bfH_\rho } \E \bigl[ \norm[\big]{ X_s^I }_{ \bfH_\rho } \bigr] + \abs[\big]{ \bfP_I \bfF \vert_{ \bfH_\rho } }_{ \mathrm{Lip}^0( \bfH_\rho, \bfH_\rho ) } \E \bigl[ \norm[\big]{ X_s^I }_{ \bfH_\rho }^2 \bigr] \Bigr) \ds \\
& \quad + \int_{0}^{t} \Bigl( \norm{ \bfP_I \bfB(0) }_{ L_2( U, \bfH_\rho ) }^2 + 2 \norm{ \bfP_I \bfB(0) }_{ L_2( U, \bfH_\rho ) } \abs[\big]{ \bfP_I \bfB \vert_{ \bfH_\rho } }_{ \mathrm{Lip}^0( \bfH_\rho, L_2( U, \bfH_\rho ) ) } \E \bigl[ \norm[\big]{ X_s^I }_{ \bfH_\rho } \bigr] \\
& \qquad\qquad + \abs[\big]{ \bfP_I \bfB \vert_{ \bfH_\rho } }_{ \mathrm{Lip}^0( \bfH_\rho, L_2( U, \bfH_\rho ) ) }^2 \E \bigl[ \norm[\big]{ X_s^I }_{ \bfH_\rho }^2 \bigr] \Bigr) \ds \\
& \leq \E \bigl[ \norm{ \bfP_{I} \xi }_{ \bfH_\rho }^2 \bigr] + \Bigl( 2 \norm[\big]{ \bfP_I \bfF \vert_{ \bfH_\rho } }_{ \mathrm{Lip}^0( \bfH_\rho, \bfH_\rho ) } + \norm[\big]{ \bfP_I \bfB \vert_{ \bfH_\rho } }_{ \mathrm{Lip}^0( \bfH_\rho, L_2( U, \bfH_\rho ) ) }^2 \Bigr) \int_{0}^{t} 1 \vee \E \bigl[ \norm[\big]{ X_s^I }_{ \bfH_\rho }^2 \bigr] \ds.
\end{split} \raisetag{3.8cm}
\end{equation} 
Gronwall's lemma hence implies for all $ I \in \cP( \bbH ) $ that
\begin{equation}
\begin{split}
\sup_{t \in [0,T]} \bigl( 1 \vee \norm[\big]{ X_t^I }_{ L^2( \P; \bfH_\rho ) } \bigr) & \leq \bigl( 1 \vee \norm{ \bfP_{I} \xi }_{ L^2( \P; \bfH_\rho ) } \bigr) \\
& \quad \cdot \exp \Bigl( T \Bigl[ \norm[\big]{ \bfP_I \bfF \vert_{ \bfH_\rho }  }_{ \mathrm{Lip}^0( \bfH_\rho, \bfH_\rho ) } + \tfrac{1}{2} \norm[\big]{ \bfP_I \bfB \vert_{ \bfH_\rho } }_{ \mathrm{Lip}^0( \bfH_\rho, L_2( U, \bfH_\rho ) ) }^2 \Bigr] \Bigr) \\
& \leq \bigl( 1 \vee \norm{ \xi }_{ L^2( \P; \bfH_\rho ) } \bigr) \\
& \quad \cdot \exp \Bigl( T \Bigl[ \norm[\big]{\bfF \vert_{ \bfH_\rho }  }_{ \mathrm{Lip}^0( \bfH_\rho, \bfH_\rho ) } + \tfrac{1}{2} \norm[\big]{ \bfB \vert_{ \bfH_\rho } }_{ \mathrm{Lip}^0( \bfH_\rho, L_2( U, \bfH_\rho ) ) }^2 \Bigr] \Bigr).
\end{split}
\end{equation}
The proof of Lemma~\ref{lem:finiteness} is thus completed.
\end{proof}

\begin{lemma} \label{lem:perturbation}
Assume the setting in Section~\ref{subsec:WeakSetting} and let $ X^I \colon [ 0 , T ] \times \Omega \to \bfP_{I} ( \bfH_0 ) $, $ I \in \cP( \bbH ) $, be $( \cF_t )_{ t \in [0,T] } $-predictable stochastic processes such that for all $ I \in \cP( \bbH ) $, $ t \in [0,T] $ it holds that $ \sup_{ s \in [ 0 , T ] } \norm{ X_s^I }_{ L^2 ( \P ; \bfH_0 ) } < \infty $ and $ \P $-a.s.\ that
\begin{equation}
X_t^I =  \e^{\bfA t} \bfP_{I} \xi + \int_0^t \e^{ \bfA(t-s) } \bfP_I \bfF(X_s^I) \ds  + \int_0^t \e^{ \bfA(t-s) } \bfP_I \bfB(X_s^I) \dWs.
\end{equation}
Then it holds for all $ I, J \in \cP( \bbH ) $ that
\begin{equation} \label{eq:perturbation,1}
\begin{split}
\sup_{t \in [0,T]} \norm[\big]{ X_t^I - X_t^J }_{L^2(\P; \bfH_0)} &\leq \sqrt{2}\tp \cE_1\bigl[ \sqrt{2}\tp T \abs{ \bfP_{ I \cap J } \bfF }_{ \mathrm{Lip}^0( \bfH_0, \bfH_0 ) } + \sqrt{2T} \abs{ \bfP_{ I \cap J } \bfB }_{ \mathrm{Lip}^0( \bfH_0, L_2( U, \bfH_0 ) ) } \bigr] \\
&\quad \cdot \sup_{ t \in [0,T] } \norm[\big]{ \bfP_{ I \setminus J } X_t^I + \bfP_{ J \setminus I } X_t^J }_{ L^2( \P; \bfH_0 ) } < \infty.
\end{split}
\end{equation}
\end{lemma}

\begin{proof}[Proof of Lemma~\ref{lem:perturbation}]
Observe that Corollary~3.1 in Jentzen \& Kurniawan~\cite{JentzenKurniawan2015} and Lemma~\ref{lem:length_preservation} imply for all $ I, J \in \cP( \bbH ) $ that
\begin{equation}
\begin{split}
&\sup_{t \in [0,T]} \norm[\big]{ X_t^I - X_t^J }_{L^2( \P; \bfH_0 )} \leq \sqrt{2}\tp \cE_1\bigl[ \sqrt{2}\tp T \abs{ \bfP_{ I \cap J } \bfF }_{ \mathrm{Lip}^0( \bfH_0, \bfH_0 ) } + \sqrt{2T} \abs{ \bfP_{ I \cap J } \bfB }_{ \mathrm{Lip}^0( \bfH_0, L_2( U, \bfH_0 ) ) } \bigr] \\
&\quad \cdot \sup_{ t \in [0,T] } \norm[\bigg]{ X_t^I - \biggl[ \int_{0}^{t} \e^{ \bfA(t-s) } \bfP_{ I \cap J } \bfF(X_s^I) \ds + \int_{0}^{t} \e^{ \bfA(t-s) } \bfP_{ I \cap J } \bfB(X_s^I) \dWs \biggr] \\
&\qquad\qquad\qquad\!\+ + \biggl[ \int_{0}^{t} \e^{ \bfA(t-s) } \bfP_{ I \cap J } \bfF(X_s^J) \ds + \int_{0}^{t} \e^{ \bfA(t-s) } \bfP_{ I \cap J } \bfB(X_s^J) \dWs \biggr] - X_t^J }_{ L^2( \P; \bfH_0 ) } \\
& = \sqrt{2}\tp \cE_1\bigl[ \sqrt{2}\tp T \abs{ \bfP_{ I \cap J } \bfF }_{ \mathrm{Lip}^0( \bfH_0, \bfH_0 ) } + \sqrt{2T} \abs{ \bfP_{ I \cap J } \bfB }_{ \mathrm{Lip}^0( \bfH_0, L_2( U, \bfH_0 ) ) } \bigr] \\
&\quad \cdot \sup_{ t \in [0,T] } \biggl\| X_t^I - \bfP_J \biggl( \e^{ \bfA(t-s) } \bfP_I \xi + \int_{0}^{t} \e^{ \bfA(t-s) } \bfP_{ I } \bfF(X_s^I) \ds + \int_{0}^{t} \e^{ \bfA(t-s) } \bfP_{ I } \bfB(X_s^I) \dWs \biggr) \\
& \quad + \bfP_I \biggl( \e^{ \bfA(t-s) } \bfP_J \xi + \int_{0}^{t} \e^{ \bfA(t-s) } \bfP_{ J } \bfF(X_s^J) \ds + \int_{0}^{t} \e^{ \bfA(t-s) } \bfP_{ J } \bfB(X_s^J) \dWs \biggr) - X_t^J \biggr\|_{ L^2( \P; \bfH_0 ) } \\
& = \sqrt{2}\tp \cE_1\bigl[ \sqrt{2}\tp T \abs{ \bfP_{ I \cap J } \bfF }_{ \mathrm{Lip}^0( \bfH_0, \bfH_0 ) } + \sqrt{2T} \abs{ \bfP_{ I \cap J } \bfB }_{ \mathrm{Lip}^0( \bfH_0, L_2( U, \bfH_0 ) ) } \bigr] \\
& \quad \cdot \sup_{ t \in [0,T] } \norm[\big]{ \bfP_{ I \setminus J } X_t^I - \bfP_{ J \setminus I } X_t^J }_{ L^2( \P; \bfH_0 ) }.
\end{split}
\end{equation}
This implies \eqref{eq:perturbation,1} and thus completes the proof of Lemma~\ref{lem:perturbation}.
\end{proof}

\begin{remark} \label{rmk:existence_kolmogorov}
Assume the setting in Section \ref{subsec:WeakSetting}. Then Remark~\ref{rem:existence} shows that there exist up to modifications unique $ ( \cF_t )_{ t \in [ 0 , T ] } $-predictable stochastic processes $ X^{ J , x } \colon [ 0 , T ] \times \Omega \to \bfP_J( \bfH_0 ) $, $ x \in \bfP_J( \bfH_0 ) $, $ J \in \cP( \bbH ) $, satisfying that for all $ J \in \cP( \bbH ) $, $ x \in \bfP_J( \bfH_0 ) $, $ t \in [ 0 , T ] $ it holds that $ \sup_{ s \in [ 0 , T ] } \norm{ X_s^{ J , x } }_{ L^2 ( \P ; \bfH_{0} ) } < \infty $ and $ \P $-a.s.\ that
\begin{equation}
X_t^{ J, x } = \e^{ \bfA t } x + \int_0^t \e^{ \bfA(t-s) } \bfP_J \bfF( X_s^{ J, x } ) \ds + \int_0^t \e^{ \bfA(t-s) } \bfP_J \bfB( X_s^{ J, x } ) \dWs.
\end{equation}
\end{remark}

\begin{lemma} \label{lem:finiteness_kolmogorov}
Assume the setting in Section \ref{subsec:WeakSetting}, let $ X^{ J , x } \colon [ 0 , T ] \times \Omega \to \bfP_J( \bfH_0 ) $, $ x \in \bfP_J( \bfH_0 ) $, $ J \in \cP_0( \bbH ) $, be $ ( \cF_t )_{ t \in [ 0 , T ] } $-predictable stochastic processes such that for all $ J \in \cP_0( \bbH ) $, $ x \in \bfP_J( \bfH_0 ) $, $ t \in [ 0 , T ] $ it holds that $ \sup_{ s \in [ 0 , T ] } \norm{ X_s^{ J , x } }_{ L^2 ( \P ; \bfP_J( \bfH_{0} ) ) } < \infty $ and $ \P $-a.s.\ that
\begin{equation}
X_t^{ J, x } = \e^{ \bfA t } x + \int_0^t \e^{ \bfA(t-s) } \bfP_J \bfF( X_s^{ J, x } ) \ds + \int_0^t \e^{ \bfA(t-s) } \bfP_J \bfB( X_s^{ J, x } ) \dWs,
\end{equation}
let $ \phi \in C^{2}_\mathrm{b}( \bfH_0 , \R ) $, and let $ u^J \colon [0,T] \times \bfP_J( \bfH_0 ) \to \R $, $ J \in \cP_0( \bbH ) $, be the mappings which satisfy for all $ J \in \cP_0( \bbH ) $, $ (t,x) \in [0,T] \times \bfP_J( \bfH_0 ) $ that $ u^J(t,x) = \E\bigl[ \phi \bigl(X_t^{ J, x }\bigr) \bigr] $. Then for all $ J \in \cP_0( \bbH ) $ it holds that $ u^J \in C^{ 1 , 2 }( [0,T] \times \bfP_J( \bfH_0 ), \R ) $ and
\begin{align}
\begin{split} \label{eq:finiteness_kolmogorov,1}
& \sup_{ K \in \cP_0( \bbH ) } \sup_{t \in [0,T]} \abs{ u^K( t, \cdot ) }_{ C_\mathrm{b}^1( \bfP_K( \bfH_0), \R ) } \\
& \leq \abs{ \phi }_{ C_{ \mathrm{b} }^1( \bfH_0, \R ) } \exp \bigl( T \bigl[ \abs{ \bfF }_{ \mathrm{Lip}^0( \bfH_0, \bfH_0 ) } + \tfrac{1}{2} \abs{ \bfB }_{ \mathrm{Lip}^0( \bfH_0, L_2( U, \bfH_0 ) ) }^2 \bigr] \bigr) < \infty, \\
\end{split} \\
\begin{split} \label{eq:finiteness_kolmogorov,2}
& \sup_{ K \in \cP_0( \bbH ) } \sup_{t \in [0,T]} \abs{ u^K( t, \cdot ) }_{ C_\mathrm{b}^2( \bfP_K( \bfH_0), \R ) } \leq \norm{ \phi }_{ C_{ \mathrm{b} }^2( \bfH_0, \R ) } \Bigl( 1 \vee \bigl[ T \bigl( C_{ \bfF }^2 + 2 C_{ \bfB }^2 \bigr) \bigr]^{ \nicefrac{1}{2} } \Bigr) \\
& \quad \cdot \exp \bigl( T \bigl[ \tfrac{1}{2} + 3 \abs{ \bfF }_{ \mathrm{Lip}^0( \bfH_0, \bfH_0 ) } + 4 \abs{ \bfB }_{ \mathrm{Lip}^0( \bfH_0, L_2( U, \bfH_0 ) ) }^2 \bigr] \bigr) < \infty.
\end{split} 
\end{align}
\end{lemma}
\begin{proof}[Proof of Lemma~\ref{lem:finiteness_kolmogorov}]
The assumptions that $ \bfF \vert_{ \bigcap_{ r \in \R } \bfH_r } \in C_{\mathrm{b}}^2( \bigcap_{ r \in \R } \bfH_r, \bfH_0 ) $ and $ \bfB \vert_{ \bigcap_{ r \in \R } \bfH_r } \in C_{ \mathrm{b} }^2( \bigcap_{ r \in \R } $ $ \bfH_r, L_2( U, \bfH_0 ) ) $ and the fact that for all $ J \in \cP_0( \bbH ) $ it holds that $ \bfP_J( \bfH_0 ) $ is a finite-dimensional $ \R $-vector space ensure that $ \bfP_J \bfF \vert_{ \bfP_J( \bfH_0 ) } \in C_{\mathrm{b}}^2( \bfP_J( \bfH_0 ), \bfP_J( \bfH_0 ) ) $ and $ \bfP_J \bfB \vert_{ \bfP_J( \bfH_0 ) } \in C_{ \mathrm{b} }^2( \bfP_J( \bfH_0 ), $ $ L_2( U, \bfP_J( \bfH_0 ) ) ) $. Lemma~\ref{lem:kolmogorov} and Lemma~\ref{lem:length_preservation} then prove for all $ J \in \cP_0( \bbH ) $ that
\begin{equation}
\begin{split}
& \sup_{t \in [0,T]} \abs{ u^J( t, \cdot ) }_{ C_\mathrm{b}^1( \bfP_J( \bfH_0 ) , \R ) } \leq \abs[\big]{ \phi \vert_{ \bfP_J( \bfH_0 ) } }_{ C_{ \mathrm{b} }^1( \bfP_J( \bfH_0 ), \R ) } \\
& \quad \cdot \exp \bigl( T \bigl[ \abs[\big]{ \bfP_J \bfF \vert_{ \bfP_J( \bfH_0 ) } }_{ C_{ \mathrm{b} }^1( \bfP_J( \bfH_0 ), \bfP_J( \bfH_0 ) ) } + \tfrac{1}{2} \abs[\big]{ \bfP_J \bfB \vert_{ \bfP_J( \bfH_0 ) } }_{ C_{ \mathrm{b} }^1( \bfP_J( \bfH_0 ), L_2( U, \bfP_J( \bfH_0 ) ) ) }^2 \bigr] \bigr)
\end{split}
\end{equation}
and that
\begin{equation}
\begin{split}
& \sup_{t \in [0,T]} \abs{ u^J( t, \cdot ) }_{ C_\mathrm{b}^2( \bfP_J( \bfH_0 ), \R ) } \leq \norm[\big]{ \phi \vert_{ \bfP_J( \bfH_0 ) } }_{ C_{ \mathrm{b} }^2( \bfP_J( \bfH_0 ), \R ) } \\
& \quad \cdot \Bigl( 1 \vee \Bigl[ T \bigl( \abs[\big]{ \bfP_J \bfF \vert_{ \bfP_J( \bfH_0 ) } }_{ C_{ \mathrm{b} }^2( \bfP_J( \bfH_0 ), \bfP_J( \bfH_0 ) ) }^2 + 2 \abs[\big]{ \bfP_J \bfB \vert_{ \bfP_J( \bfH_0 ) } }_{ C_{ \mathrm{b} }^2( \bfP_J( \bfH_0 ), L_2( U, \bfP_J( \bfH_0 ) ) ) }^2 \bigr) \Bigr]^{ \nicefrac{1}{2} } \Bigr) \\
& \quad \cdot \exp \bigl( T \bigl[ \tfrac{1}{2} + 3 \abs[\big]{ \bfP_J \bfF \vert_{ \bfP_J( \bfH_0 ) } }_{ C_{ \mathrm{b} }^1( \bfP_J( \bfH_0 ), \bfP_J( \bfH_0 ) ) } + 4 \abs[\big]{ \bfP_J \bfB \vert_{ \bfP_J( \bfH_0 ) } }_{ C_{ \mathrm{b} }^1( \bfP_J( \bfH_0 ), L_2( U, \bfP_J( \bfH_0 ) ) ) }^2 \bigr] \bigr) < \infty.
\end{split}
\end{equation}
This implies \eqref{eq:finiteness_kolmogorov,1} and \eqref{eq:finiteness_kolmogorov,2} and thus completes the proof of Lemma~\ref{lem:finiteness_kolmogorov}.
\end{proof}

In the proof of the main result of this article, Theorem~\ref{thm:weakrates} below, we use the following elementary and well-known lemma.

\begin{lemma} \label{lem:limsup}
Let $ p \in [0, \infty)$, let $ \cJ_n$, $ n \in \N_0 $, be sets such that $ \bigcup_{n \in \N} \cJ_n = \cJ_0 $ and such that for all $ n \in \N $ it holds that $ \cJ_n \subseteq \cJ_{n+1} $, and let $ g \colon \cJ_0 \to (0,\infty) $ be a mapping with the property that $ \sum_{h \in \cJ_0} (g_h)^p < \infty $. Then
\begin{equation} \label{eq:limsup,3}
\lim_{n \to \infty} \sup \bigl( \{ g_h \colon h \in \cJ_0 \setminus \cJ_n \} \cup \{ 0 \} \bigr) = 0.
\end{equation}
\end{lemma}

\begin{proof}[Proof of Lemma~\ref{lem:limsup}]
Without loss of generality we assume that $ p \in (0,\infty)$ (otherwise \eqref{eq:limsup,3} is clear). Then observe that for all $ n \in \N $ it holds that
\begin{equation} \label{eq:limsup,1}
\sum_{h \in \cJ_n} (g_h)^p + \bigl[ \sup \bigl( \{ g_h \colon h \in \cJ_0 \setminus \cJ_n \} \cup \{ 0 \} \bigr) \bigr]^p \leq \sum_{h \in \cJ_n} (g_h)^p + \sum_{h \in \cJ_0 \setminus \cJ_n} (g_h)^p = \sum_{h \in \cJ_0} (g_h)^p < \infty.
\end{equation}
Moreover, note that Lebesgue's theorem of dominated convergence proves that
\begin{equation} \label{eq:limsup,2}
\lim_{n \to \infty} \sum_{h \in \cJ_n} (g_h)^p = \sum_{h \in \cJ_0} (g_h)^p.
\end{equation}
Letting $ n \to \infty $ in \eqref{eq:limsup,1} and combining this with \eqref{eq:limsup,2} complete the proof of Lemma~\ref{lem:limsup}.
\end{proof}

\begin{theorem}\label{thm:weakrates}
Assume the setting in Section \ref{subsec:WeakSetting}, let $ X^I \colon [ 0 , T ] \times \Omega \to \bfP_{I} ( \bfH_\rho ) $, $ I \in \cP( \bbH ) $, and $ X^{ J , x } \colon [ 0 , T ] \times \Omega \to \bfP_J( \bfH_0 ) $, $ x \in \bfP_J( \bfH_0 ) $, $ J \in \cP_0( \bbH ) $, be $( \cF_t )_{ t \in [0,T] } $-predictable stochastic processes such that for all $ I \in \cP( \bbH ) $, $ J \in \cP_0( \bbH ) $, $ x \in \bfP_J( \bfH_0 ) $, $ t \in [ 0 , T ] $ it holds that $ \sup_{ s \in [ 0 , T ] } \bigl( \norm{ X_s^I }_{ L^2 ( \P ; \bfH_\rho ) } + \norm{ X_s^{ J , x } }_{ L^2 ( \P ; \bfH_{0} ) } \bigr) < \infty $ and $ \P $-a.s.\ that
\begin{align}
X_t^I & =  \e^{\bfA t} \bfP_{I} \xi + \int_0^t \e^{ \bfA(t-s) } \bfP_I \bfF(X_s^I) \ds  + \int_0^t \e^{ \bfA(t-s) } \bfP_I \bfB(X_s^I) \dWs,  \\
X_t^{ J, x } & = \e^{ \bfA t } x + \int_0^t \e^{ \bfA(t-s) } \bfP_J \bfF( X_s^{ J, x } ) \ds + \int_0^t \e^{ \bfA(t-s) } \bfP_J \bfB( X_s^{ J, x } ) \dWs,
\end{align}
let $ \phi \in C^{2}_\mathrm{b}( \bfH_0 , \R ) $, and let $ u^J \colon [0,T] \times \bfP_J( \bfH_0 ) \to \R $, $ J \in \cP_0( \bbH ) $, be the mappings which satisfy for all $ J \in \cP_0( \bbH ) $, $ (t,x) \in [0,T] \times \bfP_J( \bfH_0 ) $ that $ u^J(t,x) = \E\bigl[ \phi \bigl(X_t^{ J, x }\bigr) \bigr] $. Then it holds for all $ I \in \cP( \bbH ) \setminus \{ \bbH \} $ that
\begin{equation}
\begin{split}
& \abs[\big]{ \E \bigl[ \phi \bigl(X_T^\bbH \bigr)\bigr] - \E \bigl[\phi \bigl(X^I_T \bigr) \bigr] } \\
& \leq \biggl(  \abs{ \phi }_{ \mathrm{Lip}^{0} ( \bfH_0, \R ) } \norm{ \xi }_{ L^1(\P ; \bfH_{ 2(\gamma - \beta ) } ) } \exp \bigl( T \bigl[ \abs{ \bfF }_{ \mathrm{Lip}^0( \bfH_0, \bfH_0 ) } + \tfrac{1}{2} \abs{ \bfB }_{ \mathrm{Lip}^0( \bfH_0, L_2( U, \bfH_0 ) ) }^2 \bigr] \bigr) \\
&\quad + \biggl[ \sup_{K \in \cP_0( \bbH )} \sup_{ t \in [ 0 , T ] } \abs[\big]{ u^K (t, \cdot ) }_{ C_{\mathrm{b}}^1 ( \bfP_K( \bfH_0 ), \R ) } \biggr] \sup_{ K \in \cP_0( \bbH ) } \int_{0}^T \E \bigl[ \norm[\big]{ \bfF(X_s^K) }_{ \bfH_{ 2(\gamma-\beta) } } \bigr] \ds \\
&\quad + \norm{ \bfLambda^{ -\beta } }_{ L_2(\bfH_0) }^2 \biggl[ \sup_{ K \in \cP_0( \bbH ) } \sup_{ t \in [ 0 , T ] } \abs[\big]{ u^K( t, \cdot ) }_{ C_{\mathrm{b}}^2 ( \bfP_K( \bfH_0 ), \R ) } \biggr] \sup_{ K \in \cP_0( \bbH ) } \int_{0}^{T} \E \bigl[ \norm[\big]{ \bfB( X_s^K ) }_{L(U, \bfH_{\gamma})}^2 \bigr] \ds \biggr) \\
& \quad \cdot \biggl[ \inf_{ h \in \bbH \setminus I } \abs{ \lambda_h } \biggr]^{ \beta - \gamma } < \infty.
\end{split}
\end{equation}
\end{theorem}

\begin{proof}[Proof of Theorem~\ref{thm:weakrates}]
Throughout this proof let $ v^J, v_{ 1, 0 }^J \colon [0,T] \times \bfP_J( \bfH_0 ) \to \R $, $ J \in \cP_0( \bbH ) $, and $ v_{ 0 , \ell }^J  \colon [ 0, T ] \times \bfP_J( \bfH_0 ) \to L^{ (\ell) }( \bfP_J( \bfH_0 ), \R ) $, $ \ell \in \{ 1, 2 \} $, $ J \in \cP_0( \bbH ) $, be the mappings such that for all $ J \in \cP_0( \bbH ) $, $ ( k, \ell) \in \{ (1,0), (0,1), (0,2) \} $, $ (t,x) \in [0,T] \times \bfP_J( \bfH_0 ) $ it holds that $ v^J(t,x) = \E \bigl[ \phi \bigl(X_{T-t}^{ J, x } \bigr) \bigr] $ and that $  v_{ k , \ell }^J( t, x ) = \bigl( \tfrac{ \partial^{k + \ell} }{ \partial t^k \partial x^\ell } v^J \bigr) (t , x ) $. Note that for all $ J \in \cP_0( \bbH ) $, $ (t,x) \in [0,T] \times \bfP_J( \bfH_0 ) $ it holds that $ v^J (t,x) = u^J ( T - t, x ) $. Next observe for all $ J \in \cP_0( \bbH ) $, $ x \in \bfP_J( \bfH_0 ) $ that 
\begin{equation} \label{eq:weakrates,1}
\phi(x) = \E [ \phi( x ) ] = u^J ( 0 , x ) = v^J ( T , x ).
\end{equation} 
Moreover, note for all $ J \in \cP_0( \bbH ) $ that
\begin{equation} \label{eq:weakrates,2}
\E \bigl[ \phi \bigl( X_T^J \bigr) \bigr] = \E \bigl[ u^J \bigl( T , X_0^J \bigr) \bigr] = \E \bigl[ v^J \bigl( 0 , X_0^J \bigr) \bigr].
\end{equation}
Combining \eqref{eq:weakrates,1} and \eqref{eq:weakrates,2} shows for all $ J \in \cP_0( \bbH ) $, $ I \in \cP( J ) $ that
\begin{equation} \label{eq:weakrates,3}
\begin{split}
& \abs[\big]{ \E \bigl[ \phi \bigl( X_{ T }^J \bigr) \bigr] - \E \bigl[ \phi \bigl( X_{ T }^{ I } \bigr) \bigr] } = \abs[\big]{ \E \bigl[ \phi \bigl( X_{ T }^{ I } \bigr) \bigr] - \E \bigl[ \phi \bigl( X_{ T }^J \bigr) \bigr] } \\
&= \abs[\big]{ \E \bigl[ v^J \bigl( T , X_{ T }^{ I } \bigr) \bigr] - \E \bigl[ v^J \bigl( 0 , X_{ 0 }^J \bigr) \bigr] } \\
& \leq \abs[\big]{ \E \bigl[ v^J \bigl( T , X_{ T }^{ I } \bigr) \bigr] - \E \bigl[ v^J \bigl( 0 , X^I_0 \bigr) \bigr] } + \abs[\big]{ \E \bigl[ v^J \bigl( 0 , X_0^I\bigr) \bigr] - \E \bigl[ v^J \bigl( 0 , X_0^J \bigr) \bigr] }.
\end{split}
\end{equation}
In a first step we establish an estimate for the second summand on the right hand side of \eqref{eq:weakrates,3}. For this observe that Corollary~1 in Da~Prato et al.~\cite{DaPratoJentzenRoeckner2010}, the Cauchy-Schwarz inequality, and Lemma~\ref{lem:length_preservation} ensure for all $ J \in \cP_0( \bbH ) $, $ x, y \in \bfP_J( \bfH_0 ) $, $ t \in [0,T] $ that
\begin{equation} \label{eq:weakrates,18}
\begin{split}
& \E \bigl[ \norm[\big]{ X_t^{ J, x } - X_t^{ J, y } }_{ \bfH_0 }^2 \bigr] = \E \bigl[ \norm[\big]{ \e^{ \bfA t } \bigl( X_0^{ J, x } - X_0^{ J, y } \bigr) }_{ \bfH_0 }^2 \bigr] \\
&\quad + 2 \int_{0}^{t} \E \bigl[ \bigl\langle \e^{ \bfA(t-s) } \bigl( X_s^{ J, x } - X_s^{ J, y }  \bigr), \e^{ \bfA(t-s) } \bigl( \bfP_J \bfF( X_s^{ J, x } ) - \bfP_J \bfF( X_s^{ J, y } ) \bigr) \bigr\rangle_{ \bfH_0 } \bigr] \ds \\
&\quad + \int_{0}^{t} \E \bigl[ \norm[\big]{ \e^{ \bfA(t-s) } \bigl( \bfP_J \bfB( X_s^{ J, x } ) - \bfP_J \bfB( X_s^{ J, y } ) \bigr) }_{ L_2( U, \bfH_0 ) }^2 \bigr] \ds \\
& \leq \norm{x - y }_{ \bfH_0 }^2 + \bigl[ 2\abs{ \bfP_J \bfF }_{ \mathrm{Lip}^0( \bfH_0, \bfH_0 ) } + \abs{ \bfP_J \bfB }_{ \mathrm{Lip}^0( \bfH_0, L_2( U, \bfH_0 ) ) }^2 \bigr] \int_{0}^{t} \E \bigl[ \norm[\big]{ X_s^{ J, x } - X_s^{ J, y } }_{ \bfH_0 }^2 \bigr] \ds.
\end{split}
\end{equation}
Gronwall's lemma and Lemma~\ref{lem:perturbation} hence show for all $ J \in \cP_0( \bbH ) $, $ x, y \in \bfP_J( \bfH_0 ) $ that
\begin{equation}
\sup_{ t \in [ 0 , T ] } \norm[\big]{ X_t^{ J, x } - X_t^{ J, y}}_{ L^2 ( \P ; \bfH_0 ) } \leq \norm{ x - y }_{ \bfH_0 } \exp \bigl( T \bigl[ \abs{ \bfP_J \bfF }_{ \mathrm{Lip}^0( \bfH_0, \bfH_0 ) } + \tfrac{1}{2} \abs{ \bfP_J \bfB }_{ \mathrm{Lip}^0( \bfH_0, L_2( U, \bfH_0 ) ) }^2 \bigr] \bigr).
\end{equation}
This implies for all $ J \in \cP_0( \bbH ) $, $ x, y \in \bfP_J( \bfH_0 ) $ that
\begin{equation} \label{eq:weakrates,6}
\begin{split}
& \abs[\big]{ v^J ( 0 , x ) - v^J ( 0 , y ) }  = \abs[\big]{ \E \bigl[ \phi \bigl( X_T^{ J, x } \bigr) \bigr] - \E \bigl[ \phi \bigl( X_T^{ J, y } \bigr) \bigr] } \\
& \leq \abs{ \phi }_{ \mathrm{Lip}^{0} ( \bfH_0, \R ) } \norm[\big]{ X_T^{ J, x} - X_T^{ J, y} }_{ L^1 ( \P ; \bfH_0 ) } \\
& \leq \abs{ \phi }_{ \mathrm{Lip}^{0} ( \bfH_0, \R ) } \norm{ x - y }_{ \bfH_0 } \exp \bigl( T \bigl[ \abs{ \bfP_J \bfF }_{ \mathrm{Lip}^0( \bfH_0, \bfH_0 ) } + \tfrac{1}{2} \abs{ \bfP_J \bfB }_{ \mathrm{Lip}^0( \bfH_0, L_2( U, \bfH_0 ) ) }^2 \bigr] \bigr).
\end{split}
\end{equation}
Furthermore, it holds for all $x \in \bfH_{ 2(\gamma - \beta) } $, $ I, J \in \cP( \bbH ) $ with $ I \neq J $ that
\begin{equation} \label{eq:weakrates,7}
\begin{split}
\norm{ \bfP_I (x)  - \bfP_J (x) }_{ \bfH_0 } & \leq \norm{ \bfLambda^{ 2(\beta - \gamma) } \bfP_{ I \triangle J }  }_{ L(\bfH_0) } \norm{ \bfP_{ I \triangle J } (x) }_{ \bfH_{ 2(\gamma - \beta ) } } \\
& = \biggl[ \inf_{ h \in I \triangle J } \abs{ \lambda_h } \biggr]^{ \beta - \gamma } \norm{ \bfP_{ I \triangle J } (x) }_{ \bfH_{ 2(\gamma - \beta ) } } \leq \biggl[ \inf_{ h \in I \triangle J } \abs{ \lambda_h } \biggr]^{ \beta - \gamma } \norm{ x }_{ \bfH_{ 2(\gamma - \beta ) } }
\end{split}
\end{equation}
(cf., e.g., Proposition~8.1.4 in \cite{Jentzen2015}). Putting \eqref{eq:weakrates,6} and \eqref{eq:weakrates,7} together proves for all $ J \in \cP_0( \bbH ) $, $ I \in \cP( J ) \setminus \{ \bbH \} $ that
\begin{equation} \label{eq:weakrates,16}
\begin{split}
& \abs[\big]{ \E \bigl[ v^J \bigl( 0 , X_0^I \bigr) \bigr] - \E \bigl[ v^J \bigl( 0 , X_0^J \bigr) \bigr] } \leq \abs{ \phi }_{ \mathrm{Lip}^{0} ( \bfH_0, \R ) } \norm[\big]{ \bfP_I \bigl( X_0^J \bigr) - \bfP_J \bigl( X_0^J \bigr) }_{ L^1(\P ; \bfH_0 ) } \\
& \quad \cdot \exp \bigl( T \bigl[ \abs{ \bfP_J \bfF }_{ \mathrm{Lip}^0( \bfH_0, \bfH_0 ) } + \tfrac{1}{2} \abs{ \bfP_J \bfB }_{ \mathrm{Lip}^0( \bfH_0, L_2( U, \bfH_0 ) ) }^2 \bigr] \bigr) \\
& \leq \abs{ \phi }_{ \mathrm{Lip}^{0} ( \bfH_0, \R ) } \norm{ \xi }_{ L^1(\P ; \bfH_{ 2(\gamma - \beta ) } ) } \exp \bigl( T \bigl[ \abs{ \bfF }_{ \mathrm{Lip}^0( \bfH_0, \bfH_0 ) } + \tfrac{1}{2} \abs{ \bfB }_{ \mathrm{Lip}^0( \bfH_0, L_2( U, \bfH_0 ) ) }^2 \bigr] \bigr) \\
& \quad \cdot \biggl[ \inf_{ h \in \bbH \setminus I } \abs{ \lambda_h } \biggr]^{ \beta - \gamma }  < \infty.
\end{split}
\end{equation} 
Inequality \eqref{eq:weakrates,16} provides an estimate for the second summand on the right hand side of \eqref{eq:weakrates,3}. In a second step we establish an estimate for the fist summand on the right hand side of \eqref{eq:weakrates,3}. The chain rule and Lemma~\ref{lem:kolmogorov} show that for all $ J \in \cP_0( \bbH ) $, $ ( t , x ) \in [ 0 , T ] \times \bfP_J( \bfH_0 ) $ it holds that
\begin{equation} \label{eq:weakrates,8}
v_{ 1, 0 }^J ( t , x ) = - v_{ 0, 1 }^J ( t , x ) \bigl[ \bfA x + \bfP_J \bfF( x ) \bigr] - \tfrac{1}{2} \sum_{ u \in \mathbb{U} } v_{ 0, 2 }^J ( t , x ) ( \bfP_J \bfB(x) u , \bfP_J \bfB(x) u ).
\end{equation}
The standard It\^{o} formula and \eqref{eq:weakrates,8} prove for all $ J \in \cP_0( \bbH ) $, $ I \in \cP(J) $ that
\begin{equation} \label{eq:weakrates,19}
\begin{split}
& \E \bigl[ v^J \bigl( T, X_T^I \bigr) \bigr] - \E \bigl[ v^J \bigl( 0, X^I_0 \bigr) \bigr] = \int_{0}^T \E \bigl[ v_{ 1, 0 }^J \bigl( s, X_s^I \bigr) \bigr] \ds + \int_{0}^T \E \bigl[  v_{ 0, 1 }^J \bigl( s, X_s^I \bigr) \bfA X_s^I \bigr] \ds \\
& \quad + \int_{0}^T \E \bigl[ v_{ 0, 1 }^J \bigl( s, X_s^I \bigr) \bfP_I \bfF( X_s^I ) \bigr] \ds + \tfrac{1}{2} \sum_{ u \in \bbU } \int_0^T \E \bigl[ v_{ 0, 2 }^J \bigl( s, X_s^I \bigr) \bigl( \bfP_I \bfB( X^I_s )u , \bfP_I \bfB( X^I_s )u \bigr) \bigr] \ds \\
& =  \int_{0}^T \E \bigl[ v_{ 0, 1 }^J \bigl( s, X_s^I \bigr) \bfP_I \bfF( X_s^I ) \bigr] \ds - \int_{0}^T \E \bigl[ v_{ 0, 1 }^J \bigl( s, X_s^I \bigr) \bfP_J \bfF( X_s^I ) \bigr] \ds \\
& \quad + \frac{1}{2} \sum_{ u \in \bbU } \int_0^T \Bigl( \E \bigl[ v_{ 0, 2 }^J \bigl( s, X_s^I \bigr) \bigl( \bfP_I \bfB( X^I_s )u , \bfP_I \bfB( X^I_s )u \bigr) \bigr] \\
& \quad\qquad\qquad\quad\ \ - \E \bigl[ v_{ 0, 2 }^J \bigl( s, X_s^I \bigr) \bigl( \bfP_J \bfB( X^I_s )u , \bfP_J \bfB( X^I_s )u \bigr) \bigr] \Bigr) \ds.
\end{split}
\end{equation}
This shows for all $ J \in \cP_0( \bbH ) $, $ I \in \cP(J) $ that
\begin{equation} \label{eq:weakrates,9}
\begin{split}
& \abs[\big]{ \E \bigl[ v^J \bigl( T, X_T^I \bigr) \bigr] - \E \bigl[ v^J \bigl( 0, X^I_0 \bigr) \bigr] } \leq \int_{0}^T \abs[\big]{ \E \bigl[ v_{ 0 , 1 }^J \bigl( s, X_s^I \bigr) \bigl( \bfP_{I} \bfF( X_s^I ) - \bfP_J \bfF( X_s^I ) \bigr) \bigr] } \ds \\
& \quad + \abs[\bigg]{ \frac{1}{2} \sum_{ u \in \bbU } \int_0^T \E \bigl[ v_{ 0 , 2 }^J \bigl( s, X_s^I \bigr) \bigl( \bfP_I \bfB( X^I_s )u + \bfP_J \bfB( X^I_s )u, \bfP_I \bfB( X^I_s )u - \bfP_J \bfB( X^I_s )u \bigr) \bigr] \ds }.
\end{split}
\end{equation}
Inequality \eqref{eq:weakrates,7}, Lemma~\ref{lem:finiteness_kolmogorov}, and Lemma~\ref{lem:finiteness} thus prove for all $ J \in \cP_0( \bbH ) $, $ I \in \cP(J) \setminus \{ J \} $ that
\begin{equation} \label{eq:weakrates,14}
\begin{split}
& \int_0^T \abs[\big]{ \E \bigl[ v_{ 0 , 1 }^J \bigl( s, X_s^I \bigr) \bigl( \bfP_I \bfF(X_s^I) - \bfP_J \bfF(X_s^I) \bigr) \bigr] } \ds \\
& \leq \int_0^T  \E \bigl[ \abs[\big]{ v_{ 0 , 1 }^J \bigl( s, X_s^I \bigr) \bigl( \bfP_I \bfF(X_s^I) - \bfP_J \bfF(X_s^I) \bigr) } \bigr] \ds \\
& \leq \sup_{ t \in [ 0 , T ] } \abs[\big]{ u^J (t,\cdot) }_{ C_{\mathrm{b}}^1 ( \bfP_J (\bfH_0) , \R ) } \int_{0}^T \E \bigl[ \norm[\big]{ \bfP_I \bfF(X_s^I) - \bfP_J \bfF(X_s^I) }_{\bfH_0} \bigr] \ds \\
& \leq \biggl[ \sup_{K \in \cP_0( \bbH )} \sup_{ t \in [ 0 , T ] } \abs[\big]{ u^K (t, \cdot ) }_{ C_{\mathrm{b}}^1 ( \bfP_K( \bfH_0 ), \R ) } \biggr] \sup_{ K \in \cP_0( \bbH ) } \int_{0}^T \E \bigl[ \norm[\big]{ \bfF(X_s^K) }_{ \bfH_{ 2(\gamma-\beta) } } \bigr] \ds \\
& \quad \cdot \biggl[ \inf_{ h \in J \setminus I } \abs{ \lambda_h } \biggr]^{ \beta - \gamma } < \infty.
\end{split}
\end{equation}
This estimates the first summand on the right hand side of \eqref{eq:weakrates,9}. Next we consider the second summand on the right hand side of \eqref{eq:weakrates,9}. Observe for all $ J \in \cP_0( \bbH ) $, $ s \in [0,T] $, $ I \in \cP(J) $, $ \omega \in \Omega $ that $ v_{ 0 , 2 }^J \bigl( s, X_s^I (\omega) \bigr) \in L^{ (2) }( \bfP_J( \bfH_0), \R ) $. We define random variables $ R_{ I, J, s } \colon \Omega \to L( \bfP_J( \bfH_0 ) ) $, $ I \in \cP( J ) $, $ J \in \cP_0( \bbH ) $,  $ s \in [0,T] $, such that for all $ s \in [0,T] $, $ J \in \cP_0( \bbH ) $, $ I \in \cP( J )$, $ \omega \in \Omega $ it holds that
\begin{equation}
R_{ I, J, s } (\omega) = \cJ_{ v_{ 0 , 2 }^J( s, X_s^I(\omega) ) }^{ \bfP_J( \bfH_0 ) }.
\end{equation}
Then note that for all $ s \in [0,T] $, $ J \in \cP_0( \bbH ) $, $ y_1, y_2 \in \bfP_J( \bfH_0 ) $, $ I \in \cP( J ) $ it holds that $ v_{ 0 , 2 }^J \bigl(s, X_s^I \bigr) ( y_1, $ $  y_2 ) = \langle y_1 , R_{ I, J, s } y_2 \rangle_{ \bfH_0 } $. Therefore, the H\"{o}lder inequality for Schatten norms implies for all $ s \in [0,T] $, $ J \in \cP_0( \bbH ) $, $ I \in \cP( J ) $ that
\begin{equation} \label{eq:weakrates,10}
\begin{split}
& \abs[\bigg]{ \sum_{ u \in \bbU } \E \bigl[ v_{ 0 , 2 }^J \bigl( s, X_s^I \bigr) \bigl( ( \bfP_I + \bfP_J ) \bfB( X^I_s )u , ( \bfP_I - \bfP_J ) \bfB( X^I_s )u \bigr) \bigr] } \\
& = \abs[\bigg]{ \E \biggl[ \sum_{ u \in \bbU } \langle ( \bfP_I + \bfP_J ) \bfB( X^I_s )u , R_{ I, J, s } ( \bfP_I - \bfP_J ) \bfB( X^I_s )u \rangle_{ \bfH_0} \biggr] } \\
& = \abs[\big]{ \E \bigl[ \tr_{ U } ( \bfB( X^I_s )^{\star} ( \bfP_I + \bfP_J )^{ \star } R_{ I, J, s } ( \bfP_I - \bfP_J ) \bfB( X^I_s ) ) \bigr] } \\
& \leq \E \bigl[ \norm[\big]{ \bfB( X^I_s )^{\star} (\bfP_I + \bfP_J )^{ \star } R_{ I, J, s } ( \bfP_I - \bfP_J ) \bfB( X^I_s ) }_{ L_1 ( U ) } \bigr] \\
& \leq \E \Bigl[ \norm[\big]{ \bfB( X^I_s )^{\star} (\bfP_I + \bfP_J )^{\star} }_{ L_{ \nicefrac{ (2\beta) }{ \gamma } }(\bfH_0,U) } \norm{ R_{ I, J, s } }_{ L( \bfP_J(\bfH_0) ) } \norm[\big]{ (\bfP_I - \bfP_J ) \bfB( X^I_s ) }_{ L_{ \nicefrac{ (2\beta) }{ (2\beta-\gamma) } } (U,\bfH_0) } \Bigr].
\end{split} \raisetag{2.7cm}
\end{equation}
Moreover, observe for all $ s \in[0,T] $, $ J \in \cP_0( \bbH ) $, $ I \in \cP( J ) \setminus \{ J \} $ that
\begin{equation}
\begin{split}
& \norm[\big]{ \bfB( X^I_s )^{\star} (\bfP_I + \bfP_J )^{\star} }_{ L_{ \nicefrac{ (2\beta) }{ \gamma } }(\bfH_0,U) } = \norm[\big]{ \bfB( X^I_s )^{\star} \bfLambda^{ \gamma } \bfLambda^{ - \gamma } (\bfP_I + \bfP_J )^{\star} }_{ L_{ \nicefrac{ (2\beta) }{ \gamma } }(\bfH_0,U) } \\
& \leq \norm[\big]{ \bfB( X^I_s )^\star \bfLambda^{ \gamma } }_{ L( \bfH_0, U ) } \norm{ \bfLambda^{ - \gamma } }_{L_{ \nicefrac{ (2\beta) }{ \gamma } }( \bfH_0 )} \norm{ (\bfP_I + \bfP_J )^{\star} }_{ L(\bfH_0) } \\
& = \norm[\big]{ \bfB( X^I_s ) }_{ L(U, \bfH_\gamma) } \norm{ \bfLambda^{ -\beta } }_{ L_2(\bfH_0) }^{ \nicefrac{ \gamma }{ \beta } } \norm{ \bfP_I + \bfP_J }_{ L(\bfH_0) } \leq 2 \tp \norm[\big]{ \bfB( X^I_s ) }_{ L(U, \bfH_\gamma) } \norm{ \bfLambda^{ -\beta } }_{ L_2(\bfH_0) }^{ \nicefrac{ \gamma }{ \beta } } < \infty
\end{split}
\end{equation}
and
\begin{equation}
\begin{split}
&\norm[\big]{ (\bfP_I - \bfP_J ) \bfB( X^I_s ) }_{L_{ \nicefrac{ (2\beta) }{ (2\beta - \gamma) } } (U,\bfH_0)} \leq \norm{ \bfP_I - \bfP_J }_{L_{ \nicefrac{ (2\beta) }{ (2\beta - \gamma) } } ( \bfH_{\gamma},\bfH_0)} \norm[\big]{ \bfB( X^I_s ) }_{L(U,\bfH_{\gamma})} \\
& \leq \norm{(\bfP_I - \bfP_J ) \bfLambda^{ 2(\beta - \gamma) }}_{ L(\bfH_0) } \norm{ \bfLambda^{ 2(\gamma-\beta) } }_{L_{ \nicefrac{ (2\beta) }{ (2\beta - \gamma) } } ( \bfH_{\gamma},\bfH_0)} \norm[\big]{ \bfB( X^I_s ) }_{L(U, \bfH_{\gamma})} \\
& =  \biggl[ \inf_{ h \in J \setminus I } \abs{ \lambda_h } \biggr]^{ \beta - \gamma } \norm{ \bfLambda^{-\beta} }_{L_2 (\bfH_0)}^{ \nicefrac{ (2\beta - \gamma) }{\beta} } \norm[\big]{ \bfB( X^I_s ) }_{L(U, \bfH_{\gamma})}.
\end{split}
\end{equation}
In addition, Lemma~\ref{lem:finiteness_kolmogorov} implies for all $ s \in[0,T] $, $ J \in \cP_0( \bbH ) $, $ I \in \cP( J ) $ that
\begin{equation}  \label{eq:weakrates,11}
\norm{ R_{ I, J, s } }_{ L ( \bfP_J( \bfH_0 ) ) } \leq \sup_{ t \in [ 0 , T ] } \abs[\big]{ u^J( t, \cdot ) }_{ C_{\mathrm{b}}^2 ( \bfP_J( \bfH_0 ) , \R ) } < \infty.
\end{equation}
Inequalities \eqref{eq:weakrates,10}--\eqref{eq:weakrates,11}, Lemma~\ref{lem:finiteness}, and Lemma~\ref{lem:finiteness_kolmogorov} prove for all $ J \in \cP_0( \bbH ) $, $ I \in \cP( J ) \setminus \{ J \} $ that
\begin{equation}
\begin{split}
& \abs[\bigg]{ \frac{1}{2} \sum_{ u \in \bbU } \int_0^T \E \bigl[ v_{ 0 , 2 }^J \bigl( s,X_s^I \bigr) \bigl( \bfP_I \bfB( X^I_s )u + \bfP_J \bfB( X^I_s )u, \bfP_I \bfB( X^I_s )u - \bfP_J \bfB( X^I_s )u \bigr) \bigr] \ds } \\
& \leq \norm{ \bfLambda^{ -\beta } }_{ L_2(\bfH_0) }^2 \biggl[ \sup_{ K \in \cP_0( \bbH ) } \sup_{ t \in [ 0 , T ] } \abs[\big]{ u^K( t, \cdot ) }_{ C_{\mathrm{b}}^2 ( \bfP_K( \bfH_0 ), \R ) } \biggr] \sup_{ K \in \cP_0( \bbH ) } \int_{0}^{T} \E \bigl[ \norm[\big]{ \bfB( X_s^K ) }_{L(U, \bfH_{\gamma})}^2 \bigr] \ds \\
& \quad  \cdot \biggl[ \inf_{ h \in J \setminus I } \abs{ \lambda_h } \biggr]^{ \beta - \gamma } < \infty. \\
\end{split}
\end{equation}
Combining this with \eqref{eq:weakrates,9} and \eqref{eq:weakrates,14} ensures for all $ J \in \cP_0( \bbH ) $, $ I \in \cP( J ) \setminus \{ \bbH \} $ that
\begin{equation} \label{eq:weakrates,15}
\begin{split}
& \abs[\big]{ \E \bigl[ v^J \bigl( T, X_T^I \bigr) \bigr] - \E \bigl[ v^J \bigl( 0, X^I_0 \bigr) \bigr] } \\
& \leq \biggl( \biggl[ \sup_{K \in \cP_0( \bbH )} \sup_{ t \in [ 0 , T ] } \abs[\big]{ u^K (t, \cdot ) }_{ C_{\mathrm{b}}^1 ( \bfP_K( \bfH_0 ), \R ) } \biggr] \sup_{ K \in \cP_0( \bbH ) } \int_{0}^T \E \bigl[ \norm[\big]{ \bfF(X_s^K) }_{ \bfH_{ 2(\gamma-\beta) } } \bigr] \ds \\
&\quad + \norm{ \bfLambda^{ -\beta } }_{ L_2(\bfH_0) }^2 \biggl[ \sup_{ K \in \cP_0( \bbH ) } \sup_{ t \in [ 0 , T ] } \abs[\big]{ u^K( t, \cdot ) }_{ C_{\mathrm{b}}^2 ( \bfP_K( \bfH_0 ), \R ) } \biggr] \sup_{ K \in \cP_0( \bbH ) } \int_{0}^{T} \E \bigl[ \norm[\big]{ \bfB( X_s^K ) }_{L(U, \bfH_{\gamma})}^2 \bigr] \ds \biggr) \\
& \quad \cdot \biggl[ \inf_{ h \in \bbH \setminus I } \abs{ \lambda_h } \biggr]^{ \beta - \gamma } < \infty.
\end{split}
\end{equation}
This constitutes an estimate for the first summand on the right hand side of \eqref{eq:weakrates,3}. Inequalities \eqref{eq:weakrates,15}, \eqref{eq:weakrates,3}, and \eqref{eq:weakrates,16} show for all $ J \in \cP_0( \bbH ) $, $ I \in \cP( J ) \setminus \{ \bbH \} $ that
\begin{equation} \label{eq:weakrates,5}
\begin{split}
& \abs[\big]{ \E \bigl[ \phi \bigl(X_T^J \bigr)\bigr] - \E \bigl[\phi \bigl(X^I_T \bigr) \bigr] } \\
& \leq \biggl(  \abs{ \phi }_{ \mathrm{Lip}^{0} ( \bfH_0, \R ) } \norm{ \xi }_{ L^1(\P ; \bfH_{ 2(\gamma - \beta ) } ) } \exp \bigl( T \bigl[ \abs{ \bfF }_{ \mathrm{Lip}^0( \bfH_0, \bfH_0 ) } + \tfrac{1}{2} \abs{ \bfB }_{ \mathrm{Lip}^0( \bfH_0, L_2( U, \bfH_0 ) ) }^2 \bigr] \bigr) \\
&\quad + \biggl[ \sup_{K \in \cP_0( \bbH )} \sup_{ t \in [ 0 , T ] } \abs[\big]{ u^K (t, \cdot ) }_{ C_{\mathrm{b}}^1 ( \bfP_K( \bfH_0 ), \R ) } \biggr] \sup_{ K \in \cP_0( \bbH ) } \int_{0}^T \E \bigl[ \norm[\big]{ \bfF(X_s^K) }_{ \bfH_{ 2(\gamma-\beta) } } \bigr] \ds \\
&\quad + \norm{ \bfLambda^{ -\beta } }_{ L_2(\bfH_0) }^2 \biggl[ \sup_{ K \in \cP_0( \bbH ) } \sup_{ t \in [ 0 , T ] } \abs[\big]{ u^K( t, \cdot ) }_{ C_{\mathrm{b}}^2 ( \bfP_K( \bfH_0 ), \R ) } \biggr] \sup_{ K \in \cP_0( \bbH ) } \int_{0}^{T} \E \bigl[ \norm[\big]{ \bfB( X_s^K ) }_{L(U, \bfH_{\gamma})}^2 \bigr] \ds \biggr) \\
& \quad \cdot \biggl[ \inf_{ h \in \bbH \setminus I } \abs{ \lambda_h } \biggr]^{ \beta - \gamma } < \infty.
\end{split}
\end{equation}
In a third step Lemma~\ref{lem:perturbation}, Lemma~\ref{lem:length_preservation}, Minkowski's integral inequality, the Burkholder-Davis-Gundy inequality (see, e.g., Lemma~7.7 in Da~Prato \& Zabczyk~\cite{DaPratoZabczyk1992}), and the Cauchy-Schwarz inequality imply for all  $ J_0 \in \cP( \bbH ) $ and $ J_n \in \cP_0( \bbH ) $, $ n \in \N$, which satisfy $ \bigcup_{ n \in \N } J_n = J_0 $ and $ \forall \tp n \in \N \colon J_n \subseteq J_{n+1} $ and all $ n \in \N $ that
\begin{equation} \label{eq:weakrates,13}
\begin{split}
&\sup_{t \in [0,T]} \norm[\big]{ X_t^{J_0} - X_t^{J_n} }_{L^2( \P; \bfH_0 ) } \\
& \leq \sqrt{2}\tp \cE_1\bigl[ \sqrt{2}\tp T \abs{ \bfP_{ J_n } \bfF }_{ \mathrm{Lip}^0( \bfH_0, \bfH_0 ) } + \sqrt{2T} \abs{ \bfP_{ J_n} \bfB }_{ \mathrm{Lip}^0( \bfH_0, L_2( U, \bfH_0 ) ) } \bigr] \sup_{ t \in [0,T] } \norm[\big]{ \bfP_{ J_0 \setminus J_n } X_t^{J_0} }_{ L^2( \P; \bfH_0 ) } \\
& \leq \sqrt{2}\tp \cE_1\bigl[ \sqrt{2}\tp T \abs{ \bfF }_{ \mathrm{Lip}^0( \bfH_0, \bfH_0 ) } + \sqrt{2T} \abs{ \bfB }_{ \mathrm{Lip}^0( \bfH_0, L_2( U, \bfH_0 ) ) } \bigr] \\
& \quad \cdot \biggl( \norm{ \bfP_{ J_0 \setminus J_n } \xi }_{ L^2( \P; \bfH_0 ) } + \biggl[ T \int_{0}^{T} \norm[\big]{ \bfP_{ J_0 \setminus J_n } \bfF( X_s^{J_0} ) }_{ L^2( \P; \bfH_0 ) }^2 \ds \biggr]^{ \nicefrac{1}{2} } \\
&\qquad\ + \biggl[ \int_{0}^{T} \norm[\big]{ \bfP_{ J_0 \setminus J_n } \bfB( X_s^{J_0} ) }_{ L^2( \P; L_2( U, \bfH_0 ) ) }^2 \ds \biggr]^{ \nicefrac{1}{2} } \biggr).
\end{split}
\end{equation}
Therefore, Lebesgue's theorem of dominated convergence proves for all $ J_0 \in \cP( \bbH ) $ and $ J_n \in \cP_0( \bbH ) $, $ n \in \N$, which satisfy $ \bigcup_{ n \in \N } J_n = J_0 $ and $ \forall \tp n \in \N \colon J_n \subseteq J_{n+1} $ that 
\begin{equation} \label{eq:weakrates,17}
\lim_{ n \to \infty } \sup_{t \in [0,T]} \norm[\big]{ X_t^{J_0} - X_t^{J_n} }_{L^2( \P; \bfH_0 ) } = 0.
\end{equation} 
In a next step, \eqref{eq:weakrates,5} shows for all $ I \in \cP_0( \bbH ) \setminus \{ \bbH \} $ and $ J_n \in \cP_0( \bbH ) $, $ n \in \N$, which satisfy $ \bigcup_{n \in \N } J_n = \bbH $, $ I \subseteq J_1 $, and $ \forall \tp n \in \N \colon J_n \subseteq J_{n+1} $, and all $ n \in \N $ that
\begin{equation} \label{eq:weakrates,12}
\begin{split}
& \abs[\big]{ \E \bigl[ \phi \bigl( X_T^\bbH \bigr) \bigr] - \E \bigl[\phi \bigl(X^I_T \bigr) \bigr] } \leq \abs[\big]{ \E \bigl[ \phi \bigl(X_T^\bbH \bigl) \bigr] - \E \bigl[ \phi \bigl(X^{J_n}_T \bigr) \bigr] } + \abs[\big]{ \E \bigl[ \phi \bigl( X_T^{J_n} \bigr) \bigr] - \E \bigl[ \phi \bigl( X^{I}_T \bigr) \bigr] } \\
& \leq \abs{\phi}_{ \mathrm{Lip}^0 ( \bfH_0 , \bfH_0 ) } \norm[\big]{ X_T^\bbH - X_T^{J_n} }_{L^2( \P; \bfH_0 ) } \\
&\quad + \biggl(  \abs{ \phi }_{ \mathrm{Lip}^{0} ( \bfH_0, \R ) } \norm{ \xi }_{ L^1(\P ; \bfH_{ 2(\gamma - \beta ) } ) } \exp \bigl( T \bigl[ \abs{ \bfF }_{ \mathrm{Lip}^0( \bfH_0, \bfH_0 ) } + \tfrac{1}{2} \abs{ \bfB }_{ \mathrm{Lip}^0( \bfH_0, L_2( U, \bfH_0 ) ) }^2 \bigr] \bigr) \\
&\quad + \biggl[ \sup_{K \in \cP_0( \bbH )} \sup_{ t \in [ 0 , T ] } \abs[\big]{ u^K (t, \cdot ) }_{ C_{\mathrm{b}}^1 ( \bfP_K( \bfH_0 ), \R ) } \biggr] \sup_{ K \in \cP_0( \bbH ) } \int_{0}^T \E \bigl[ \norm[\big]{ \bfF(X_s^K) }_{ \bfH_{ 2(\gamma-\beta) } } \bigr] \ds \\
&\quad + \norm{ \bfLambda^{ -\beta } }_{ L_2(\bfH_0) }^2 \biggl[ \sup_{ K \in \cP_0( \bbH ) } \sup_{ t \in [ 0 , T ] } \abs[\big]{ u^K( t, \cdot ) }_{ C_{\mathrm{b}}^2 ( \bfP_K( \bfH_0 ), \R ) } \biggr] \sup_{ K \in \cP_0( \bbH ) } \int_{0}^{T} \E \bigl[ \norm[\big]{ \bfB( X_s^K ) }_{L(U, \bfH_{\gamma})}^2 \bigr] \ds \biggr) \\
& \quad \cdot \biggl[ \inf_{ h \in \bbH \setminus I } \abs{ \lambda_h } \biggr]^{ \beta - \gamma }.
\end{split}
\end{equation}
Letting $ n \to \infty $ in \eqref{eq:weakrates,12} and \eqref{eq:weakrates,17} complete the proof of Theorem~\ref{thm:weakrates} in the case that $ I \in \cP_0( \bbH ) \setminus \{ \bbH \} $. In a last step we prove the remaining cases. The estimate \eqref{eq:weakrates,12} ensures for all $ I_0 \in \cP( \bbH ) \setminus \{ \bbH \} $ and $ I_n \in \cP_0( I_0 ) $, $ n \in \N$, which satisfy $ \bigcup_{ n \in \N } I_n = I_0 $ and $ \forall \tp n \in \N \colon I_n \subseteq I_{n+1} $ and all $ n \in \N $ that
\begin{equation}
\begin{split}
& \abs[\big]{ \E \bigl[ \phi \bigl( X_T^\bbH \bigr) \bigr] - \E \bigl[\phi \bigl(X_T^{I_0} \bigr) \bigr] } \leq \abs[\big]{ \E \bigl[ \phi \bigl(X_T^\bbH \bigl) \bigr] - \E \bigl[ \phi \bigl(X^{I_n}_T \bigr) \bigr] } + \abs[\big]{ \E \bigl[ \phi \bigl( X_T^{I_0} \bigr) \bigr] - \E \bigl[ \phi \bigl( X^{I_n}_T \bigr) \bigr] } \\
& \leq \biggl(  \abs{ \phi }_{ \mathrm{Lip}^{0} ( \bfH_0, \R ) } \norm{ \xi }_{ L^1(\P ; \bfH_{ 2(\gamma - \beta ) } ) } \exp \bigl( T \bigl[ \abs{ \bfF }_{ \mathrm{Lip}^0( \bfH_0, \bfH_0 ) } + \tfrac{1}{2} \abs{ \bfB }_{ \mathrm{Lip}^0( \bfH_0, L_2( U, \bfH_0 ) ) }^2 \bigr] \bigr) \\
&\quad + \biggl[ \sup_{K \in \cP_0( \bbH )} \sup_{ t \in [ 0 , T ] } \abs[\big]{ u^K (t, \cdot ) }_{ C_{\mathrm{b}}^1 ( \bfP_K( \bfH_0 ), \R ) } \biggr] \sup_{ K \in \cP_0( \bbH ) } \int_{0}^T \E \bigl[ \norm[\big]{ \bfF(X_s^K) }_{ \bfH_{ 2(\gamma-\beta) } } \bigr] \ds \\
&\quad + \norm{ \bfLambda^{ -\beta } }_{ L_2(\bfH_0) }^2 \biggl[ \sup_{ K \in \cP_0( \bbH ) } \sup_{ t \in [ 0 , T ] } \abs[\big]{ u^K( t, \cdot ) }_{ C_{\mathrm{b}}^2 ( \bfP_K( \bfH_0 ), \R ) } \biggr] \sup_{ K \in \cP_0( \bbH ) } \int_{0}^{T} \E \bigl[ \norm[\big]{ \bfB( X_s^K ) }_{L(U, \bfH_{\gamma})}^2 \bigr] \ds \biggr) \\
& \quad \cdot \biggl[ \inf_{ h \in \bbH \setminus I } \abs{ \lambda_h } \biggr]^{ \beta - \gamma } + \abs{\phi}_{ \mathrm{Lip}^0 ( \bfH_0 , \bfH_0 ) } \norm[\big]{ X_T^{I_0} - X_T^{I_n} }_{L^2( \P; \bfH_0 ) }.
\end{split}
\end{equation}
Equation \eqref{eq:weakrates,17} and Lemma~\ref{lem:limsup} thus complete the proof of Theorem~\ref{thm:weakrates}.
\end{proof}

The next corollary is a direct consequence of Theorem~\ref{thm:weakrates} and Lemma~\ref{lem:finiteness}.

\begin{cor} \label{cor:weakrates}
Assume the setting in Section \ref{subsec:WeakSetting}, let $ X^I \colon [ 0 , T ] \times \Omega \to \bfP_{I} ( \bfH_\rho ) $, $ I \in \cP( \bbH ) $, and $ X^{ J , x } \colon [ 0 , T ] \times \Omega \to \bfP_J( \bfH_0 ) $, $ x \in \bfP_J( \bfH_0 ) $, $ J \in \cP_0( \bbH ) $, be $( \cF_t )_{ t \in [0,T] } $-predictable stochastic processes such that for all $ I \in \cP( \bbH ) $, $ J \in \cP_0( \bbH ) $, $ x \in \bfP_J( \bfH_0 ) $, $ t \in [ 0 , T ] $ it holds that $ \sup_{ s \in [ 0 , T ] } \bigl( \norm{ X_s^I }_{ L^2 ( \P ; \bfH_\rho ) } + \norm{ X_s^{ J , x } }_{ L^2 ( \P ; \bfH_{0} ) } \bigr) < \infty $ and $ \P $-a.s.\ that
\begin{align}
X_t^I & =  \e^{\bfA t} \bfP_{I} \xi + \int_0^t \e^{ \bfA(t-s) } \bfP_I \bfF(X_s^I) \ds  + \int_0^t \e^{ \bfA(t-s) } \bfP_I \bfB(X_s^I) \dWs,  \\
X_t^{ J, x } & = \e^{ \bfA t } x + \int_0^t \e^{ \bfA(t-s) } \bfP_J \bfF( X_s^{ J, x } ) \ds + \int_0^t \e^{ \bfA(t-s) } \bfP_J \bfB( X_s^{ J, x } ) \dWs,
\end{align}
let $ \phi \in C^{2}_\mathrm{b}( \bfH_0 , \R ) $, and let $ u^J \colon [0,T] \times \bfP_J( \bfH_0 ) \to \R $, $ J \in \cP_0( \bbH ) $, be the mappings which satisfy for all $ J \in \cP_0( \bbH ) $, $ (t,x) \in [0,T] \times \bfP_J( \bfH_0 ) $ that $ u^J(t,x) = \E\bigl[ \phi \bigl(X_t^{ J, x }\bigr) \bigr] $. Then it holds for all $ I \in \cP( \bbH ) \setminus \{ \bbH \} $ that
\begin{equation}
\begin{split}
& \abs[\big]{ \E \bigl[ \phi \bigl(X_T^\bbH \bigr)\bigr] - \E \bigl[\phi \bigl(X^I_T \bigr) \bigr] } \\
& \leq \biggl(  \abs{ \phi }_{ \mathrm{Lip}^{0} ( \bfH_0, \R ) } \norm{ \xi }_{ L^1(\P ; \bfH_{ 2(\gamma - \beta ) } ) } \exp \bigl( T \bigl[ \abs{ \bfF }_{ \mathrm{Lip}^0( \bfH_0, \bfH_0 ) } + \tfrac{1}{2} \abs{ \bfB }_{ \mathrm{Lip}^0( \bfH_0, L_2( U, \bfH_0 ) ) }^2 \bigr] \bigr) \\
&\quad + T \biggl[ \max_{ i \in \{ 1, 2 \} } \sup_{K \in \cP_0( \bbH )} \sup_{ t \in [ 0 , T ] } \abs[\big]{ u^K (t, \cdot ) }_{ C_{\mathrm{b}}^i ( \bfP_K( \bfH_0 ), \R ) } \biggr] \sup_{K \in \cP_0( \bbH )} \sup_{ t \in [ 0 , T ] } \bigl( 1 \vee \E \bigl[ \norm[\big]{ X_t^K }_{ \bfH_\rho }^2 \bigr] \bigr) \\
& \quad \cdot \Bigl[ \norm[\big]{ \bfF \vert_{ \bfH_\rho } }_{ \mathrm{Lip}^0( \bfH_\rho , \bfH_{ 2(\gamma - \beta ) } ) } + \norm{ \bfLambda^{ -\beta } }_{ L_2(\bfH_0) }^2 \norm[\big]{ \bfB \vert_{ \bfH_\rho } }_{ \mathrm{Lip}^0( \bfH_\rho , L( U , \bfH_{\gamma} ) ) }^2 \Bigr] \biggr) \biggl[ \inf_{ h \in \mathbb{H} \setminus I } \abs{\lambda_h} \biggr]^{ \beta - \gamma } < \infty.
\end{split}
\end{equation}
\end{cor}

The last result in this section, Corollary~\ref{cor:weak_rates_estimate} below, follows immediately from Corollary~\ref{cor:weakrates}, Lemma~\ref{lem:finiteness_kolmogorov}, and Lemma~\ref{lem:finiteness}.

\begin{cor} \label{cor:weak_rates_estimate}
Assume the setting in Section \ref{subsec:WeakSetting} and let $ X^I \colon [ 0 , T ] \times \Omega \to \bfP_{I} ( \bfH_\rho ) $, $ I \in \cP( \bbH ) $, be $( \cF_t )_{ t \in [0,T] } $-predictable stochastic processes such that for all $ I \in \cP( \bbH ) $, $ t \in [ 0 , T ] $ it holds that $ \sup_{ s \in [ 0 , T ] } \norm{ X_s^I }_{ L^2 ( \P ; \bfH_\rho ) } < \infty $ and $ \P $-a.s.\ that
\begin{equation}
X_t^I =  \e^{\bfA t} \bfP_{I} \xi + \int_0^t \e^{ \bfA(t-s) } \bfP_I \bfF(X_s^I) \ds  + \int_0^t \e^{ \bfA(t-s) } \bfP_I \bfB(X_s^I) \dWs.
\end{equation}
Then it holds for all $ \phi \in C^{2}_\mathrm{b}( \bfH_0 , \R ) $, $ I \in \cP( \bbH ) \setminus \{ \bbH \} $ that
\begin{equation}
\begin{split}
& \abs[\big]{ \E \bigl[ \phi \bigl(X_T^\bbH \bigr)\bigr] - \E \bigl[\phi \bigl(X^I_T \bigr) \bigr] } \\
& \leq \norm{ \phi }_{ C_{ \mathrm{b} }^2( \bfH_0, \R ) } ( 1 \vee T ) \bigl( 1 \vee \norm{ \xi }_{ L^2( \P; \bfH_\rho ) }^2 \bigr) \\
& \quad \cdot \Bigl( \norm{ \xi }_{ L^1(\P ; \bfH_{ 2(\gamma - \beta ) } ) } + \norm[\big]{ \bfF \vert_{ \bfH_\rho } }_{ \mathrm{Lip}^0( \bfH_\rho , \bfH_{ 2(\gamma - \beta ) } ) } + \norm{ \bfLambda^{ -\beta } }_{ L_2(\bfH_0) }^2 \norm[\big]{ \bfB \vert_{ \bfH_\rho } }_{ \mathrm{Lip}^0( \bfH_\rho , L( U , \bfH_{\gamma} ) ) }^2 \Bigr) \\
& \quad \cdot \Bigl( 1 \vee \bigl[ T \bigl( C_{ \bfF }^2 + 2 C_{ \bfB }^2 \bigr) \bigr]^{ \nicefrac{1}{2} } \Bigr) \exp \bigl( T \bigl[ \tfrac{1}{2} + 3 \abs{ \bfF }_{ \mathrm{Lip}^0( \bfH_0, \bfH_0 ) } + 4 \abs{ \bfB }_{ \mathrm{Lip}^0( \bfH_0, L_2( U, \bfH_0 ) ) }^2 \bigr] \bigr) \\
& \quad \cdot \exp \Bigl( T \Bigl[ 2 \norm[\big]{ \bfF \vert_{ \bfH_\rho } }_{ \mathrm{Lip}^0( \bfH_\rho, \bfH_\rho ) } + \norm[\big]{ \bfB \vert_{ \bfH_\rho } }_{ \mathrm{Lip}^0( \bfH_\rho, L_2( U, \bfH_\rho ) ) }^2 \Bigr] \Bigr) \biggl[ \inf_{ h \in \mathbb{H} \setminus I } \abs{\lambda_h} \biggr]^{ \beta - \gamma }  < \infty.
\end{split}
\end{equation}
\end{cor}

\subsection{Examples} \label{subsec:examples}

\subsubsection{Semilinear stochastic wave equations and the hyperbolic Anderson model}

The following elementary lemma is well-known (cf., e.g., Example~37.1 in Sell \& You~\cite{SellYou2002}).

\begin{lemma}\label{lem:interpolation_spaces}
Let $ \K \in \{ \R, \C \} $, let $ (H, \langle \cdot, \cdot \rangle_H, \norm{\cdot}_H ) $ be a $ \K $-Hilbert space,  let $ \bbH \subseteq H $ be an orthonormal basis of $ H $, let $ A \colon D(A) \subseteq H \to H $ be a symmetric diagonal linear operator with $ \inf( \sigma_{\mathrm{P}}(A) ) > 0 $, and let $ (H_r, \langle \cdot, \cdot \rangle_{H_r}, \norm{\cdot}_{H_r} ) $, $ r \in \R $, be a family of interpolation spaces associated to $ A $. Then
\begin{itemize}
\item[(i)]
for all $ v \in \bigcup_{ s \in \R } H_s $, $ r \in \R $ it holds that $ v \in H_r $ if and only if
\begin{equation} \label{eq:interpolation_spaces,1}
\sup_{ w \in \mathrm{span}_{ H_0 }( \bbH ) \setminus \{0\} } \frac{ \abs{ \langle w, v \rangle_{ H_0 } } }{ \norm{ w }_{ H_{-r} } } < \infty,
\end{equation}
\item[(ii)]
for all $ s \in \R $, $ v \in H_{-s} $, $ r \in [-s,\infty) $ it holds that $ v \in H_r $ if and only if
\begin{equation}
\sup_{ w \in H_s \setminus \{0\} } \frac{ \abs{ \langle w, v \rangle_{ H_0 } } }{ \norm{ w }_{ H_{-r} } } < \infty,
\end{equation}
\item[(iii)]
and for all $ r \in \R $, $ v \in H_r $, $ s \in [-r, \infty) $ it holds that
\begin{equation} \label{eq:interpolation_spaces,2}
\norm{ v }_{ H_r } = \sup_{ w \in \mathrm{span}_{ H_0 }( \bbH ) \setminus \{0\} } \frac{ \abs{ \langle w, v \rangle_{ H_0 } } }{ \norm{ w }_{ H_{-r} } } = \sup_{ w \in H_s \setminus \{0\} } \frac{ \abs{ \langle w, v \rangle_{ H_0 } } }{ \norm{ w }_{ H_{-r} } }.
\end{equation}
\end{itemize}
\end{lemma}

In the next result, Corollary~\ref{cor:wave_equation}, we illustrate Corollary~\ref{cor:weakrates} by a simple example. The proof of Corollary~\ref{cor:wave_equation} is elementary and well-known.

\begin{cor} \label{cor:wave_equation}
Let $ T, \vartheta \in ( 0 , \infty ) $, $ \gamma \in ( \nicefrac{1}{4}, \nicefrac{1}{2} ) $, $ \rho \in [0, \nicefrac{1}{2} ] $, $ r \in [ \nicefrac{1}{6}, \infty) $, let $ ( \Omega , \cF , \P ) $ be a probability space with a normal filtration $ ( \cF_t )_{ t \in [ 0 , T ] } $, let $ ( H, \langle \cdot, \cdot \rangle_H, \norm{\cdot}_H ) $ be the $ \R $-Hilbert space given by $ ( H, \langle \cdot, \cdot \rangle_H, \norm{\cdot}_H ) = \bigl( L^2( \lambda_{ (0,1) }; \R ), \langle \cdot, \cdot \rangle_{ L^2( \lambda_{ (0,1) }; \R ) }, \norm{ \cdot }_{ L^2( \lambda_{ (0,1) }; \R ) } \bigr) $, let $ ( W_t )_{ t \in [ 0 , T ] } $ be an $ \id_H $-cylindrical $ ( \cF_t )_{ t \in [ 0 , T ] } $-Wiener process, let $ \{ e_n \}_{ n \in \N } \subseteq H $ satisfy for all $ n \in \N $ and $ \lambda_{ (0,1) } $-a.e.\ $ x \in ( 0, 1 ) $ that $ e_n( x ) = \sqrt{2} \sin( n \pi x ) $, let $ A \colon D(A) \subseteq H \to H $ be the Laplacian with Dirichlet boundary conditions on $ H $ multiplied by $ \vartheta $, let $ ( H_s, \langle \cdot, \cdot \rangle_{ H_s }, \norm{ \cdot }_{ H_s } )$, $ s \in \R $, be a family of interpolation spaces associated to $ -A $, let $ \bfP_N \colon H_0 \times H_{ -\nicefrac{1}{2} } \to H_0 \times H_{ -\nicefrac{1}{2} } $, $ N \in \N \cup \{ \infty \} $, be the mappings which satisfy for all $ N \in \N \cup \{ \infty \} $, $ (v,w) \in H_0 \times H_{ -\nicefrac{1}{2} } $ that $ \bfP_N( v, w ) =  \sum_{ n = 1 }^N ( \langle e_n, v \rangle_H e_n, \langle \sqrt{\vartheta} \pi n e_n, w \rangle_{ H_{ -\nicefrac{1}{2} } } \sqrt{\vartheta} \pi n e_n ) $, let $ \bfA \colon D ( \bfA ) \subseteq H_0 \times H_{ -\nicefrac{1}{2} } \to H_0 \times H_{ -\nicefrac{1}{2} } $ be the linear operator such that $ D ( \bfA ) = H_{ \nicefrac{1}{2} } \times H_0 $ and such that for all $ (v,w) \in H_{ \nicefrac{1}{2} } \times H_0 $ it holds that $ \bfA( v , w ) = ( w , A v ) $, let $ \xi \in L^2 ( \P \vert_{\cF_0} ; H_{ \nicefrac{1}{2} } \times H_0 ) $, $ \phi \in C_{ \mathrm{b} }^2( H_0 \times H_{ -\nicefrac{1}{2} }, \R ) $, $ f \in \mathrm{Lip}^2( (0,1) \times \R, \R ) $, $ B \in \mathrm{Lip}^0( H_0 , L_2( H_0, H_{ -\nicefrac{1}{2} } ) ) $ satisfy that $ B \vert_{ H_\rho } \in \mathrm{Lip}^0( H_\rho, L_2( H_0, H_{ \rho - \nicefrac{1}{2} } ) \cap L( H_0, H_{ \gamma - \nicefrac{1}{2} } ) ) $, $ B \vert_{ H_r } \in C_{ \mathrm{b} }^2( H_r, L_2( H_0, H_{ -\nicefrac{1}{2} } ) ) $, and $ \sup_{ x \in H_r } $ $ \sup_{ v_1, v_2 \in H_r, \; \norm{ v_1 }_{ H_0 } \vee \norm{ v_2 }_{ H_0 } \leq 1 } \norm{ B''(x)( v_1, v_2 ) }_{ L_2( H_0, H_{ -\nicefrac{1}{2} } ) } < \infty $, and let $ \bfF \colon H_0 \times H_{ -\nicefrac{1}{2} } \to H_{ \nicefrac{1}{2} } \times H_0 $ and $ \bfB \colon H_0 \times H_{ -\nicefrac{1}{2} } \to L_2( H_0, H_0 \times H_{ - \nicefrac{1}{2} } ) $ be the mappings which satisfy for all $ (v,w) \in H_0 \times H_{ -\nicefrac{1}{2} } $ and $ \lambda_{ (0,1) } $-a.e.\ $ x \in ( 0, 1 ) $ that $ \bigl( \bfF(v,w) \bigr)( x ) = \bigl( 0, f( x, v(x) ) \bigr) $ and $ \bfB( v, w ) = \bigl( 0, B( v ) \bigr) $. Then
\begin{itemize}
\item[(i)]
it holds that $ \bfF \in \mathrm{Lip}^0( H_0 \times H_{ -\nicefrac{1}{2} }, H_{ \nicefrac{1}{2} } \times H_0 ) $, $ \bfF \vert_{ H_r \times H_{ r - \nicefrac{1}{2} } } \in \mathrm{Lip}^2( H_r \times H_{ r - \nicefrac{1}{2} }, H_{ \nicefrac{1}{2} } \times H_0 ) $, $ \bfB \in \mathrm{Lip}^0( H_0 \times H_{ -\nicefrac{1}{2} }, L_2( H_0, H_0 \times H_{ -\nicefrac{1}{2} } ) ) $, $ \bfB \vert_{ H_\rho \times H_{ \rho - \nicefrac{1}{2} } } \in \mathrm{Lip}^0( H_\rho \times H_{ \rho - \nicefrac{1}{2} }, L_2( H_0, H_\rho \times H_{ \rho - \nicefrac{1}{2} } ) \cap L( H_0, H_\gamma \times H_{ \gamma - \nicefrac{1}{2} } ) ) $, $ \bfB \vert_{ H_r \times H_{ r - \nicefrac{1}{2} } } \in C_{ \mathrm{b} }^2( H_r \times H_{ r - \nicefrac{1}{2} }, L_2( H_0, H_0 \times H_{ -\nicefrac{1}{2} } ) ) $, and
\begin{equation}
\forall \tp \delta \in ( - \infty, \nicefrac{1}{4} ) \colon \sup_{ \substack{ x \in H_r \times H_{ r - \nicefrac{1}{2} }, \\ v_1, v_2 \in H_r \times H_{ r - \nicefrac{1}{2} } \setminus \{0\} } } \tfrac{ \norm{ \bfF''(x)( v_1, v_2 ) }_{ H_\delta \times H_{ \delta -\nicefrac{1}{2} } } + \norm{ \bfB''(x)( v_1, v_2 ) }_{ L_2( H_0, H_0 \times H_{ -\nicefrac{1}{2} } ) } }{ \norm{v_1}_{ H_0 \times H_{ -\nicefrac{1}{2} } } \norm{v_2}_{ H_0 \times H_{ -\nicefrac{1}{2} } } }  < \infty,  \label{eq:wave_equation,9}
\end{equation}
\item[(ii)]
there exist up to modifications unique $ ( \cF_t )_{ t \in [0,T] } $-predictable stochastic processes $ X^N \colon [ 0 , T ] \times \Omega \to \bfP_N ( H_\rho \times H_{ \rho - \nicefrac{1}{2} } )$, $ N \in \N \cup \{ \infty \} $, which satisfy for all $ N \in \N \cup \{ \infty \} $, $ t \in [ 0 , T ] $ that $ \sup_{ s \in [ 0 , T ] } \norm{ X_s^N }_{ L^2 ( \P ; H_\rho \times H_{ \rho - \nicefrac{1}{2} } ) } < \infty $ and $ \P $-a.s.\ that
\begin{equation}
X_t^N =  \e^{\bfA t} \bfP_N \xi + \int_0^t \e^{ \bfA(t-s) } \bfP_N \bfF(X_s^N) \ds  + \int_0^t \e^{ \bfA(t-s) } \bfP_N \bfB(X_s^N) \dWs,
\end{equation}
\item[(iii)]
and for all $ \epsilon \in ( 4( \nicefrac{1}{2} -\gamma) , \infty ) $ there exists a real number $ C \in [0, \infty) $ such that for all $ N \in \N $ it holds that
\begin{equation}
\abs[\big]{ \E \bigl[ \phi \bigl(X_T^\infty \bigr)\bigr] - \E \bigl[\phi \bigl(X_T^N \bigr) \bigr] } \leq C \cdot N^{ \epsilon - 1 }.
\end{equation}
\end{itemize}
\end{cor}

\begin{proof}[Proof of Corollary~\ref{cor:wave_equation}]
Throughout this proof let $f_{ k, \ell } \colon (0,1) \times \R \to \R $, $ k, \ell \in \{ 0, 1, 2 \} $ with $ k + \ell \leq 2 $, be the mappings such that for all $ k, \ell \in \{ 0, 1, 2 \} $, $ ( x, y ) \in (0,1) \times \R $ with $ k + \ell \leq 2 $ it holds that $ f_{ k, \ell } (x,y) = \bigl( \frac{ \partial^{ k + \ell } }{ \partial x^k \partial y^\ell } f \bigr) (x,y) $ and let $ F \colon H_0 \to H_0 $ be the mapping such that for all $ v \in H_0 $ and $ \lambda_{ (0,1) } $-a.e.\ $ x \in ( 0, 1 ) $ it holds that $ \bigl( F( v ) \bigr) (x) =  f( x, v(x) ) $. Then note for all $ u, v \in H_0 $, $ w \in H_{ -\nicefrac{1}{2} } $ that $ \bigl( \bfF(v,w) \bigr)( x ) = \bigl( 0, F(v) \bigr) $ and that
\begin{equation}
\norm{ F(u) - F(v) }_{ H_0 } = \biggl( \int_{0}^{1} \abs{ f(x,u(x)) - f(x,v(x)) }^2 \dx \biggr)^{ \nicefrac{1}{2} } \leq \abs{ f }_{ \mathrm{Lip}^0( (0,1) \times \R, \R ) } \norm{ u - v }_{ H_0 },
\end{equation}
which proves that $ F \in \mathrm{Lip}^0( H_0, H_0 ) $ and hence that $ \bfF \in \mathrm{Lip}^0( H_0 \times H_{ -\nicefrac{1}{2} }, H_{ \nicefrac{1}{2} } \times H_0 ) $. Next observe that the Sobolev Embedding Theorem ensures for all $ \delta \in [1, 6] $ that
\begin{equation} \label{eq:wave_equation,5}
\sup_{ w \in H_r \setminus \{ 0 \} } \frac{ \norm{ w }_{ L^\delta( \lambda_{ (0,1) }; \R ) } }{ \norm{ w }_{ H_r } } < \infty.
\end{equation}
Moreover, it holds for all $ v, h \in H_0 $ and $ \lambda_{ (0,1) } $-a.e.\ $ x \in ( 0, 1 ) $ that
\begin{equation}
\begin{split}
& \abs{ f( x, v(x) + h(x) ) - f( x, v(x) ) -  f_{0,1}( x, v(x) ) h(x) } \\
& = \abs[\bigg]{ \int_{0}^{1} \bigl[ f_{ 0, 1 }( x, v(x) + y h(x) ) - f_{ 0, 1 }( x, v(x) ) \bigr] h(x) \ud y } \leq \abs{ f }_{ \mathrm{Lip}^1( (0,1) \times \R, \R ) } \abs{ h(x) }^2.
\end{split}
\end{equation}
This, H\"{o}lder's inequality, and \eqref{eq:wave_equation,5} imply for all $ v \in H_r $, $ h \in H_r \setminus \{ 0 \} $ that
\begin{equation} \label{eq:wave_equation,1}
\begin{split}
& \frac{1}{ \norm{h}_{ H_r } }\biggl( \int_{0}^{1} \abs{ f( x, v(x) + h(x) ) - f( x, v(x) ) -  f_{0,1}( x, v(x) ) h(x) }^2 \dx \biggr)^{ \nicefrac{1}{2} } \\
& \leq  \abs{ f }_{ \mathrm{Lip}^1( (0,1) \times \R, \R ) } \frac{\norm{ h }_{ L^4( \lambda_{ (0,1) } ; \R ) }^2}{ \norm{h}_{ H_r } } \leq \abs{ f }_{ \mathrm{Lip}^1( (0,1) \times \R, \R ) } \biggl( \sup_{ w \in H_r \setminus \{ 0 \} } \frac{ \norm{ w }_{ L^4( \lambda_{ (0,1) } ; \R ) } }{ \norm{ w }_{ H_r } }\biggr)^2 \norm{ h }_{ H_r } < \infty.
\end{split}
\end{equation}
In addition, it holds for all $ v, h \in H_r $ that
\begin{equation} \label{eq:wave_equation,2}
\begin{split}
\biggl( \int_{0}^{1} \abs{ f_{ 0, 1 }( x, v(x) ) h(x) }^2 \dx \biggr)^{ \nicefrac{1}{2} } & \leq \abs{ f }_{ C_{ \mathrm{b} }^1( (0,1) \times \R, \R ) } \norm{ h }_{ H_0 } \leq \abs{ f }_{ C_{ \mathrm{b} }^1( (0,1) \times \R, \R ) } \norm{ h }_{ H_r } \\
& = \abs{ f }_{ \mathrm{Lip}^0( (0,1) \times \R, \R ) } \norm{ h }_{ H_r }  < \infty.
\end{split}
\end{equation}
Inequalities \eqref{eq:wave_equation,1} and \eqref{eq:wave_equation,2} prove that $ F \vert_{ H_r } \colon H_r \to H_0 $ is Fr\'{e}chet differentiable, that for all $ v, h \in H_r $ and $ \lambda_{ (0,1) } $-a.e.\ $ x \in ( 0, 1 ) $ it holds that
\begin{equation} \label{eq:wave_equation,7}
\bigl( F'(v) h \bigr) (x) = f_{ 0, 1 }( x, v(x) ) h(x),
\end{equation}
and that $ \sup_{ v \in H_r } \norm{ F'(v) }_{ L( H_r, H_0 ) } \leq \abs{ f }_{ C_{ \mathrm{b} }^1( (0,1) \times \R, \R ) } < \infty $. Furthermore, H\"{o}lder's inequality and \eqref{eq:wave_equation,5} show for all $ u, v, h \in H_r $ that
\begin{equation}
\begin{split}
\norm{ ( F'(u) - F'(v) ) h }_{ H_0 } & = \biggl( \int_{0}^{1} \abs[\big]{ \bigl[  f_{ 0, 1 }( x, u(x) ) -  f_{ 0, 1 }( x, v(x) ) \bigr] h(x) }^2 \dx \biggr)^{ \nicefrac{1}{2} } \\
& \leq \abs{ f }_{ \mathrm{Lip}^1( (0,1) \times \R, \R ) } \norm{ u - v }_{ L^4( \lambda_{ (0,1) } ; \R ) } \norm{ h }_{ L^4( \lambda_{ (0,1) } ; \R ) } \\
& \leq \abs{ f }_{ \mathrm{Lip}^1( (0,1) \times \R, \R ) } \biggl( \sup_{ w \in H_r \setminus \{ 0 \} } \frac{ \norm{ w }_{ L^4( \lambda_{ (0,1) } ; \R ) } }{ \norm{ w }_{ H_r } }\biggr)^2 \norm{ u - v }_{ H_r } \norm{ h }_{ H_r } < \infty,
\end{split}
\end{equation}
which ensures that $ F \vert_{ H_r } \in \mathrm{Lip}^1( H_r, H_0 ) $. Similarly, observe for all $ v, h, g \in H_0 $ and $ \lambda_{ (0,1) } $-a.e.\ $ x \in ( 0, 1 ) $ that
\begin{equation}
\begin{split}
& \abs{ f_{ 0, 1 }( x, v(x) + g(x) ) h(x) - f_{ 0, 1 }( x, v(x) ) h(x) -  f_{0,2}( x, v(x) ) h(x) g(x) } \\
& = \abs[\bigg]{ \int_{0}^{1} \bigl[ f_{ 0, 2 }( x, v(x) + y g(x) ) - f_{ 0, 2 }( x, v(x) ) \bigr] h(x) g(x) \ud y } \leq \abs{ f }_{ \mathrm{Lip}^2( (0,1) \times \R, \R ) } \abs{ h(x) } \abs{ g(x) }^2.
\end{split}
\end{equation}
This, H\"{o}lder's inequality, and \eqref{eq:wave_equation,5} establish for all $ v, h \in H_r $, $ g \in H_r \setminus \{ 0 \} $ that
\begin{equation} \label{eq:wave_equation,3}
\begin{split}
& \frac{1}{ \norm{g}_{ H_r } } \biggl( \int_{0}^{1} \abs{ f_{ 0, 1 }( x, v(x) + g(x) ) h(x) - f_{ 0, 1 }( x, v(x) ) h(x) -  f_{0,2}( x, v(x) ) h(x) g(x) }^2 \dx \biggr)^{ \nicefrac{1}{2} } \\
& \leq \frac{ \abs{ f }_{ \mathrm{Lip}^2( (0,1) \times \R, \R ) } }{ \norm{g}_{ H_r } } \biggl( \int_{0}^{1} \abs{ h(x) }^2 \abs{ g(x) }^4 \dx \biggr)^{ \nicefrac{1}{2} } \\
& \leq \abs{ f }_{ \mathrm{Lip}^2( (0,1) \times \R, \R ) } \frac{ \norm{ h }_{ L^6( \lambda_{ (0,1) } ; \R ) } \norm{ g }_{ L^6( \lambda_{ (0,1) } ; \R ) }^2 }{ \norm{g}_{ H_r } } \\
& \leq \abs{ f }_{ \mathrm{Lip}^2( (0,1) \times \R, \R ) } \biggl( \sup_{ w \in H_r \setminus \{ 0 \} } \frac{ \norm{ w }_{ L^6( \lambda_{ (0,1) } ; \R ) } }{ \norm{ w }_{ H_r } } \biggr)^3 \norm{ h }_{ H_r } \norm{ g }_{ H_r } < \infty.
\end{split}
\end{equation}
Furthermore, H\"{o}lder's inequality and \eqref{eq:wave_equation,5} also prove for all $ v, h, g \in H_r $ that
\begin{equation} \label{eq:wave_equation,4}
\begin{split}
& \biggl( \int_{0}^{1} \abs{ f_{0,2}( x, v(x) ) h(x) g(x) }^2 \dx \biggr)^{ \nicefrac{1}{2} } \leq \abs{ f }_{ C_{ \mathrm{b} }^2( (0,1) \times \R, \R ) } \norm{ h }_{ L^4( \lambda_{ (0,1) } ; \R ) } \norm{ g }_{ L^4( \lambda_{ (0,1) } ; \R ) } \\
& \leq \abs{ f }_{ C_{ \mathrm{b} }^2( (0,1) \times \R, \R ) } \biggl( \sup_{ w \in H_r \setminus \{ 0 \} } \frac{ \norm{ w }_{ L^4( \lambda_{ (0,1) } ; \R ) } }{ \norm{ w }_{ H_r } } \biggr)^2 \norm{ h }_{ H_r } \norm{ g }_{ H_r } \\
& = \abs{ f }_{ \mathrm{Lip}^1( (0,1) \times \R, \R ) } \biggl( \sup_{ w \in H_r \setminus \{ 0 \} } \frac{ \norm{ w }_{ L^4( \lambda_{ (0,1) } ; \R ) } }{ \norm{ w }_{ H_r } } \biggr)^2 \norm{ h }_{ H_r } \norm{ g }_{ H_r }  < \infty.
\end{split}
\end{equation}
Combining \eqref{eq:wave_equation,3} and \eqref{eq:wave_equation,4} ensures that $ F \vert_{ H_r } \colon H_r \to H_0 $ is twice Fr\'{e}chet differentiable, that for all $ v, h, g \in H_r $ and $ \lambda_{ (0,1) } $-a.e.\ $ x \in ( 0, 1 ) $ it holds that
\begin{equation}
\bigl( F''(v)( h, g ) \bigr) (x) = f_{ 0, 2 }( x, v(x) ) h(x) g(x),
\end{equation}
and that 
\begin{equation}
\sup_{ v \in H_r } \norm{ F''(v) }_{ L^{(2)}( H_r, H_0 ) } \leq \abs{ f }_{ C_{ \mathrm{b} }^2( (0,1) \times \R, \R ) } \biggl( \sup_{ w \in H_r \setminus \{ 0 \} } \frac{ \norm{ w }_{ L^4( \lambda_{ (0,1) } ; \R ) } }{ \norm{ w }_{ H_r } } \biggr)^2 < \infty.
\end{equation}
In addition, H\"{o}lder's inequality and \eqref{eq:wave_equation,5} establish for all $ u, v, h, g \in H_r $ that
\begin{equation}
\begin{split}
& \norm{ ( F''(u) - F''(v) ) ( h, g ) }_{ H_0 } = \biggl( \int_{0}^{1} \abs[\big]{ \bigl[  f_{ 0, 2 }( x, u(x) ) -  f_{ 0, 1 }( x, v(x) ) \bigr] h(x) g(x) }^2 \dx \biggr)^{ \nicefrac{1}{2} } \\
& \leq \abs{ f }_{ \mathrm{Lip}^2( (0,1) \times \R, \R ) } \norm{ u - v }_{ L^6( \lambda_{ (0,1) } ; \R ) } \norm{ h }_{ L^6( \lambda_{ (0,1) } ; \R ) } \norm{ g }_{ L^6( \lambda_{ (0,1) } ; \R ) } \\
& \leq \abs{ f }_{ \mathrm{Lip}^2( (0,1) \times \R, \R ) } \biggl( \sup_{ w \in H_r \setminus \{ 0 \} } \frac{ \norm{ w }_{ L^6( \lambda_{ (0,1) } ; \R ) } }{ \norm{ w }_{ H_r } }\biggr)^3 \norm{ u - v }_{ H_r } \norm{ h }_{ H_r } \norm{ g }_{ H_r } < \infty.
\end{split}
\end{equation}
This shows that $ F \vert_{ H_r } \in \mathrm{Lip}^2( H_r, H_0 ) $ and hence that \smash{$ \bfF \vert_{ H_r \times H_{ r - \nicefrac{1}{2} } } \in \mathrm{Lip}^2( H_r \times H_{ r - \nicefrac{1}{2} }, H_{ \nicefrac{1}{2} } \times H_0 ) $}. Next, note that the assumptions that $ B \in \mathrm{Lip}^0( H_0 , L_2( H_0, H_{ -\nicefrac{1}{2} } ) ) $, $ B \vert_{ H_\rho } \in \mathrm{Lip}^0( H_\rho, L_2( H_0, H_{ \rho - \nicefrac{1}{2} } ) $ $ \cap L( H_0, H_{ \gamma - \nicefrac{1}{2} } ) ) $, and $ B \vert_{ H_r } \in C_{ \mathrm{b} }^2( H_r, L_2( H_0, H_{ -\nicefrac{1}{2} } ) ) $ ensure that $ \bfB \in \mathrm{Lip}^0( H_0 \times H_{ -\nicefrac{1}{2} }, L_2( H_0, H_0 \times H_{ -\nicefrac{1}{2} } ) ) $, $ \bfB \vert_{ H_\rho \times H_{ \rho - \nicefrac{1}{2} } } \in \mathrm{Lip}^0( H_\rho \times H_{ \rho - \nicefrac{1}{2} }, L_2( H_0, H_\rho \times H_{ \rho - \nicefrac{1}{2} } ) \cap L( H_0, H_\gamma \times H_{ \gamma - \nicefrac{1}{2} } ) ) $, and $ \bfB \vert_{ H_r \times H_{ r - \nicefrac{1}{2} } } \in C_{ \mathrm{b} }^2( H_r \times H_{ r - \nicefrac{1}{2} }, L_2( H_0, H_0 \times H_{ -\nicefrac{1}{2} } ) ) $. In addition, Lemma~\ref{lem:interpolation_spaces} proves for all $ \delta \in ( -\infty, \nicefrac{1}{4} ) $ that
\begin{equation}
\begin{split}
\norm{ F''(v)( h, g) }_{ H_{ \delta - \nicefrac{1}{2} } } & = \sup_{ w \in H_{ \nicefrac{1}{2} - \delta } \setminus \{ 0 \} } \frac{ \langle w, F''(v)( h, g ) \rangle_{ H_0 } }{ \norm{ w }_{ H_{ \nicefrac{1}{2} - \delta } } } \\
& \leq \abs{ f }_{ C_{ \mathrm{b} }^2( (0,1) \times \R, \R ) } \biggl( \sup_{ w \in H_{ \nicefrac{1}{2} - \delta } \setminus \{ 0 \} } \frac{ \norm{ w }_{ L^\infty( \lambda_{ (0,1) }; \R ) } }{ \norm{ w }_{ H_{ \nicefrac{1}{2} - \delta } } } \biggr) \norm{ h }_{ H_0 } \norm{ g }_{ H_0 } < \infty.
\end{split}
\end{equation}
This and the assumption that $ \sup_{ x, v_1, v_2 \in H_r, \; \norm{ v_1 }_{ H_0 } \vee \norm{ v_2 }_{ H_0 } \leq 1 } \norm{ B''(x)( v_1, v_2 ) }_{ L_2( H_0, H_{ -\nicefrac{1}{2} } ) } < \infty $ show \eqref{eq:wave_equation,9} and thus complete the proof of (i). Furthermore, (ii) follows directly from (i) and Remark~\ref{rmk:existence_regularity}. In the remainder of this proof (iii) is established. Let $ \epsilon \in ( 4( \nicefrac{1}{2} -\gamma) , 1 ] $, $ \beta \in ( \nicefrac{1}{2}, 2 \gamma ] $ and $ \lambda_n \in \R $, $ n \in \N $, be real numbers which satisfy for all $ n \in \N $ that $ \beta = \nicefrac{1}{2} + \nicefrac{ ( \epsilon - 4( \nicefrac{1}{2} - \gamma ) ) }{2} $ and $ \lambda_n = - \vartheta \pi^2 n^2 $. In addition, let $ \bfLambda \colon D(\bfLambda) \subseteq H_0 \times H_{ -\nicefrac{1}{2} } \to H_0 \times H_{ -\nicefrac{1}{2} } $ be the linear operator such that $ D(\bfLambda) = H_{ \nicefrac{1}{2} } \times H_0 $ and such that for all $ (v,w) \in H_{ \nicefrac{1}{2} } \times H_0 $ it holds that
\begin{equation}
\bfLambda (v,w) = \twovector{ \sum_{ n=1 }^\infty  \abs{\lambda_n}^{ \nicefrac{1}{2} } \langle e_n , v \rangle_{ H_0 } e_n }{ \sum_{ n = 1 }^\infty  \abs{\lambda_n}^{ \nicefrac{1}{2} } \bigl\langle \abs{ \lambda_n }^{\nicefrac{1}{2}} e_n, w \bigr\rangle_{ H_{ -\nicefrac{1}{2} } } \abs{ \lambda_n }^{\nicefrac{1}{2}} e_n }.
\end{equation}
Then note for all $ v \in H_1 $ that $ A v = \sum_{n=1}^{\infty} \lambda_n \langle e_n, v \rangle_{ H_0 } e_n $ and that \smash{$ \norm{ \bfLambda^{ - \beta } }_{ L_2( H_0 \times H_{ - \nicefrac{1}{2} } ) } < \infty $}. Furthermore, observe that (i) and the fact that $ 2 \gamma - \beta = \nicefrac{ (1 - \epsilon) }{2} $ imply that $ \bfF \in \mathrm{Lip}^0( H_0 \times H_{ -\nicefrac{1}{2} }, H_{ 2\gamma - \beta } \times H_{ 2\gamma - \beta - \nicefrac{1}{2} } ) $. This and again (i) enable us to apply Corollary~\ref{cor:weakrates} to obtain that there exists a real number $ C \in [0, \infty ) $ such that for all $ N \in \N $ it holds that
\begin{equation}
\abs[\big]{ \E \bigl[ \phi \bigl(X_T^\infty \bigr)\bigr] - \E \bigl[\phi \bigl(X_T^N \bigr) \bigr] } \leq C \abs{ \lambda_{ N+1 } }^{ \beta - 2 \gamma } \leq \vartheta^{ \nicefrac{ ( \epsilon - 1 ) }{2} } \cdot C \cdot N^{ \epsilon - 1 }.
\end{equation}
The proof of Corollary~\ref{cor:wave_equation} is thus completed.
\end{proof}

In the proof of Corollary~\ref{cor:Anderson_model} below we use the following elementary and well-known result, Lemma~\ref{lem:linear_operator}.

\begin{lemma}\label{lem:linear_operator}
Let $ \K \in \{ \R, \C \} $, let $ (H, \langle \cdot, \cdot \rangle_H, \norm{\cdot}_H ) $ be a $ \K $-Hilbert space,  let $ \bbH \subseteq H $ be an orthonormal basis of $ H $, let $ A \colon D(A) \subseteq H \to H $ be a symmetric diagonal linear operator with $ \inf( \sigma_{\mathrm{P}}(A) ) > 0 $, let $ (H_r, \langle \cdot, \cdot \rangle_{H_r}, \norm{\cdot}_{H_r} ) $, $ r \in \R $, be a family of interpolation spaces associated to $ A $, and let $ q \in \R $, $ p \in [q, \infty) $, $ s \in \R $, $ r \in [s, \infty) $. Then 
\begin{itemize}
\item[(i)]
for all $ B \in L(H_q, H_s ) $ it holds that $ B \in L( H_q, H_r ) $ if and only if
\begin{equation} \label{eq:linear_operator,3}
\biggl( B \bigl( \mathrm{span}_{ H_0 }( \bbH ) \bigr) \subseteq H_r \quad \text{and} \quad \sup_{ w \in \mathrm{span}_{ H_0 }( \bbH ) \setminus \{0\} } \frac{ \norm{ B w }_{ H_r } }{\norm{w}_{ H_q } } < \infty \biggr),
\end{equation}
\item[(ii)]
for all $ B \in L(H_q, H_s ) $ it holds that $ B \in L( H_q, H_r ) $ if and only if
\begin{equation}
\biggl(  B( H_p ) \subseteq H_r \quad \text{and} \quad \sup_{ w \in H_p \setminus \{0\} } \frac{ \norm{ B w }_{ H_r } }{\norm{w}_{ H_q } } < \infty \biggr),
\end{equation}
\item[(iii)]
and for all $ B \in L( H_q, H_r ) $ it holds that
\begin{equation} \label{eq:linear_operator,1}
\norm{ B }_{ L( H_q, H_r ) } = \sup_{ w \in \mathrm{span}_{ H_0 }( \bbH ) \setminus \{0\} } \frac{ \norm{ B w }_{ H_r } }{\norm{w}_{ H_q } } = \sup_{ w \in H_p \setminus \{0\} } \frac{ \norm{ B w }_{ H_r } }{\norm{w}_{ H_q } }.
\end{equation}
\end{itemize}
\end{lemma}

\begin{cor}[Hyperbolic Anderson model] \label{cor:Anderson_model}
Let $ T, \vartheta \in ( 0 , \infty ) $, $ \alpha, \beta \in \R $, let $ ( \Omega , \cF , \P ) $ be a probability space with a normal filtration $ ( \cF_t )_{ t \in [ 0 , T ] } $, let $ ( H, \langle \cdot, \cdot \rangle_H, \norm{\cdot}_H ) $ be the $ \R $-Hilbert space given by $ ( H, \langle \cdot, \cdot \rangle_H, \norm{\cdot}_H ) = \bigl( L^2( \lambda_{ (0,1) }; \R ), \langle \cdot, \cdot \rangle_{ L^2( \lambda_{ (0,1) }; \R ) }, \norm{ \cdot }_{ L^2( \lambda_{ (0,1) }; \R ) } \bigr) $, let $ ( W_t )_{ t \in [ 0 , T ] } $ be an $ \id_H $-cylindrical $ ( \cF_t )_{ t \in [ 0 , T ] } $-Wiener process, let $ \{ e_n \}_{ n \in \N } \subseteq H $ satisfy for all $ n \in \N $ and $ \lambda_{ (0,1) } $-a.e.\ $ x \in ( 0, 1 ) $ that $ e_n( x ) = \sqrt{2} \sin( n \pi x ) $, let $ A \colon D(A) \subseteq H \to H $ be the Laplacian with Dirichlet boundary conditions on $ H $ multiplied by $ \vartheta $, let $ ( H_r, \langle \cdot, \cdot \rangle_{ H_r }, \norm{ \cdot }_{ H_r } )$, $ r \in \R $, be a family of interpolation spaces associated to $ -A $, let $ \bfP_N \colon H_0 \times H_{ -\nicefrac{1}{2} } \to H_0 \times H_{ -\nicefrac{1}{2} } $, $ N \in \N \cup \{ \infty \} $, be the mappings which satisfy for all $ N \in \N \cup \{ \infty \} $, $ (v,w) \in H_0 \times H_{ -\nicefrac{1}{2} } $ that $ \bfP_N( v, w ) =  \sum_{ n = 1 }^N ( \langle e_n, v \rangle_H e_n, \langle \sqrt{\vartheta} \pi n e_n, w \rangle_{ H_{ -\nicefrac{1}{2} } } \sqrt{\vartheta} \pi n e_n ) $, let $ \bfA \colon D ( \bfA ) \subseteq H_0 \times H_{ -\nicefrac{1}{2} } \to H_0 \times H_{ -\nicefrac{1}{2} } $ be the linear operator such that $ D ( \bfA ) = H_{ \nicefrac{1}{2} } \times H_0 $ and such that for all $ (v,w) \in H_{ \nicefrac{1}{2} } \times H_0 $ it holds that $ \bfA( v , w ) = ( w , A v ) $, let $ \xi \in L^2 ( \P \vert_{\cF_0} ; H_{ \nicefrac{1}{2} } \times H_0 ) $, $ \phi \in C_{ \mathrm{b} }^2( H_0 \times H_{ -\nicefrac{1}{2} }, \R ) $, $ f \in \mathrm{Lip}^2( (0,1) \times \R, \R ) $, and let $ \bfF \colon H_0 \times H_{ -\nicefrac{1}{2} } \to H_{ \nicefrac{1}{2} } \times H_0 $ and $ \bfB \colon H_0 \times H_{ -\nicefrac{1}{2} } \to L_2( H_0, H_0 \times H_{ -\nicefrac{1}{2} } ) $ be the mappings which satisfy for all $ (v,w) \in H_0 \times H_{ -\nicefrac{1}{2} } $, $ u \in H_1 $ and $ \lambda_{ (0,1) } $-a.e.\ $ x \in ( 0, 1 ) $ that $ \bigl( \bfF(v,w) \bigr)( x ) = \bigl( 0, f( x, v(x) ) \bigr) $ and $ \bigl( \bfB( v, w ) u \bigr) (x) = \bigl( 0, ( \alpha + \beta v(x) ) u(x) \bigr) $. Then
\begin{itemize}
\item[(i)]
there exist up to modifications unique $ ( \cF_t )_{ t \in [0,T] } $-predictable stochastic processes $ X^N \colon [ 0 , T ] \times \Omega \to \bigcap_{ \rho \in [0, \nicefrac{1}{4} ) } \bfP_N ( H_\rho \times H_{ \rho -\nicefrac{1}{2} } )$, $ N \in \N \cup \{ \infty \} $, which satisfy for all $ \rho \in [0, \nicefrac{1}{4} ) $, $ N \in \N \cup \{ \infty \} $, $ t \in [ 0 , T ] $ that $ \sup_{ s \in [ 0 , T ] } \norm{ X_s^N }_{ L^2 ( \P ; H_\rho \times H_{ \rho-\nicefrac{1}{2} } ) } < \infty $ and $ \P $-a.s.\ that
\begin{equation}
X_t^N =  \e^{\bfA t} \bfP_N \xi + \int_0^t \e^{ \bfA(t-s) } \bfP_N \bfF(X_s^N) \ds  + \int_0^t \e^{ \bfA(t-s) } \bfP_N \bfB(X_s^N) \dWs
\end{equation}
\item[(ii)]
and for all $ \epsilon \in (0,\infty) $ there exists a real number $ C \in [0, \infty) $ such that for all $ N \in \N $ it holds that
\begin{equation}
\abs[\big]{ \E \bigl[ \phi \bigl(X_T^\infty \bigr)\bigr] - \E \bigl[\phi \bigl(X_T^N \bigr) \bigr] } \leq C \cdot N^{ \epsilon - 1 }.
\end{equation}
\end{itemize}
\end{cor}

\begin{proof}[Proof of Corollary~\ref{cor:Anderson_model}]
Throughout this proof let $ B \colon H_0 \to L_2( H_0, H_{ -\nicefrac{1}{2} } ) $ be the mapping which satisfies for all $ v \in H_0 $, $ u \in H_1 $ and $ \lambda_{ (0,1) } $-a.e.\ $ x \in ( 0, 1 ) $ that $ \bigl( B( v ) u \bigr) (x) = ( \alpha + \beta v(x) ) u(x) $. Section~7.2.1 in \cite{Jentzen2015} then implies for all $ \rho \in [ 0, \nicefrac{1}{4} ) $ that $ B \in \mathrm{Lip}^0( H_0, L_2( H_0, H_{ \rho - \nicefrac{1}{2} } ) ) $. Item (i) in Corollary~\ref{cor:wave_equation} and Remark~\ref{rmk:existence_regularity} thus prove (i). Next observe that the Sobolev Embedding Theorem proves for all $ \rho \in ( 0, \nicefrac{1}{4} ) $ that
\begin{equation}
\Biggl[ \sup_{ w \in H_1 \setminus \{ 0 \} } \frac{ \norm{ w }_{ L^{ \nicefrac{1}{ (2\rho) } }( \lambda_{ (0,1) }; \R ) } }{ \norm{w}_{ H_{ { \nicefrac{1}{4} - \rho } } } } \Biggr] \vee \Biggl[ \sup_{ w \in H_\rho \setminus \{ 0 \} } \frac{ \norm{ w }_{ L^{ \nicefrac{2}{ (1- 4\rho) } }( \lambda_{ (0,1) }; \R ) }  }{ \norm{w}_{H_\rho} } \Biggr] < \infty.
\end{equation}
This and H\"{o}lder's inequality ensure for all $ \rho \in ( 0, \nicefrac{1}{4} ) $, $ v \in H_\rho $, $ u \in H_1 $ that
\begin{equation} \label{eq:Anderson_model,1}
\begin{split}
& \sup_{ w \in H_1 \setminus \{ 0 \} } \frac{ \abs{ \langle w, B(v)u \rangle_{ H_0 } } }{ \norm{w}_{ H_{ \nicefrac{1}{4} - \rho } } } \\
& \leq \Biggl[ \sup_{ w \in H_1 \setminus \{ 0 \} } \frac{ \norm{ w }_{ L^{ \nicefrac{1}{ (2\rho) } }( \lambda_{ (0,1) }; \R ) } }{ \norm{w}_{ H_{ { \nicefrac{1}{4} - \rho } } } } \Biggr] \norm{ \alpha + \beta v }_{ L^{ \nicefrac{2}{ (1- 4\rho) } }( \lambda_{ (0,1) }; \R ) } \norm{ u }_{ L^2( \lambda_{ (0,1) }; \R ) } \\
& \leq  \Biggl[ \sup_{ w \in H_1 \setminus \{ 0 \} } \frac{ \norm{ w }_{ L^{ \nicefrac{1}{ (2\rho) } }( \lambda_{ (0,1) }; \R ) } }{ \norm{w}_{ H_{ { \nicefrac{1}{4} - \rho } } } } \Biggr]  \Biggl[ \sup_{ w \in H_\rho \setminus \{ 0 \} } \frac{ \norm{ w }_{ L^{ \nicefrac{2}{ (1- 4\rho) } }( \lambda_{ (0,1) }; \R ) }  }{ \norm{w}_{H_\rho} } \Biggr] \norm{ \alpha + \beta v }_{ H_\rho } \norm{ u }_{ H_0 } < \infty.
\end{split}
\end{equation}
Lemma~\ref{lem:interpolation_spaces} hence shows for all $ \rho \in ( 0, \nicefrac{1}{4} ) $, $ v \in H_\rho $, $ u \in H_1 $ that $ B(v) u \in H_{ \rho-\nicefrac{1}{4} } $. In addition, \eqref{eq:Anderson_model,1}, Lemma~\ref{lem:interpolation_spaces}, and Lemma~\ref{lem:linear_operator} prove for all $ \rho \in ( 0, \nicefrac{1}{4} ) $, $ v \in H_\rho $ that $ B(v) \in L( H_0, H_{ \rho - \nicefrac{1}{4} } ) $. For the remainder of this proof, let $ \epsilon \in ( 0, 1 ] $, $ \gamma \in ( \nicefrac{1}{2} - \nicefrac{\epsilon}{4}, \nicefrac{1}{2} ) $ and $ \rho \in [ \gamma - \nicefrac{1}{4}, \nicefrac{1}{4} ) $. It then follows for all $ v, w \in H_\rho $, $ u \in H_1 $ that
\begin{equation}
\begin{split}
& \norm{ ( B(v) - B(w) ) u }_{ H_{ \gamma - \nicefrac{1}{2} } } = \sup_{ w \in H_1 \setminus \{ 0 \} } \frac{ \abs{ \langle w, ( B(v) - B(w) ) u \rangle_{ H_0 } } }{ \norm{w}_{ H_{ \nicefrac{1}{4} - \rho } } } \\
& \leq  \Biggl[ \sup_{ w \in H_1 \setminus \{ 0 \} } \frac{ \norm{ w }_{ L^{ \nicefrac{1}{ (2\rho) } }( \lambda_{ (0,1) }; \R ) } }{ \norm{w}_{ H_{ { \nicefrac{1}{4} - \rho } } } } \Biggr]  \Biggl[ \sup_{ w \in H_\rho \setminus \{ 0 \} } \frac{ \norm{ w }_{ L^{ \nicefrac{2}{ (1- 4\rho) } }( \lambda_{ (0,1) }; \R ) }  }{ \norm{w}_{H_\rho} } \Biggr] \abs{\beta} \norm{ v - w }_{ H_\rho } \norm{ u }_{ H_0 } < \infty.
\end{split}
\end{equation}
This, Lemma~\ref{lem:interpolation_spaces}, and Lemma~\ref{lem:linear_operator} establish that $ B \vert_{ H_\rho } \in \mathrm{Lip}^0( H_\rho, L( H_0, H_{ \gamma - \nicefrac{1}{2} } ) ) $. Corollary~\ref{cor:wave_equation} thus completes the proof of Corollary~\ref{cor:Anderson_model}.
\end{proof}

\printbibliography

\end{document}